\documentclass[preprint]{earticle}

\usepackage[margin=1.2in]{geometry}

\usepackage{hyperref}
\usepackage{amsmath, amssymb, amsthm, mathtools, bbm}
\usepackage{graphicx,subcaption}
\usepackage{algorithm,algorithmicx,algpseudocode}
\usepackage{thmtools}
\usepackage{siunitx}
\usepackage{xcolor}

\newcommand{\bigO}{\mathcal{O}}
\newcommand{\pR}{\mathbb{R}}

\newcommand{\abs}{\mathrm{abs}}

\newcommand*\Col[1]{\mathcal{C}(#1)}

\newcommand{\rict}{\Delta} 
\newcommand{\ric}{\bar{\Delta}}
\newcommand{\rcc}{\gamma}

\newcommand{\cohc}{\mu}
\newcommand{\bR}{\mathbb{R}}

\DeclareMathOperator{\rank}{rank}
\DeclareMathOperator{\LS}{LS}
\DeclareMathOperator{\E}{\mathbb{E}}
\DeclareMathOperator{\A}{\mathcal{A}}
\DeclareMathOperator{\cA}{\mathcal{A}}
\DeclareMathOperator{\cB}{\mathcal{B}}
\DeclareMathOperator{\G}{\mathcal{G}} 
\DeclareMathOperator{\V}{\mathcal{V}}
\DeclareMathOperator{\cR}{\mathfrak{R}}
\DeclareMathOperator{\fR}{\mathfrak{R}}

\DeclareMathOperator{\eb}{\varepsilon_b}
\DeclareMathOperator{\ep}{\varepsilon_p}
\DeclareMathOperator{\epb}{\varepsilon_b}
\DeclareMathOperator{\epp}{\varepsilon_p}
\DeclareMathOperator{\step}{\alpha}
\DeclareMathOperator{\supp}{supp}

\DeclareMathOperator*{\argmin}{arg\,min}

\DeclareMathOperator{\Proj}{\mathcal{P}}
\DeclareMathOperator{\RPCA}{\mathrm{RPCA}}
\newcommand*\Projarg[3]{\mathrm{Proj}_{(#1, #2)}\left( #3\right)}
\newcommand*\RPCAarg[5]{\mathrm{RPCA}_{#1, #2, #3}(#4, #5)}

\newcommand*\HTSarg[2]{\mathrm{HT}_{#2}(#1)}
\DeclarePairedDelimiterX{\inp}[2]{\langle}{\rangle}{#1, #2}

\newcommand*\inner[2]{\left\langle #1,\,#2 \right\rangle}
\newcommand*\mmat[1]{\mathrm{mat}\left(#1\right)}
\newcommand*\mvec[1]{\mathrm{vec}\left(#1\right)}

\newcommand*{\horzbar}{\rule[.5ex]{2.5ex}{0.5pt}}
\newcommand*\prooftitle[3]{\noindent\textbf{Proof of #1} (#2), stated on page~#3.}
\newcommand*\prooftitleT[3]{\noindent\textbf{Proof of #1} (#2).}

\newcommand{\e}{\mathrm{e}}

\newtheorem{definition}{Definition}[section]
\declaretheorem[name=Lemma,numberwithin=section]{lemma}
\declaretheorem[name=Theorem]{theorem}
\newtheorem{corollary}{Corollary}[section]

\newcommand{\Rev}[1]{#1}

\begin{document}

\begin{frontmatter}

\title{Compressed sensing of low-rank plus sparse matrices}

\author[mymainaddress,mysecondaryaddress]{Jared Tanner}
\ead{tanner@maths.ox.ac.uk}
\ead[url]{https://people.maths.ox.ac.uk/tanner/}

\author[mymainaddress]{Simon Vary\corref{mycorrespondingauthor}}
\cortext[mycorrespondingauthor]{Corresponding author}
\ead{vary.simon@gmail.com}
\ead[url]{https://simonvary.github.io}

\address[mymainaddress]{Mathematical Institute, University of Oxford, Oxford OX2 6GG, UK}
\address[mysecondaryaddress]{The Alan Turing Institute, The British Library, London NW1 2DB, UK}

\fntext[myfootnote]{This publication is based on work partially supported by: the EPSRC I-CASE studentship (voucher 15220165) in partnership with Leonardo, The Alan Turing Institute through EPSRC (EP/N510129/1) and the Turing Seed Funding grant SF019.}

\begin{abstract}
    
Expressing a matrix as the sum of a low-rank matrix plus a sparse
matrix is a flexible model capturing global and local features in data.  This
model is the foundation of robust principle component analysis
\cite{Candes2011robust, Chandrasekaran2009ranksparsity}, and popularized by dynamic-foreground/static-background
separation \cite{Bouwmans2016decomposition}.  Compressed
sensing, matrix completion, and their variants \cite{Eldar2012compressed,Foucart2013a} have
established that data satisfying low complexity models can be
efficiently measured and recovered from a number of measurements
proportional to the model complexity rather than the ambient
dimension.  This manuscript develops similar guarantees showing that
$m\times n$ matrices that can be expressed as the sum of a rank-$r$ matrix and a $s$-sparse matrix can be recovered by computationally tractable methods from 
$\mathcal{O}(r(m+n-r)+s)\log(mn/s)$ linear measurements.  More specifically,
we establish that \Rev{the low-rank plus sparse matrix set is closed provided the incoherence of the low-rank component is upper bounded as $\mu < \sqrt{mn}/(r\sqrt{s})$, and subsequently,} the restricted isometry constants for the
aforementioned matrices remain bounded independent of problem size
provided $p/mn$, $s/p$, and $r(m+n-r)/p$ remain fixed.  Additionally,
we show that semidefinite programming and two hard threshold gradient descent algorithms, NIHT and NAHT, converge to the measured matrix provided
the measurement operator's RIC's are sufficiently small. \Rev{These results also provably solve convex and non-convex formulation of Robust PCA with the asymptotically optimal fraction of corruptions $\alpha = \bigO\left(1 / (\mu r) \right)$, where $s = \alpha^2 mn$, and improve the previously best known guarantees by not requiring that the fraction of corruptions is spread in every column and row by being upper bounded by $\alpha$.} Numerical experiments illustrating these results are shown for synthetic problems, dynamic-foreground/static-background separation, and multispectral imaging.
\end{abstract}


\end{frontmatter}


\section{Introduction}

Data with a known underlying low-dimensional structure can often be estimated from a number of measurements proportional to the degrees of freedom of the underlying model, rather than what its ambient dimension would suggests. Examples of such low-dimensional structures for which the aforementioned is true include: compressed sensing \cite{Donoho2006compressed,Candes2006robust,Candes2005decoding}, matrix completion \cite{Candes2009exact,Candes2010thepower, Recht2010guaranteed}, sparse measures \cite{Candes2014towards,Duval2015exact,Eftekhari2019sparse}, and atomic decompositions \cite{Chi2020harnessing} more generally. Our work extends these results to the matrices which are formed as the sum of a low-rank matrix and a sparse matrix, a model popularized by the  work on robust principle component anaysis (Robust PCA) \cite{Candes2011robust,Chandrasekaran2009ranksparsity}.   Specifically, we consider matrices $X\in\mathbb{R}^{m\times n}$ of the form $X = L + S$, where $L$ is of rank at most $r$, and $S$ has at most $s$ non-zero entries, $\|S\|_0 \leq s$.   The low-rank plus sparse model is a rich model with the low rank component modeling global correlations, while the additive sparse component allows a fixed number of entries to deviate from this global model in an arbitrary way. Among applications of this model are image restoration \cite{Gu2014weighted}, hyperspectral image denoising \cite{Gogna2014split, Chen2017denoising,Wei2016hyperspectral}, face detection \cite{Luan2014extracting,Wright2009robust}, acceleration of dynamic MRI data acquisition \cite{Xu2017dynamic}, analysis of medical imagery \cite{Gao2011robust}, separation of moving objects in at otherwise static scene \cite{Bouwmans2016decomposition}, and target detection \cite{Oreifej2013simultaneous}.

Unlike Robust PCA where $X$ is directly available, we consider the compressed sensing setting where $X$ is measured through a linear operator $\mathcal{A}(\cdot)$, where $\mathcal{A}:\mathbb{R}^{m\times n} \rightarrow \mathbb{R}^p$, $b\in\mathbb{R}^p$ and typically $p \ll mn$. Our contributions extend existing results on {\em restricted isometry constants} (RIC) for Gaussian and other measurement operators for sparse vectors~\cite{Baraniuk2008a} or low-rank matrices~\cite{Recht2010guaranteed} to the sets of low-rank plus sparse matrices.  For the set of matrices which are the sum of a low-rank plus a sparse matrix the results differ subtly due to the space not being closed, in that there are matrices $X$ for which there does not exist a nearest projection to the set of low-rank plus sparse matrices \cite{Tanner2019matrix}.  To overcome this, we introduce the set of low-rank plus sparse matrices with \Rev{the incoherence constraint on the singular vectors of the low-rank component, see Definition~\ref{def:ls}}  
\begin{definition}[Low-rank plus sparse set \Rev{$\LS_{m,n}(r,s,\mu)$}]\label{def:ls}
Denote the set of $m\times n$ real matrices that are the sum of a rank $r$ matrix and a $s$ sparse matrix as
\begin{equation}
	\LS_{m,n}(r,s,\mu) = \left\{
			L + S \in\mathbb{R}^{m\times n}:\,
				\rank(L)\leq r,\,  \left\|S\right\|_0 \leq s,\,
				\begin{array}{c}\max\limits_{i\in[m]}\left\|U^T e_i \right\|_2 \leq \sqrt{\mu r/m} \\ 
				\max\limits_{i\in[n]} \left\|V^T f_i \right\|_2 \leq \sqrt{\mu r / n}
				\end{array}
		\right\},
\end{equation}
where $U\in \mathbb{R}^{m\times r}$, $V \in \mathbb{R}^{n\times r}$ are the first $r$ left and the right singular vectors of $L$ respectively, $e_i \in \mathbb{R}^{m}, f_j\in\mathbb{R}^{n}$ are the canonical basis vectors, and $\mu\in\left[ 1, \sqrt{mn}/r \right]$ controls the incoherence of $L$.
\end{definition}
\Rev{The parameter $\mu$ is referred to as the \emph{incoherence} of the low-rank component \cite{Candes2011robust,Chandrasekaran2009ranksparsity} and it controls correlation between the low-rank component and the sparse component. We show that $\LS_{m,n}(r,s,\mu)$ sets are closed when the incoherence is sufficiently upper bounded as $\mu < \sqrt{mn}/(r\sqrt{s})$, see Lemma~\ref{lemma:LS_mu_closed}. This bound is equivalent to the asymptotically optimal scaling in terms of $r,s$ and $\mu$ in the recovery guarantees independently achieved in Robust PCA using convex relaxation \cite{Hsu2011} or in nonconvex methods \cite{Netrapalli2014provable}, but without the need for the assumption that the fraction of corruptions in each column and row is upper bounded.}

The natural generalization of the RIC definition from sparse vectors and low-rank matrices to the space $\LS_{m,n}(r,s,\mu)$ is given in Definition~\ref{def:ric}.
\begin{definition}[RIC for $\LS_{m,n}(r,s,\mu)$] \label{def:ric}
Let $\A: \pR^{m\times n}\rightarrow \pR^{p}$ be a linear map. For every pair of integers $(r,s)$ and every $\mu\geq1$, define the $(r,s,\mu)$-restricted isometry constant to be the smallest $\rict_{r, s,\mu}(\A) > 0$ such that
\begin{equation}
	\left(1-\rict_{r, s,\mu}(\A)\right)\|X\|^2_F \leq \|\A(X)\|^2_2 \leq \left(1+\rict_{r, s,\mu}(\A)\right)\|X\|^2_F, \label{eq:ric_definition}
\end{equation}
for all matrices $X\in\LS_{m,n}(r,s,\mu)$.
\end{definition}

Random linear maps $\A$ which have a sufficient concentration of measure phenomenon can overcome the dimensionality of $\LS_{m,n}(r,s,\mu)$ to achieve $\rict_{r,s,\mu}$ which is bounded by a fixed value independent of dimension size provided the number of measurements $p$ is proportional to the degrees of freedom of a rank-$r$ plus sparsity-$s$ matrix $r(m+n -r)+s$. A suitable class of random linear maps is captured in the following definition.

\begin{definition}[Nearly isometrically distributed map] \label{def:near_isometry}
Let $\A$ be a random variable that takes values in linear maps $\pR^{m\times n}\rightarrow \pR^p$. We say that $\A$ is nearly isometrically distributed if, for $\forall X\in\pR^{m\times n}$,
\begin{equation}
	\E\left[ \left\|\A(X)\right\|^2 \right] = \|X\|^2_F \label{eq:expectation}
\end{equation}
and for all $\varepsilon \in (0,1)$, we have
\begin{equation}\label{eq:conc_measure}
\Pr\left( \left| \|\A(X)\|_2^2 - \|X\|_F^2 \right| \geq \varepsilon \|X\|^2_F \right) \leq 2 \exp{\left(-\frac{p}{2}\left( \varepsilon^2/2 - \varepsilon^3/3 \right) \right)},
\end{equation}
and there exists some constant $\gamma >0$ such that for all $t>0$, we have
\begin{equation}\label{eq:energy_decay}
	\Pr\left( \| \A \|\geq 1+ \sqrt{\frac{mn}{p}} + t \right) \leq \exp\left( -\gamma p t^2 \right).
\end{equation}
\end{definition}
There are two crucial properties for a random map to be nearly isometric. Firstly, it needs to be isometric in expectation as in~\eqref{eq:expectation}, and exponentially concentrated around the expected value as in~\eqref{eq:conc_measure}. Secondly, the probability of large distortions of length must be exponentially small as in~\eqref{eq:energy_decay}. This ensures that even after taking a union bound over an exponentially large covering number for $\LS_{m,n}(r,s,\mu)$, see Lemma~\ref{lemma:ls_rs_cover}, the probability of distortion remains small~\cite{Baraniuk2008a, Recht2010guaranteed}.

In addition to developing RIC bounds as in Definition~\ref{def:ric} we also show that the RIC of an operator implies uniqueness of the decomposition and that exact recovery is possible with computationally efficient algorithms such as convex relaxations or gradient descent methods.  
The following subsection summarizes our main contributions. The rest of the paper is organized as 
\begin{itemize}
	\item{In Section~\ref{sec:ric}, we prove that the RICs of $\LS_{m,n}(r,s,\mu)$ for Gaussian and fast Johnson-Lindenstrauss transform (FJLT) measurement operators remain bounded independent of problem size provided the number of measurements $p$ is proportional to $\bigO\left(r(m+n-r)+s\right)$.} 
	\item{In Section~\ref{sec:alg}, we prove that when the RICs of $\A(\cdot)$ are suitably bounded then the solution to a linear system $\A(X_0) = b$ has a unique decomposition in $\LS_{m,n}(r,s,\mu)$ that can be recovered using computationally tractable convex optimization solvers and hard thresholding gradient descent algorithms which are natural extensions of algorithms developed for compressed sensing \cite{Blumensath2010normalized} and matrix completion \cite{Tanner2013normalized}. \Rev{These results also provably solve Robust PCA with the asymptotically optimal fraction of corruptions $\alpha = \bigO\left(1 / (\mu r) \right)$, where $s = \alpha^2 mn$, and improve the previously known guarantees by not requiring the fraction of the sparse corruptions in every column and row is bounded by some $\alpha\in(0,1)$.}}
	\item{In Section~\ref{sec:numerics}, we empirically study the average case of recovery on synthetic data by solving convex optimization and by the proposed gradient descent methods and observe a phase transition in the space of parameters for which the methods succeed. We also give an example of two practical applications of the low-rank plus sparse matrix recovery in the form of a subsampled dynamic-foreground/static-background video separation and robust recovery of multispectral imagery.}
\end{itemize}

\subsection{Main contribution}
\Rev{We show that for sufficiently incoherent matrices the $\LS_{m,n}(r,s,\mu)$ set is a closed set, which is essential in developing the recovery guarantees with asymptotically optimal scaling $\mu = \bigO(\sqrt{mn}/(r\sqrt{s}))$.
\begin{lemma}[$\LS_{m,n}(r,s,\mu)$ is a closed set]\label{lemma:LS_mu_closed}
	Let $\mu < \sqrt{mn}/(r\sqrt{s})$ and $X = L + S\in\LS_{m,n}(r,s,\mu)$. Then the following holds
	\begin{enumerate}
		\item{\quad$\left|\inner{L}{S}\right| \leq \mu \frac{r \sqrt{s}}{\sqrt{mn}}\left\| L \right\|_F \left\| S \right\|_F$,}
		\item{\quad$\left\| L \right\|_F \leq \left(1 - \mu^2\frac{r^2 s}{mn} \right)^{-1/2} \left\| X \right\|_F$ and $\left\| S \right\|_F \leq \left(1 - \mu^2\frac{r^2 s}{mn} \right)^{-1/2} \left\| X \right\|_F$,}
		\item{\quad$\LS_{m,n}(r,s,\mu)$ is a closed set.} 
	\end{enumerate}
\end{lemma}}

\Rev{The proof, given in \ref{sec:appendix_lemma} on page \pageref{lemma:LS_mu_closed_proof}, is a consequence of an upper bound on the magnitude of the inner product beteween a sufficiently incoherent low-rank matrix and a sparse matrix and then employing this bound to show that the Frobenius norm of the two components is upper bounded by the Frobenius norm of their sum, which also makes $\LS_{m,n}(r,s,\mu)$ a closed set.}

The foundational analytical tool for our recovery results is the RIC for $\LS_{m,n}(r,s,\mu)$, which as for other RICs \cite{Baraniuk2008a,Recht2010guaranteed}, follows from balancing a covering number for the set $\LS_{m,n}(r,s,\mu)$ and the measurement operator being a near isometry as defined in Definition~\ref{def:near_isometry}.

\begin{theorem}[RIC for $\LS_{m,n}\left(r, s,\mu \right)$]\label{thm:rip_ls}
For a given $m,n,p\in\mathbb{N}$, $\mu < \sqrt{mn}/(r\sqrt{s})$, $\rict \in (0,1)$, and a random linear transform $\mathcal{A}:\mathbb{R}^{m\times n}\rightarrow\mathbb{R}^p$ satisfying the concentration of measure inequalities in Definition \ref{def:near_isometry}, there exist constants $c_0, c_1 > 0$ such that the RIC for $\mathrm{LS}_{m,n}(r,s,\mu)$ is upper bounded with $\rict_{r,s,\mu}(\mathcal{A}) \leq \rict$  provided 
\begin{equation}
	p > c_0\left( r(m+n-r) + s \right)\log\left( \Rev{\left(1 - \mu^2\frac{r^2 s}{mn} \right)^{-1/2}}\frac{mn}{s}\right), \label{eq:ric_sample}
\end{equation}
with probability at least  $1-\exp{(-c_1 p)}$, where $c_0, c_1$ are constants that depend only on~$\rict$.
\end{theorem}

Theorem \ref{thm:rip_ls} establishes that for random ensembles of linear transformations that satisfy the concentration of measure inequalities in Definition~\ref{def:near_isometry}, the RIC for $\mathrm{LS}_{m,n}(r,s,\mu)$ is upper bounded in the asymptotic regime as $m,n$ and $p$ approach infinity at appropriate rates \Rev{and the incoherence $\mu$ is sufficently upper bounded ensuring the set is closed; see Lemma~\ref{lemma:LS_mu_closed}}.  Specifically, the RIC remains bounded independent of the problem dimensions $m$ and $n$ provided $p$ to be taken proportional to the order of degrees of freedom of the rank-$r$ plus sparsity-$s$ matrices times a logarithmic factor as in \eqref{eq:ric_sample}.

Examples of random ensembles of $\mathcal{A}$ which satisfy the conditions of Definition \ref{def:near_isometry} include random Gaussian ensemble which acquires the information about the matrix $X$ through $p$ linear measurements of the form
\begin{equation}
	b_\ell := \A(X)_\ell = \inp{A^{(\ell)}}{X} \quad \text{for}\quad \ell = 1,2, \ldots, p,
\end{equation}
where the $p$ distinct sensing matrices $A^{(\ell)}\in\pR^{m\times n}$ are the sensing operators defining $\A$ and have entries sampled from the Gaussian distribution as $A^{(\ell)}_{i,j}\sim\mathcal{N}(0, 1/p)$. Other notable examples include symmetric Bernoulli ensembles, and Fast Johnson-Lindenstrauss Transform (FJLT)~\cite{Ailon2009the,Krahmer2011new}.

For a linear transform $\A$ which has RIC suitably upper bounded and a given vector of samples $b  = \A(X_0)$, the matrix $X_0$ is the only matrix in the set $\LS_{m,n}(r,s,\mu)$ that satisfies the linear constraint.
\begin{theorem}[Existence of a unique solution for $\A$ with RIC]\label{thm:thmexistunique}
Suppose that $\Delta_{2r,2s,\mu}(\mathcal{A}) < 1$ for some integers $r, s \geq 1$ and \Rev{$\mu < \sqrt{mn}/(r\sqrt{s})$}.  Let $b = \A(X_0)$, then $X_0$ is the only matrix in the set $\LS_{m,n}(r,s,\mu)$ satisfying $\A(X) = b$.
\end{theorem}
\begin{proof}
	Assume, on the contrary, that there exists a matrix $X\in\LS_{m,n}(r,s,\mu)$ such that $\A(X)=b$ and $X\neq X_0$. Then $Z:= X_0 - X$ is a non-zero matrix for which $\A(Z) = 0$ and $Z\in\LS_{m,n}(2r, 2s,\mu)$ \Rev{by the subadditivity property of $\LS_{m,n}(r,s,\mu)$ sets in Lemma~\ref{lemma:ls_mu_additive}}. But then by the RIC we would have $0 = \left\|\A(Z)\right\|^2_2 \geq (1-\Delta_{2r,2s,\mu}) \left\|Z\right\|_F^2 > 0$, which is a contradiction.
\end{proof}

As in compressed sensing and matrix completion, it is in general NP-hard to recover $X_0 = L_0 + S_0\in\mathrm{LS}_{m,n}(r,s,\mu)$ from $\mathcal{A}(X_0)$  for minimal $r,s$ when $p\ll mn$.  This follows from the non-convexity of the feasible set $\LS_{m,n}(r,s,\mu)$.  However, we show that if the linear transformation $\mathcal{A}$ has sufficiently small RIC over the set $\LS_{m,n}(r,s,\mu)$, which requires $\mu < \sqrt{mn}/(r\sqrt{s})$, then the solution can be obtained with computationally tractable methods such as by solving the semidefinite program
\begin{equation}
	\min_{X = L + S \in \mathbb{R}^{m\times n}} \left\|L\right\|_* + \lambda \left\|S\right\|_1, \qquad \text{s.t.}\quad \left\|\mathcal{A}(X) - b\right\|_2\leq \eb, \label{eq:convex}
\end{equation}
where $\|\cdot\|_*$ is the Schatten $1$-norm and $\|\cdot\|_1$ is the sum of the absolute value of the entries\footnote{Our use of $\|\cdot\|_1$ as the sum of the modulus of the entries of a matrix differs from the vector induced $1$-norm of a matrix.} and $\eb$ is the model misfit.  

\begin{theorem}[Guaranteed recovery by the convex relaxation]\label{thm:convex_recovery}
Let $b = \A(X_0)$ and suppose that $r, s \geq 1$ and \Rev{$\mu < \sqrt{mn}/(4r\sqrt{2s})$} are such that the restricted isometry constant $\rict_{4r,2s,2\mu}(\mathcal{A})\leq \Rev{\frac{1}{7} - 2\rcc}$ where $\rcc := \mu \frac{4r \sqrt{2s}}{\sqrt{mn}}$. Let $X^* = L^* + S^*$ be the solution of \eqref{eq:convex} with $\lambda = \sqrt{\Rev{2}r/s}$, then $\|X^* - X_0\|_F \leq 42\eb$.
\end{theorem}

Alternatively, $X_0$ can be obtained from its compressed measurements $\mathcal{A}(X_0)$ by iterative gradient descent methods that are guaranteed to converge to a global minimizer of the non-convex optimization problem
\begin{equation}
	\min_{X = L + S \in \mathbb{R}^{m\times n}} \left\|\mathcal{A}(X) - b\right\|_2, \qquad \text{s.t.}\quad X\in\mathrm{LS}_{m,n}(r,s,\mu).\label{eq:non-convex}
\end{equation}
We introduce two natural extensions of the simple yet effective Normalized Iterative Hard Thresholding (NIHT) for compressed sensing \cite{Blumensath2010normalized} and matrix completion \cite{Tanner2013normalized} algorithms, here called NIHT and Normalized Alternative Hard Thresholding (NAHT) for low-rank plus sparse matrices, Algorithms \ref{algo:niht} and \ref{algo:naht} respectively.  In both cases we establish  that if the measurement operator has suitably small RICs then NIHT and NAHT provably converge to the global minimum of the non-convex problem formulated in~\eqref{eq:non-convex} and recover $X_0\in\LS_{m,n}(r,s,\mu)$ for which $b = \A(X_0)$.

\begin{algorithm}[t]
	\caption{Normalized Iterative Hard Thresholding (NIHT) for $\LS$ recovery\label{algo:niht}}
	\hspace*{\algorithmicindent} \textbf{Input:}  $b = \A(X_0), \A, r, s$, and termination criteria\\
	\hspace*{\algorithmicindent} \textbf{Set:}  $(L^0,S^0) = \Proj\left(\A^*(b); \, r, s, \mu, \varepsilon \right),\, X^0 = L^0 + S^0, \, j = 0$ \\
	\hspace*{\algorithmicindent}\hspace{2.5em}  $\Omega^0 = \supp(S^0)$ and $U^0$ as the top $r$ left singular vectors of $L^0$
	\begin{algorithmic}[1]
		\While{not converged}
		\State Compute the residual $R^{j} = \A^*\left(\A(X^j) - b\right)$
		\State Compute the stepsize: $\step_j = \left\| \Projarg{U^j}{\Omega^j}{R^j} \right\|^2_F / \left\| \A\left( \Projarg{U^j}{\Omega^j}{R^j} \right)\right\|^2_2$
		\State Set $W^j = X^j - \step_j R^{j}$
		\State Compute $(L^{j+1}, S^{j+1}) = \RPCAarg{r}{s}{\mu}{W^j}{\ep}$ and set $X^{j+1} = L^{j+1} + S^{j + 1}$
		\State Let $\Omega^{j+1} = \supp(S^{j+1})$ and $U^{j+1}$ be the top $r$ left singular vectors of $L^{j+1}$
		\State $j = j +1$
            \EndWhile
    \end{algorithmic}
    \hspace*{\algorithmicindent} \textbf{Output:}  $X^j$
\end{algorithm}

\begin{theorem}[Guaranteed recovery by NIHT]\label{thm:niht_convergence}
	Suppose that $r, s \in \mathbb{N}$ and \mbox{\Rev{$\mu<\sqrt{mn}\big/\left(3r \sqrt{3s}\right)$}} are such that the restricted isometry constant $\Delta_3  := \Delta_{3r,3s,\mu}(\cA)< \frac{1}{5}$. Then NIHT applied to $b = \cA(X_0)$ as described in Algorithm~\ref{algo:niht} will linearly converge to $X_0$ as
	\begin{equation}
		\left\| X^{j+1} - X_0 \right\|_F \leq \Rev{\frac{4\rict_3}{1-\rict_3}  \left\| X^{j} - X_0 \right\|_F + \epp,}
	\end{equation} 
	within the precision of $\epp$, where $\epp$ is the accuracy of the Robust PCA oblique projection that performs projection on the set of incoherent low-rank plus sparse matrices $\LS_{m,n}(r,s,\mu)$.
\end{theorem}

\begin{theorem}[Guaranteed recovery by NAHT]\label{thm:naht_convergence}
	Suppose that $r,s\in\mathbb{N}$ and \Rev{$\mu < \sqrt{mn} \big/ \left( 3r \sqrt{3s} \right)$} are such that the restricted isometry constant $\rict_3  := \rict_{3r,3s,\mu}(\cA)< \frac{1}{9} - \rcc_2$ where $\rcc_2:= \mu \frac{2r \sqrt{2s}}{\sqrt{mn}}$. Then NAHT applied to $b = \cA(X_0)$ as described in Algorithm \ref{algo:naht} will linearly converge to $X_0 = L_0 + S_0$ as
	\begin{equation}
		\left\|L^{j+1} - L_0 \right\|_F + \left\|  S^{j+1} - S_0 \right\|_F \leq  \frac{6\rict_3 + \frac{9}{8}\rcc_2 }{1 - 3\rict_3 - \frac{9}{8}\rcc_2 } \left(  \left\| L^j - L_0 \right\|_F + \left\| S^j - S_0 \right\|_F \right).
	\end{equation}
\end{theorem}

Note that the projection used in computing the stepsize is defined as $\Projarg{U^j}{\Omega^j}{R^j} := P_{U^j} R^j + \mathbbm{1}_{\Omega^j} \circ(R^j - P_{U^j} R^j)$, where $P_{U^j} := U^j \left(U^j\right)^*$, $\mathbbm{1}_{\Omega^j}$ is a matrix with ones at indices $\Omega^j$, and $\circ$ denotes the entry-wise Hadamard product. This corresponds to first projecting the left singular vectors of $R^j$ on the subspace spanned by columns of $U^j$, \Rev{which makes the incoherence of the resulting matrix bounded by $\mu$,} and then setting entries at indices $\Omega^j$ to be equal to the entries of $R^j$ at indices $\Omega^j$. One can repeat this process to achieve better more precise projection of $R^j$ in the low-rank plus sparse matrix set defined by $\left(U^j, \Omega^j\right)$.

The hard thresholding projection in Algorithm \ref{algo:niht} is performed by computing Robust PCA which is solved to an accuracy proportional to $\epp$. The Robust PCA projection of a matrix $W\in\pR^{m\times n}$ on the set of $\LS_{m,n}(r,s,\mu)$ with precision $\ep$ returns a matrix $X\in\LS_{m,n}(r,s,\mu)$ such that
\begin{equation}
	X \leftarrow \RPCAarg{r}{s}{\mu}{W}{\ep} \qquad \mathrm{s.t.} \qquad \left\|X- W_{\mathrm{rpca}}\right\|_F \leq \ep, \label{eq:rpca_imprecise}
\end{equation}
where $W_\mathrm{rpca} := \argmin_{Y\in\LS_{m,n}(r,s,\mu)} \|Y- W\|_F$ is the optimal projection of the matrix $W$ on the set $\LS_{m,n}(r,s,\mu)$, \Rev{which can be computed by a number of efficient Robust PCA algorithms, such as the Alternating Projection algorithm (AltProj) \cite{Netrapalli2014provable} or the Accelerated Alternating Projection algorithm (AccAltProj) by \cite{Cai2019accelerated}, which have high robustness in practice and provable global linear convergence when $\alpha = \bigO \left( 1 / \left( \mu r \right) \right)$ and $\alpha = \bigO \left( 1 / \left( \mu r^2 \right) \right)$ respectively, where $s = \alpha^2 mn$.}

\begin{algorithm}[t]
	\caption{Normalized Alternating Hard Thresholding (NAHT) for $\LS$ recovery \label{algo:naht}}
	\hspace*{\algorithmicindent} \textbf{Input:}  $b = \A(X_0), \A, r, s$, and termination criteria\\
	\hspace*{\algorithmicindent} \textbf{Set:}  $(L^0,S^0) = \Proj\left(\A^*(b); \, r, s, \tau, \varepsilon \right),\, X^0 = L^0 + S^0, \, j = 0$ \\
	\hspace*{\algorithmicindent}\hspace{2.5em}  $\Omega^0 = \supp(S^0)$ and $U^0$ as the top $r$ left singular vectors of $L^0$
	\begin{algorithmic}[1]
		\While{not converged}
		\State Compute the residual $R^{j}_L = \A^*\left(\A(X^j) - b\right)$
		\State Compute the stepsize $\step_j^L = \left\| \Projarg{U^j}{\Omega^j}{R^j} \right\|^2_F \Big/ \left\| \A\left( \Projarg{U^j}{\Omega^j}{R^j} \right)\right\|^2_2$
		\State Set $V^j = L^j - \step_j^L R^{j}_L$
		\State Set $L^{j+1} = \mathrm{HT}(V^j;\, r\Rev{, \mu})$ and let $U^{j+1}$ be the left singular vectors of $L^{j+1}$ 
		\State Set $X^{j+\frac{1}{2}} = L^{j+1} + S^{j}$
		\State Compute the residual $R^{j}_S = \A^*\left(\A(X^{j+\frac{1}{2}}) - b\right)$
		\State Compute the stepsize $\step_j^S = \left\| \Projarg{U^{j+1}}{\Omega^j}{R^j} \right\|^2_F \Big/ \left\| \A\left( \Projarg{U^{j+1}}{\Omega^j}{R^j} \right)\right\|^2_2$
		\State Set $W^j = S^j - \step_j^S R^{j}_S$
		\State Set $S^{j+1} = \mathrm{HT}(W^j;\, s)$ and let $\Omega^{j+1} = \supp(S^{j+1})$
		\State Set $X^{j+1} = L^{j+1} + S^{j+1}$
		\State $j = j +1$
            \EndWhile
    \end{algorithmic}
    \hspace*{\algorithmicindent} \textbf{Output:}  $X^j = L^j + S^j$
\end{algorithm}

\subsection{Relation to prior work}

It is well known that the low-rank plus sparse matrix decomposition solved by Robust PCA does not need to have a unique solution without further constraints, such as the singular vectors of the low-rank component being uncorrelated with the canonical basis as quantified by the incoherence condition \cite{Candes2009exact,Recht2010guaranteed} with parameter $\mu$
\begin{gather}
	\label{eq:coherence}
	\max_{i\in \left\{1,\ldots, m\right\}} \left\| U^* e_i \right\|_{2} \leq \sqrt{\frac{\mu r}{m}}, \qquad \max_{i\in \left\{1,\ldots, n\right\}} \left\| V^*f_i \right\|_{2} \leq \sqrt{\frac{\mu r}{n}},
\end{gather}
where $U\in \mathbb{R}^{m\times r}$, $V \in \mathbb{R}^{n\times r}$ are the first $r$ left and the right singular vectors of $L$ respectively, $e_i \in \mathbb{R}^{m}, f_i\in\mathbb{R}^{n}$ are the canonical basis vectors. The incoherence condition for small values of $\mu$ ensures that the left and the right singular vectors are well spread out and not sparse. \Rev{It is therefore sensible to expect that the problem of recovering $X_0\in\LS_{m,n}(r,s,\mu)$ from subsampled measurements should obey the same conditions. The incoherence assumption is directly assumed in the convergence analysis of NAHT and the convex recovery is assumed where we require $\mu < \bigO(\sqrt{mn} / (r\sqrt{s}))$ which is equivalent to the best known recovery bounds in Robust PCA \cite{Hsu2011, Netrapalli2014provable}. The incoherence assumption is also implicitly used in the convergence analysis of NIHT in the Robust PCA projection step in Algorithm~\ref{algo:niht}, Line~$5$, the solution of which is dependent on the incoherence of $L$.}

The results presented here extend the well developed literature on compressed sensing and matrix completion/sensing \cite{Eldar2012compressed,Foucart2013a} to the setting of low-rank plus sparse matrices as defined in Definition~\ref{def:ls}. These foundational RIC bound results allow for further extension to other non-convex algorithms, such as \cite{Waters2011sparcs}, further model based constraints as in \cite{Baraniuk2010model} and other additive models.

The recovery result by convex relaxation in Theorem \ref{thm:convex_recovery} controls the measurement error and/or model mismatch $\eb$. In the proof of NIHT convergence in Theorem \ref{thm:niht_convergence} we consider exact measurements but we control the error of the Robust PCA projection which is assumed to be solved only within prescribed precision $\ep$. The convergence result of NAHT in Theorem \ref{thm:naht_convergence} alternates between projecting of the low-rank and the sparse component. The non-convex algorithms are also stable to error $\eb$, but we omit the stability analysis for clarity in the proofs. 

\Rev{Theorem~\ref{thm:convex_recovery},~\ref{thm:niht_convergence}, and~\ref{thm:naht_convergence} also provably solve Robust PCA when $\A$ is chosen to be the identity and $\mu < \sqrt{mn}/(r\sqrt{s})$ which translates to the optimal scaling in terms of the number of corruptions $\alpha = \bigO\left(1 / (\mu r) \right)$, where $s = \alpha^2 mn$, but without the need of requiring that the fraction of the sparse corruptions in every column and row is bounded by $\alpha$.}

\section{Restricted Isometry Constants for $\LS_{m,n}\left(r, s,\mu \right)$\label{sec:ric}}

This section presents a proof of Theorem~\ref{thm:rip_ls}, that linear maps $\A:\pR^{m\times n}\rightarrow \pR^{p}$ sampled from a class of probability distributions obeying concentration of measure and large deviation inequalities, have bounded RIC for sets of low-rank plus sparse matrices with bounded energy as defined in  Definition \ref{def:ric}.  More precisely, that the RIC of $\A$ remains bounded independent of dimension once $p\ge \mathcal{O}\left(\left(r(m+n-r) + s\right)\log\left( \Rev{(1 - \mu^2\frac{r^2 s}{mn})^{-1/2}}\,\frac{mn}{s}\right)\right)$.
Examples of linear maps which satisfy these bounds include 
random Gaussian matrices and the Fast Johnson-Lindenstrauss transform (FJLT) \cite{Ailon2009the,Krahmer2011new}. 
We extend the method of proof used in the context of sparse vectors by \cite{Baraniuk2008a} and its alteration for the low-rank matrix recovery by \cite{Recht2010guaranteed}.

Our proof of Theorem~\ref{thm:rip_ls} follows from proving the alternative form of \eqref{eq:ric_definition} defined without the squared norms by
\begin{equation}
	\left(1-\ric_{r, s,\mu}(\A)\right)\|X\|_F \leq \|\A(X)\|_2 \leq \left(1+\ric_{r, s,\mu}(\A)\right)\|X\|_F, \label{eq:ric_nonsquare}
\end{equation} 
which we denote as $\ric$. The discrepancy between \eqref{eq:ric_nonsquare} and \eqref{eq:ric_definition} is due to \eqref{eq:ric_nonsquare} being more direct to derive and \eqref{eq:ric_definition} allowing for more concise derivation of Theorem~\ref{thm:convex_recovery},~\ref{thm:niht_convergence}, and \ref{thm:naht_convergence}, \Rev{but the two definitions are related up to a multiplicative constant\footnote{\Rev{The constant $\ric$ satisfiying the inequalities in \eqref{eq:ric_nonsquare} also implies $\left(1-\ric\right)^2\|X\|^2_F \leq \|\A(X)\|^2_2 \leq \left(1+\ric\right)^2\|X\|^2_F, \label{eq:ric_equivalent}$ which in turn ensures that $\rict$ in Definition~\ref{eq:ric_definition} is $\rict = 2\ric - \ric^2\in[0,1]$ when~$\ric \in [0,1]$.}}.}

The proof of Theorem~\ref{thm:rip_ls} begins with the derivation of an RIC for a single subspace $\Sigma_{m,n}(V,W,T,\mu)$ of $\LS_{m,n}(r,s,\mu)$ when the column space of $\Col{L}$ is restricted in the subspace $V$, the row space $\Col{L^T}$ in the subspace $W$, and the sparse component $S$ is in the subspace $T$,
\begin{align} 
	\Sigma_{m,n} \left(V, W, T, \cohc\right) = \left\{ 
		X = L+S\in\bR^{m\times n}:\,
		\begin{array}{c}
			\Col{L}\subseteq V,\, \Col{L^T}\subseteq W, \\
			\supp \left( S \right) \subseteq T, \\
			\forall i \in [m]: \, \left\|  \mathrm{P}_V e_i \right\|_2 \leq \sqrt{\frac{\mu r}{m}}, \\
			\forall i \in [n]: \, \left\| \mathrm{P}_W f_i \right\|_2 \leq \sqrt{\frac{\mu r}{n}}
		\end{array}	
	\right\}, \label{eq:ls_fixed}
\end{align}
\Rev{where $\mathrm{P}_V$ and $\mathrm{P}_W$ denote the orthogonal projection on the subspace $V$ and $W$ respectively, and $e_i \in \bR^m$ and $f_i\in\bR^{n}$ are the canonical basis vectors.}

Following this, we show that the isometry constant of $\A$ is robust to a perturbation of the column and the row subspaces $(V,W)$ of the low-rank component. Finally, we use a covering argument over all possible column and row subspaces $(V,W)$ of the low-rank component and count over all possible sparsity subspaces $T$ of the sparse component to derive an exponentially small probability bound for the event that $\A(\cdot)$ satisfies RIC with constant $\ric$ for sets
\begin{equation}
	\LS_{m,n}(r,s,\mu) = \left\{ \Sigma_{m,n}(V,W,T, \mu)\,: \, V\in\G(m,r), \, W\in\G(n,r), \, T\in\V(mn,s) \right\}, \label{eq:ls_sigma_subset}
\end{equation}
where $\G(m,r)$ is the Grassmannian manifold -- the set of all $r$-dimensional subspaces of $\pR^m$, and $\V(mn,s)$ is the set of all possible supports sets of an $m\times n$ matrix that has $s$ elements. Thus proving RIC for sets of low rank plus sparse matrices given the energy bound on the low-rank component~$L$.

The following result describes the behavior of $\A$ when constrained to a single fixed column and a row space $(V,W)$ and a single sparse matrix space $T$.
\begin{restatable}[RIC for a fixed $\LS$ subspace $\Sigma_{m,n}(V,W,T, \mu)$]{lemma}{lemmaripsingle}
\label{lemma:rip_single}
Let $\A: \pR^{m\times n} \rightarrow \pR^p$ be a nearly isometric random linear map from Definition \ref{def:near_isometry} and $\Sigma_{m,n}(V,W,T,\mu)$ as defined in \eqref{eq:ls_fixed} is fixed for some $(V,W), T$ and \Rev{$\mu < \sqrt{mn}/ (r \sqrt{s})$}.
Then for any $\ric\in(0,1)$
\begin{equation}
	\forall X \in \Sigma_{m,n} \left(V, W, T, \mu \right): \quad (1-\ric)\|X\|_F \leq \| \A(X) \| \leq (1+\ric) \| X \|_F,
\end{equation}
with probability at least
\begin{equation}\label{eq:rip_single_prob}
1 - 2 \left(\frac{24}{\ric} \tau \right)^{\dim V \cdot \dim W}  \left(\frac{24}{\ric} \tau \right)^{\dim T} \exp{\left(-\frac{p}{2}\left( \frac{\ric^2}{8} - \frac{\ric^3}{24} \right) \right)},
\end{equation}
\Rev{where $\tau := (1-\mu^2 \frac{r^2 s}{mn})^{-1/2}$.}
\end{restatable}
The proof follows the same argument as the one for sparse vectors \cite[Lemma 5.1]{Baraniuk2008a} and for low-rank matrices in \cite[Lemma 4.3]{Recht2010guaranteed} \Rev{with the exception of appropriately scaling the Frobenius norm of the two components in relation to the Frobenius norm of their sum}. Our variant of the proof for low-rank plus sparse matrices is presented in~\ref{sec:appendix_lemma} on page~\pageref{lemma:rip_single_proof}.

To establish the impact of a perturbation of the spaces $(U,V)$ on the $\ric$ in Lemma~\ref{lemma:rip_single} we define a metric $\rho(\cdot, \cdot)$ on $\G(D,d)$ as follows
\begin{equation}
	U_1, U_2 \in \G(D,d): \quad  \rho(U_1, U_2) := \| P_{U_1} - P_{U_2} \|. \label{eq:grass_distance}
\end{equation}
The Grassmannian manifold $\G\left(D, d\right)$ combined with distance $\rho(\cdot, \cdot)$ as in \eqref{eq:grass_distance} defines a metric space $\left(\G\left(D, d\right), \rho\left(\cdot, \cdot\right) \right)$, where $P_{U}$ denotes an orthogonal projection associated with the subspace $U$. Let us also denote a set of matrices whose column and row space is a subspace of $V$ and $W$ respectively
\begin{equation}
	(V,W) = \left\{ X:\, \Col{X}\subseteq V,\, \Col{X^T}\subseteq W \right\},
\end{equation}
and $P_{(V,W)}$ is an orthogonal projection that ensures that the column space and row space of $P_{(V,W)}X$ lies within $V$ and $W$. The distance between $\Sigma_1 := \Sigma_{m,n}\left(V_1, W_1, T, \mu\right)$ and $\Sigma_2 :=\Sigma_{m,n}\left(V_2, W_2, T, \mu\right)$ that have a fixed $T$ is given by
\begin{equation}
	\rho\left(\left(V_1, W_1\right), \left(V_2, W_2\right)\right) = \| P_{\left(V_1, W_1\right)} - P_{\left(V_2, W_2\right)} \|. \label{eq:sigma_distance}
\end{equation}

\begin{restatable}[Variation of $\ric$ in RIC in respect to a perturbation of $(V,W)$]{lemma}{lemmavardelta}
\label{lemma:var_delta}
Let $\Sigma_1 := \Sigma_{m,n}(V_1, W_1, T, \mu)$ and $\Sigma_2 :=\Sigma_{m,n}(V_2, W_2, T, \mu)$ be two low-rank plus sparse subspaces with the same fixed subspace~$T$ \Rev{and $\mu < \sqrt{mn}/ (r \sqrt{s})$}. Suppose that for $\ric > 0$, the linear operator $\A$ satisfies
\begin{equation}
	\forall X\in  \Sigma_1: \quad (1-\ric) \| X \|_F \leq \|\A(X)\| \leq (1+\ric)\|X\|_F.
\end{equation}
Then
\begin{equation}
	\forall Y \in  \Sigma_2: \quad (1-\ric') \| Y \|_F \leq \|\A(Y)\| \leq (1+\ric')\|Y\|_F,
\end{equation}
with $\ric' := \ric + \tau \rho\left(\left(V_1, W_1\right), \left(V_2, W_2\right)\right) \left(1+ \ric + \|\A\| \right)$ with $\rho$ as defined in \eqref{eq:grass_distance} \Rev{and \mbox{$\tau := (1-\mu^2 \frac{r^2 s}{mn})^{-1/2}$}}.
\end{restatable}
The proof is similar to the line of argument made in~\cite[Lemma 4.4]{Recht2010guaranteed}, see \ref{sec:appendix_lemma} on  page~\pageref{lemma:var_delta_proof}. The notable exception is the term $\tau := (1-\mu^2 \frac{r^2 s}{mn})^{-1/2}$ appearing in the expression for $\ric'$, which is a result of the set $\LS_{m,n}(r,s)$ not being closed, \Rev{as shown in \cite[Theorem 1.1]{Tanner2019matrix}, without the constraint $\|L\|_F\leq \tau \| X \|_F$ from Lemma~\ref{lemma:LS_mu_closed}.}

To establish the proof of Theorem~\ref{thm:rip_ls} we combine Lemma~\ref{lemma:rip_single} and Lemma~\ref{lemma:var_delta} with an $\varepsilon$-covering of $\LS_{m,n}(r,s,\mu)$, where $\varepsilon$ will be picked to control the maximal allowed perturbation between the subspaces $\rho\left(\left(V_1, W_1\right), \left(V_2, W_2\right)\right)$. The \textit{covering number} $\cR(\varepsilon)$ of $\LS_{m,n}(r,s,\mu)$ at resolution $\varepsilon$ is the smallest number of subspaces $(V_i, W_i, T_i)$ such that, for any triple of $V\in\G(m,r),W\in\G(n,r),T\in\V(mn,s)$ there exists $i$ with $\rho\left(\left(V, W\right), (V_i, W_i)\right)\leq \varepsilon$ and $T = T_i$. The following Lemma gives an upper bound on the cardinality of $\varepsilon$-covering.

\begin{restatable}[Covering number of $\LS_{m,n}(r,s)$]{lemma}{lemmalsrscover}\label{lemma:ls_rs_cover}
The covering number $\cR(\varepsilon)$ of the set $\LS_{m,n}(r,s)$ is bounded above by
\begin{equation}
	\cR (\varepsilon) \leq {mn\choose s} \left( \frac{4\pi}{\varepsilon} \right)^{r \left(m+n-2r\right)}.
\end{equation}
\end{restatable}
The proof comes by counting the possible support sets with cardinality $s$ and by the work of Szarek on $\varepsilon$-covering of the Grassmannian \cite[Theorem 8]{Szarek1998metric}, for completeness the proof is given in \ref{lemma:ls_rs_cover_proof}, page~\pageref{lemma:ls_rs_cover_proof}.

Bounds on the RIC for the set of low-rank plus sparse matrices then follow a proof technique that uses the covering number argument in combination with the concentration of measure inequalities as was done before for sparse vectors~\cite{Baraniuk2008a} and subsequently for low-rank matrices~\cite{Recht2010guaranteed}.

\prooftitle{Theorem~\ref{thm:rip_ls}}{RIC for $\LS_{m,n}(r,s,\mu)$}{\pageref{thm:rip_ls}}
\begin{proof}
\label{thm:rip_ls_proof}
By linearity of $\cA$ and conicity of $\LS_{m,n}(r,s,\cohc)$ assume without loss of generality $\|X\|_F = 1$ and consequently also $\|L\|_F\leq \tau$ and $\|S\|_F\leq \tau$ with \Rev{$\tau := (1-\mu^2 \frac{r^2 s}{mn})^{-1/2}$ by Lemma~\ref{lemma:LS_mu_closed} and by $\cohc < \sqrt{mn}/ (r \sqrt{s})$}. Let $(V_i, W_i, T_i)$ be an $\varepsilon$-covering of $\LS_{m,n}(r,s,\mu)$ whose covering number is bounded by Lemma~\ref{lemma:ls_rs_cover} since $\LS_{m,n}(r,s,\mu)\subset \LS_{m,n}(r,s)$. For every triple $(V_i, W_i, T_i)$ define a subset of matrices
\begin{equation}
	\cB_i = \left\{X\in\Sigma_{m,n}\left(V,W,T_i,\cohc \right) \,:\, \rho\left( \left(V,W\right), \left(V_i,W_i\right)\right) \leq \varepsilon \right\}.
\end{equation}
By $(V_i, W_i, T_i)$ being an $\varepsilon$-covering we have $\LS_{m,n}(r,s, \cohc)\subseteq \bigcup_i \cB_i$. Therefore, if for all $\cB_i$
\begin{equation}
	(\forall X \in \cB_i): \quad (1-\ric)\|X\|_F \leq \|\cA(X)\| \leq (1+\ric)\|X\|_F
\end{equation}
holds, then necessarily $\ric_{r,s, \cohc} \leq \ric$, proving that
\begin{align}
	\Pr(\ric_{r,s,\cohc} &\leq \ric) = \Pr\Big(\forall X \in \LS_{m,n}(r,s, \cohc):\, (1-\ric) \|X\|_F \leq \|\cA(X)\|\leq (1+\ric) \|X\|_F \Big) \\
	&\geq \Pr \Big( (\forall i), (\forall X\in \cB_i):\,  (1-\ric) \|X\|_F \leq \|\cA(X)\|\leq (1+\ric) \|X\|_F\Big),\label{eq:rip_prob1}
\end{align}
where the inequality comes from the fact that $\LS_{m,n}(r,s,\cohc)$ is a subset of the $\varepsilon$-covering $\bigcup_i \cB_i$ and therefore the statement holds with less or equal probability. It remains to derive a lower bound on the probability in the equation \eqref{eq:rip_prob1} which in turn proves the theorem.

In the case that $\| \cA \|\leq \frac{\ric}{2\tau\varepsilon} -1 \Rev{- \frac{\ric}{2}}$, which we show later in \eqref{eq:rip_bound5} occurs with \Rev{probability exponentially converging to $1$}, rearranging the terms yields
\begin{equation}
	\tau \varepsilon (1+\ric/2 + \|\cA\|) \leq \ric/2. \label{eq:rip_bound1}
\end{equation}

If the RIC holds for a fixed $(V_i, W_i, T_i)$ with $\ric/2$, then by Lemma \ref{lemma:var_delta} in combination with \eqref{eq:rip_bound1} yields  
\begin{equation}
	(\forall X\in \cB_i): \, (1-\ric)\|X\|_F \leq \|\cA \Rev{(X)}\| \leq (1+\ric)\|X\|_F.
\end{equation}
Therefore, using the probability union bound on \eqref{eq:rip_prob1} over all $i$'s and the probability of $\|\cA\|$ satisfying the bound $\varepsilon \leq \ric/\left(2\tau\left( 1+ \|\cA\| \right)\right)\leq \ric/\left(2\tau\left( 1+ \|\cA\| \right)\right)$ \Rev{by \eqref{eq:rip_bound1} and $\ric\geq 0$}. 
\begin{align}
	& \Pr \Big( (\forall i), (\forall X\in \cB_i):\,  (1-\ric) \|X\|_F \leq \|\cA(X)\|\leq (1+\ric) \|X\|_F\Big) \label{eq:rip_bound2}\\
	 \geq 1&-\sum_i \Pr \left( \exists Y\in \Sigma_{m,n}(V_i, W_i, T_i, \cohc)\, : \, 
		\begin{array}{c}
			\left\| \cA(Y) \right\| < (1-\ric/2) \\ 
			\text{or} \quad \left\| \cA(Y) \right\| > (1+\ric/2)
		\end{array}\right) \label{eq:rip_bound3} \\
	 &-\Pr\Big( \left\| \cA \right\|\geq \frac{\ric}{2\tau\varepsilon} -1 \Rev{-\frac{\ric}{2}}\Big). \label{eq:rip_bound4}
\end{align}

The probability in \eqref{eq:rip_bound3} is bounded from above as
\begin{align}
	&\sum_i \Pr \left( \exists Y\in \Sigma_{m,n}(V_i, W_i, T_i, \cohc)\, : \, 
	\begin{array}{c}
		\left\| \cA(Y) \right\| < (1-\ric/2) \\ 
		\text{or} \quad \left\| \cA(Y) \right\| > (1+\ric/2)
	\end{array}\right) \\
	&\leq 2\fR(\varepsilon) \left( \frac{48}{\ric}\tau \right)^{r^2} \left( \frac{48}{\ric} \tau \right)^{s} \exp\left(-\frac{p}{2}\left(\frac{\ric^2}{32} - \frac{\ric^3}{192}\right) \right) \\
	&\leq 2 {mn \choose s} \left( \frac{4\pi}{\varepsilon} \right)^{r(m+n-2r)}\left( \frac{48}{\ric}\tau \right)^{r^2 + s} \exp\left(-\frac{p}{2}\left(\frac{\ric^2}{32} - \frac{\ric^3}{192}\right) \right),\label{eq:rip_termA}
\end{align}
where in the first inequality we used Lemma \ref{lemma:rip_single} and in the second inequality the bound on the $\varepsilon$-covering of the subspaces by Lemma~\ref{lemma:ls_rs_cover}.

In order to complete the lower bound in \eqref{eq:rip_bound2} it remains to upper bound \eqref{eq:rip_bound4} which we obtain by selecting the covering resolution $\varepsilon$ sufficiently small so that the $\Pr\left( \| \cA \|\geq \frac{\ric}{2\tau\varepsilon} -1 \Rev{-\frac{\ric}{2}}\right)$ is exponentially small with the exponent proportional to the bound in \eqref{eq:rip_termA}. From condition \eqref{eq:energy_decay} of Definition \ref{def:near_isometry} we have that the random linear map satisfies 
\begin{equation}
	(\exists \gamma >0): \quad \Pr\left( \| \cA \|\geq 1+ \sqrt{\frac{mn}{p}} + t \right) \leq \exp\left( -\gamma p t^2 \right), \label{eq:rip_bound5}
\end{equation}
in particular
\begin{equation}
	\Pr\left( \| \cA \|\geq \frac{\ric}{2\tau\varepsilon} -1 \Rev{-\frac{\ric}{2}}\right) \leq \exp\left( -\gamma p \left(\frac{\ric}{2\tau\varepsilon} \Rev{-\frac{\ric}{2}} -\sqrt{\frac{mn}{p}} -2\right)^2 \right).
\end{equation}
Selecting the covering resolution 
$\varepsilon$
\begin{equation}
	\varepsilon < \frac{\ric}{4\tau \left(\sqrt{mn/p} + 1 \Rev{+ \ric/4}\right)}, \label{eq:rip_epsilon}
\end{equation}
obtains the following exponentially small upper bound 
\begin{equation}
\Pr\left( \| \cA \|\geq \frac{\ric}{2\tau\varepsilon} -1 - \Rev{\frac{\ric}{2}} \right) \leq \exp\left(-\gamma mn\right) \label{eq:rip_termB}.
\end{equation}

Returning to the inequality \eqref{eq:rip_bound2}, combined with the bound on the first term in \eqref{eq:rip_termA}, and setting $\varepsilon = \ric \Big/ \left( 4\tau \left(\sqrt{mn/p} + 1 + \ric/4\right)\right)$ in the second term of \eqref{eq:rip_termA}, such that \eqref{eq:rip_epsilon} is satisfied, we have that
\begin{align}
	&\qquad\quad 2 \left(\frac{emn}{s}\right)^s \left( \frac{16\pi (\sqrt{mn/p} + 1 + \Rev{\ric/4})}{\ric} \Rev{\tau} \right)^{r(m+n-2r)}\left( \frac{48}{\ric}\tau \right)^{r^2+ s} \nonumber \\
	&\hspace{2in} \cdot\,\exp\left(- \frac{p}{2}\left(\frac{\ric^2}{32} - \frac{\ric^3}{192}\right) \right) \\
	=&  \exp\Bigg( -p a(\ric) + r\left(m+n-2r\right) \log\left(\sqrt{\frac{mn}{p}} +1 \Rev{ +\frac{\ric}{4} } \right)  + r\left(m+n-2r\right) \log\left( \frac{16\pi}{\ric}\tau\right) \nonumber \\
	&\qquad\qquad  + (r^2 +s) \log\left(\frac{48}{\ric}\tau\right) + s\log\left( \frac{emn}{s} \right) \Rev{+\log(2)} \Bigg),\label{eq:rip_final_bound}
\end{align}
where we used the inequality ${mn \choose s} \leq \left(\frac{emn}{s}\right)^s$ and we define $a(\ric) := \ric^2/64 -\ric^3/384$. The $2^{nd}$, $3^{rd}$ and $4^{th}$ terms in \eqref{eq:rip_final_bound} can be bounded as
\begin{equation}
	(\exists c_2>0): \quad 2^{nd}+3^{rd}+4^{th} \leq \left(c_2/a(\ric)\right) r(m+n-r)\log\left(\frac{mn}{p}\Rev{\tau}\right),
\end{equation}
and the $5^{th}$ and $6^{th}$ term of \eqref{eq:rip_final_bound} as
\begin{equation}
	(\exists c_3>0): \quad 5^{th}+6^{th} \leq \left(c_3/a(\ric)\right) s\log\left(\frac{mn}{s}\Rev{\tau}\right),
\end{equation}
where $c_2$ and $c_3$ are dependent only on $\ric$. Therefore there exists positive constants $c_0, c_1$ that depend\footnote{\Rev{We have that $c_1=\Rev{(1+\gamma) a(\ric)^{-1}}$ and $c_0=16 \pi / (\ric \, a(\ric))$.}} only on $\ric$ such that if $p\geq c_0 \left(r(m+n-r)+s\right)\log\left(\frac{mn}{s}\tau\right)$, then RICs are upper bounded by the constant $\ric$ with probability at least $\e^{-c_1 p}$. \Rev{By the constant $\ric$ in \eqref{eq:ric_nonsquare} being related to the RIC with squared norms, the result also implies an upper bound on RICs with the squared norms in Definition~\ref{def:ric}.}
\end{proof}

\section{Provable recovery guarantees using computationally efficient algorithms \label{sec:alg}}
This section contains the proofs of our main algorithmic contributions that a low-rank plus sparse matrix $X_0\in\LS_{m,n}(r,s,\mu)$ can be efficiently recovered from subsampled measurements taken by a linear mapping $\A(\cdot)$ which satisfies given bounds on its RIC. \Rev{These algorithms also provably solve  Robust PCA when $\A$ is chosen to be the identity and $s = \bigO\left(mn / (\mu^2 r^2) \right)$ which is the optimal scaling in terms of the number of corruptions, rank, and the incoherence.} Subsection~\ref{subsec:convex} presents the proof of Theorem~\ref{thm:convex_recovery} which shows that the convex relaxation \eqref{eq:convex} of \eqref{eq:non-convex} robustly recovers  $X_0$.  Subsection \ref{subsec:nonconvex} states the proofs of Theorem~\ref{thm:niht_convergence} and Theorem~\ref{thm:naht_convergence} for the simple yet efficient hard thresholding algorithms NIHT and NAHT, described in Alg.~\ref{algo:niht} and Alg.~\ref{algo:naht} respectively.

\subsection{Recovery of $X_0\in\LS_{m,n}(r,s,\mu)$ using the convex relaxation \eqref{eq:convex}\label{subsec:convex}.}
Let $X^* = L^* + S^*$ be the solution of the convex optimization problem formulated in \eqref{eq:convex}.  Here it is shown that if the RICs of the measurement operator $\A(\cdot)$ are sufficient small, then $X^* = X_0$ when the linear constraint in the convex optimization problem \eqref{eq:convex} is satisfied exactly, or alternatively that $\left\|X^*-X_0\right\|_F$ is proportional to $\left\|\A(X^*) - b\right\|_2$.

\prooftitle{Theorem~\ref{thm:convex_recovery}}{Guaranteed recovery by the convex relaxation \eqref{eq:convex}}{\pageref{thm:convex_recovery}}
\begin{proof}
	Let $R= X^* - X_0 = (L^* - L_0) + (S^* - S_0) = R^L + R^S$ be the residual split into the low-rank component $R^L = L^* - L_0$ and the sparse component $R^S=S^* - S_0$. We treat $R^L$ and $R^S$ separately, combining the method of proof used in the context of compressed sensing by \cite{Candes2005stable} and its extension for the low-rank matrix recovery by \cite{Recht2010guaranteed} \Rev{with the important exception of needing to decompose $R^L$ into a sum of incoherent low-rank matrices using Lemma~\ref{lemma:decomposing_RL} and carefully treat its correlation~with~$R^S$}.

	By Lemma~\ref{lemma:R0L} on page \pageref{lemma:R0L} there exist matrices $R^L_0, R_c^L \in\bR^{m\times n}$ such that $R^L =R^L_0 + R^L_c$ and
	\begin{gather}
		R_0^L \in \LS_{m,n}(2r, 0, \mu)\\
		L_0 (R_c^L)^T = 0_{m\times m} \quad \text{and}\quad L_0^T R_c^L = 0_{n\times n}.
	\end{gather}
	Similarly, by the argument made in the proof of \cite[Theorem 1]{Candes2005stable}, which we state in Lemma~\ref{lemma:R0S}, there exist matrices $R_0^S,R_c^S\in\bR^{m\times n}$ such that $R^S = R_0^S + R_c^S$ and 
	\begin{gather}
		\left\|R_0^S\right\|_0 \leq s \\
		\supp{\left(S_0\right)} \cap \supp{\left(R_c^S\right)} = \emptyset.
	\end{gather}
	
	By $(L^*, S^*)$ being a minimum and $X_0$ being feasible of the convex optimization problem \eqref{eq:convex}
	\begin{align}
		\left\| L_0\right\|_* + \lambda \left\| S_0 \right\|_1 &\geq \left\|L^*\right\|_* + \lambda \left\|S^*\right\|_1 \\
		&=  \left\| L_0 + R_0^L + R_c^L\right\|_* + \lambda \left\| S_0+R_0^S+R_c^S \right\|_1\\
		&\geq \left\|L_0 + R_c^L \right\|_* - \left\| R_0^L\right\|_* + \lambda \left\|S_0+R_c^S\right\|_1 - \lambda \left\|R_0^S\right\|_1\\
		& = \left\|L_0\right\|_* + \left\| R_c^L \right\|_* - \left\|R_0^L\right\|_* + \lambda \left\|S_0\right\|_1 +\lambda \left\|R_c^S\right\|_1 - \lambda \left\| R_0^S\right\|_1, \label{eq:convex_optimality0}
	\end{align}
	where the second line comes from $L^* - L_0 = R_0^L + R_c^L$ and $S^* - S_0 = R_0^S + R_c^S$, the inequality in the third line comes from the reverse triangle inequality, and the fourth line comes from the construction of $R_c^L$ and $R_c^S$ combined with \cite[Lemma 2.3]{Recht2010guaranteed}, restated as Corollary \ref{cor:convex_nucsum}, and by $\supp(R_c^S)\cap\supp(R_0^S) = \emptyset$. Subtracting $\left\|L_0\right\|_*$ and $\left\|S_0\right\|_1$ from both sides of \eqref{eq:convex_optimality0} and rearranging terms yields
	\begin{equation}
		\left\|R_c^L \right\|_* + \lambda \left\|R_c^S\right\|_1 \leq \left\|R_0^L\right\|_* + \lambda \left\|R_0^S\right\|_1. \label{eq:convex_optimality}
	\end{equation}
	
	We proceed by decomposing the remainder terms $R^L_c$ and $R^S_c$ as sums of matrices with decreasing energy as was done by \cite{Recht2010guaranteed} for low-rank matrices and by \cite{Candes2005stable} for sparse vectors.   
	\Rev{By Lemma~\ref{lemma:decomposing_RL} there exists a decomposition $R^L_c = R^L_1 + R^L_2 + \ldots$ such that}
	\begin{gather}
		\Rev{R_i^L \in \LS_{m,n}(M_r, 0, \mu)} \label{eq:convex_RcL_rank}\\
		R_i^L \left(R_j^L\right)^T = 0_{m\times m} \quad \text{and}\quad \left(R_i^L\right)^T R_j^L = 0_{n\times n},\quad \forall i\neq j \label{eq:convex_RcL_ortho}\\
	   \left\|R_{i+1}^L\right\|_F^2 \leq \frac{1}{M_r}\left\|R_i^L\right\|_*^2 \label{eq:convex_RcL_decay}.
   	\end{gather}
	
	To decompose the residual of the sparse component order the indices of $R^S_c$ as $v_1, v_2, \ldots, v_{mn}\in [m]\times[n]$ in decreasing order of magnitude of the entries of $R^S_c$ and split the indices of the entries into sets of size $M_s$ as
	\begin{equation}
		 T_i := \left\{ v_\ell\,:\, (i-1)M_s \leq \ell \leq iM_s\right\},
	\end{equation}
	Constructing $R_i^S := \left(R^S_c\right)_{T_i}$ decomposes $R^S_c$ into a sum $R^S_c = R^S_1 + R^S_2 + \ldots$ such that
	\begin{gather}
		 \left\|R_i^S \right\|_0 \leq M_s, \qquad \forall i\geq 1 \label{eq:convex_RcS_sparse}\\
		 \emptyset =T_i  \cap T_j, \qquad \forall i\neq j \label{eq:convex_RcS_ortho}\\
		\left|R_c^S\right|_{(v)} \leq \frac{1}{\Rev{M_s}} \sum_{j\in T_i} \left|R_i^S\right|_{(j)}, \qquad\forall v\in T_{i+1}\label{eq:convex_RcS_decay}
	\end{gather}
	where the inequality \eqref{eq:convex_RcS_decay} implies that $\left\|R_{i+1}^S\right\|_F^2 \leq \frac{1}{M_s}\left\|R_i^S\right\|_1^2$. \Rev{We denote $R_i = R_i^L + R_i^S$ which are in $\LS_{mn}(M_r,M_s,\mu)$ by construction.} Combining the two decompositions of $R_c^L$ and $R_c^S$ gives the following bound
	\begin{align}
		\Rev{\sum_{j\geq 2} \left\| R_j \right\|_F } &\leq \sum_{j\geq 2} \left\|R_j^L\right\|_F +  \sum_{j\geq 2} \left\|R_j^S\right\|_F \\ 
		&\leq \sqrt{\frac{1}{M_r}} \sum_{j\geq 1} \left\|R_j^L\right\|_* + \sqrt{\frac{1}{M_s}} \sum_{j\geq 1} \left\|R_j^S\right\|_1 \\
		&= \sqrt{\frac{1}{M_r}} \left\|R_c^L\right\|_* + \sqrt{\frac{1}{M_s}} \left\|R_c^S\right\|_1 \\
		&\leq \sqrt{\frac{1}{M_r}} \left( \left\|R_0^L\right\|_* + \sqrt{\frac{M_r}{M_s}}\left\|R_0^S\right\|_1 \right) \\
		&\leq \sqrt{\frac{2r}{M_r}} \left\|R_0^L\right\|_F + \sqrt{\frac{s}{M_s}}\left\|R_0^S\right\|_F,\label{eq:convex_optimality2b}
	\end{align}
	where the inequality in the first line comes from the triangle inequality, the second inequality comes as a consequence of \eqref{eq:convex_RcL_decay} and \eqref{eq:convex_RcS_decay}, the third line comes from \eqref{eq:convex_RcL_ortho} combined with \cite[Lemma 2.3]{Recht2010guaranteed}, restated as Corollary \ref{cor:convex_nucsum}, and from \eqref{eq:convex_RcS_ortho}, the fourth inequality comes from \eqref{eq:convex_optimality} with $\lambda = \sqrt{M_r/M_s}$, and the last fifth line is a property of $\ell_1$ and Schatten-$1$ norms. Choosing $M_r = \Rev{2}r$ and $M_s = s$ in \eqref{eq:convex_optimality2b} gives
	\begin{equation}
		\Rev{\sum_{j\geq 2} \left\| R_j \right\|_F } \leq \left\|R_0^L\right\|_F + \left\|R_0^S\right\|_F \label{eq:convex_optimality2},
	\end{equation}
	and also that $\lambda = \sqrt{\Rev{2}r/s}$ as stated in the theorem.
	
	By feasibility of $X^*$ and linearity of $\cA$ we have
	\begin{equation}
		\epb \geq \left\|\cA\left(X^*\right) - b\right\|_2  = \left\| \cA\left(X^* - X_0\right)\right\|_2 =\left\|\cA\left(R\right)\right\|_2.\label{eq:R_feasibility}
	\end{equation}
	Let \Rev{$\rict := \rict_{4r, 2s,\mu}$} be the RIC with squared norms for \Rev{$\LS_{m,n}(4r, 2s, \mu)$} and \Rev{$\rcc := \mu \frac{4r \sqrt{2s}}{\sqrt{mn}} < 1$}. Then
	\begin{align}
		(1- &\rict) \|R_0^L \|^2_F \leq \left\|\cA\left(R_0^L\right)\right\|_2^2 = \left|\inner{\cA(R_0^L )}{\cA(R_0^L - R + R)}\right| \label{eq:convex_L_bound0start}\\
			& = \left| \inner{\cA(R_0^L )}{\cA(R_0^L - R)} + \inner{\cA(R_0^L)}{\cA(R)} \right| \\
			&\leq \left| \inner{\cA\left( R_0^L \right)}{\cA(-R_0^S - R_1 - \sum_{j\geq 2} R_j )}\right| + \left|\inner{\cA(R_0^L)}{\cA(R)}\right|\\
			& \leq \left( \frac{2\rcc}{1-\rcc^2} + \rict \right) \left\|R_0^L \right\|_F\left( \left\|R_0^S \right\|_F + \left\|R_1\right\|_F + \sum_{j\geq 2}\left\|R_j\right\|_F\right)  +  \left\|\cA\left(R_0^L \right)\right\|_2 \left\|\cA\left(R\right)\right\|_2,\label{eq:convex_L_bound0} \\
			&\leq \left( \frac{2\rcc}{1-\rcc^2} + \rict \right)  \left\|R_0^L\right\|_F \left( \left\| R_0^L \right\|_F + 2 \left\| R_0^S \right\|_F +  \left\|R_1\right\|_F \right) + (1+\rict) \|R_0^L \|_F \epb \label{eq:convex_L_bound1}
	\end{align}
	where the inequality in the first line comes from \Rev{$R_0^L \in \LS_{m,n}(4r, 2s,\mu)$} satisfying the RICs, the second line is a consequence of feasibility in \eqref{eq:R_feasibility}, the third line comes from Lemma \ref{lemma:Ainp1} and by sums of individual pairs in the inner product being in \Rev{$\LS_{m,n}(4r, 2s,\mu)$ by Lemma~\ref{lemma:ls_mu_additive}, and the last inequality follows from the optimality condition in \eqref{eq:convex_optimality2}}. After dividing both sides of \eqref{eq:convex_L_bound1} by $(1-\rict)\left\|R_0^L \right\|_F$ gives
	\begin{equation}
		\left\|R_0^L\right\|_F \leq \frac{ 1 }{1-\rict} \left(\frac{2\rcc}{1-\rcc^2} + \rict \right) \left( \left\| R_0^L \right\|_F + 2 \left\| R_0^S \right\|_F +  \left\|R_1\right\|_F \right) + \epb \frac{1+\rict}{1-\rict} \label{eq:convex_bound1}.
	\end{equation}

	\emph{Mutatis mutandis}, the same argument applies to $\left\|R_0^S\right\|_F$
	\begin{equation}
		\left\|R_0^S\right\|_F \leq \frac{ 1 }{1-\rict} \left(\frac{2\rcc}{1-\rcc^2} + \rict \right) \left( 2 \left\| R_0^L \right\|_F + \left\| R_0^S  \right\|_F +  \left\|R_1\right\|_F \right) + \epb \frac{1+\rict}{1-\rict}\label{eq:convex_bound2},
	\end{equation}
	and similarly to $\left\|R_1\right\|_F$ as
	\begin{equation}
		\left\|R_1\right\|_F \leq \frac{ 1 }{1-\rict} \left(\frac{2\rcc}{1-\rcc^2} + \rict \right) \left( 2 \left\| R_0^S \right\|_F + 2 \left\| R_0^L \right\|_F \right) + \epb \frac{1+\rict}{1-\rict}\label{eq:convex_bound3}.
	\end{equation}
	Adding \eqref{eq:convex_bound1}, \eqref{eq:convex_bound3}, and \eqref{eq:convex_bound3} together gives
	\begin{equation}
		\left\|R_0^L\right\|_F + \left\|R_0^S\right\|_F + \left\|R_1\right\|_F \leq \frac{ 1 }{1-\rict} \left(\frac{2\rcc}{1-\rcc^2} + \rict \right)  \left( 5 \left\| R_0^L \right\|_F + 5 \left\| R_0^S \right\|_F + 2 \left\| R_1 \right\|_F \right) + 3\epb \frac{1+\rict}{1-\rict}\label{eq:convex_bound4},
	\end{equation}

	For $\rict < \frac{1}{7} - 2\rcc$ the prefactor $\frac{1}{1-\rict} \left( \rict + \frac{2\rcc}{1-\rcc^2} \right) < \frac{1}{6}$ and therefore also $\frac{\rict}{1-\rict} < \frac{1}{6}$, resulting into \eqref{eq:convex_bound4} being upper bounded as
	\begin{align}
		\left\|R_0^L\right\|_F + \left\|R_0^S\right\|_F + \left\|R_1\right\|_F  \leq \frac{5}{6} \left( \left\|R_0^L\right\|_F + \left\|R_0^S\right\|_F + \left\|R_1\right\|_F \right) + \frac{7}{2}\epb, \label{eq:convex_eps_ineq}
	\end{align}
	which after rearranging yields
	\begin{align}
		\left\|R_0^L\right\|_F + \left\|R_0^S\right\|_F + \left\|R_1\right\|_F  \leq 21\epb. \label{eq:finalconvexbound1}
	\end{align}

	\Rev{Applying the triangle inequality on $R = R^L_0 + R^S_0 + R_1 + \left(\sum_{j\geq 2} R_j \right)$ and using the bounds in \eqref{eq:convex_optimality2} and \eqref{eq:finalconvexbound1} concludes the proof
	\begin{align*}
		\left\| R \right\|_F &\leq \left\|R_0^L\right\|_F + \left\|R_0^S\right\|_F + \left\|R_1\right\|_F + \sum_{j\geq 2}  \left\| R_j \right\|   \\
		& \leq  2 \left\|R_0^L\right\|_F + 2 \left\|R_0^S\right\|_F + \left\|R_1\right\|_F \leq 42\epb.
	\end{align*}}
\end{proof}

\subsection{Recovery of $X_0\in\LS_{m,n}(r,s,\mu)$ by Alg.~1 and Alg.~2.}\label{subsec:nonconvex}
This section presents the proofs of Theorem \ref{thm:niht_convergence} and \ref{thm:naht_convergence}, that NIHT and NAHT respectively recover $X_0\in\LS_{m,n}(r,s,\mu)$ from $\A(X_0)$ and knowledge of $(r,s,\mu)$ provided the RICs of $\A(\cdot)$ are sufficiently bounded.

The proof of NIHT follows the same line of thought as the one for low-rank matrix completion \cite{Tanner2013normalized}, with the only difference of the hard thresholding projection, in the form of $\RPCA$, being an imprecise projection with accuracy $\ep$ as stated in \eqref{eq:rpca_imprecise}. The proof consists of deriving an inequality where $\|X^{j+1} - X_0\|_F$ is bounded by a factor multiplying $\|X^j - X_0\|_F$, and then showing that this multiplicative factor is strictly less then one if $\A$ satisfies RIC with $\rict_3 := \rict_{r,s,\mu}(\A) < 1/5$.


\prooftitleT{Theorem~\ref{thm:niht_convergence}}{Guaranteed recovery by NIHT, Alg.~\ref{algo:niht}}{\pageref{thm:niht_convergence}}
\begin{proof}
	Let $b = \cA(X_0)$ be the vector of measurements of the matrix $X_0 \in \LS_{m,n}(r,s,\mu)$ and $W^j = X^j - \step_j \cA^*\left(\cA(X^j)  - b\right)$ to be the update of $X^j$ before the oblique Robust PCA projection step $X^{j+1} = \RPCAarg{r}{s}{\mu}{W^j}{\epp}$. By $X^{j+1}$ being within an $\epp$ distance in the Frobenius norm of the optimal $\RPCA$ projection $X^{j+1}_\mathrm{rpca}:=\RPCAarg{r}{s}{\mu}{W^j}{0}$ defined in \eqref{eq:rpca_imprecise}
	\begin{align}
		\left\| W^j - X^{j+1} \right\|_F^2 &= \left\| W^j - X^{j+1}_\mathrm{rpca} + X^{j+1}_\mathrm{rpca} - X^{j+1}\right\|_F^2 \label{eq:niht_eq1a}\\
			&\leq \left(\left\| W^j - X^{j+1}_\mathrm{rpca}\right\|_F + \left\| X^{j+1} - X^{j+1}_\mathrm{rpca}\right\|_F\right)^2 \\
			&\leq \left(\left\|W^j - X_0\right\|_F  + \epp\right)^2, \label{eq:niht_eq1b}
	\end{align}
	where in the second line we used the triangle inequality, and the third line comes from $X^{j+1}_\mathrm{rpca}$ being the optimal projection thus being the closest matrix in \Rev{$\LS_{m,n}(r,s,\mu)$} to $W^j$ in the Frobenius norm and by $X^{j+1}$ being within $\epp$ distance of $X^{j+1}_\mathrm{rpca}$. By expansion of the left hand side of \eqref{eq:niht_eq1a}
	\begin{align}
		\left\| W^j - X^{j+1}\right\|_F^2 &= \left\| W^j - X_0 + X_0 - X^{j+1}\right\|_F^2 \\
			=& \left\| W^j - X_0 \right\|_F^2  + \left\| X_0 - X^{j+1}\right\|_F^2  + 2\inner{W^{j}- X_0}{X_0 - X^{j+1}}\\
			=& \left(\left\|W^j - X_0\right\|_F  + \epp\right)^2 \leq  \left\|W^j - X_0\right\|_F^2 + 2\epp \left\|W^j - X_0\right\|_F  + \epp^2 \label{eq:niht_eq2a}
	\end{align}
	where the last line \eqref{eq:niht_eq2a} follows from the inequality in \eqref{eq:niht_eq1b}. Subtracting $\|W^j - X_0\|_F^2$ from both sides of \eqref{eq:niht_eq2a} gives
	\begin{equation}
		\left\|  X^{j+1} -  X_0 \right\|_F^2 \leq 2\inner{W^j - X_0}{X^{j+1}- X_0} + 2\epp \left\| W^j - X_0\right\|_F + \epp^2. \label{eq:night_eq2}
	\end{equation}
	
	The matrix $W^j$ in the inner product on the right hand side of \eqref{eq:night_eq2} can be expressed using the update rule $W^{j} = X^j - \step_j \cA^*\left( \cA \left( X^j\right) - b\right)$
	\begin{align}
			2 & \inner{W^j - X_0}{X^{j+1}- X_0}  \nonumber\\
			&= 2\inp{X^j - X_0}{X^{j+1} - X_0} - 2\step_j \inner{\cA^* \cA \left( X^j - X_0\right)}{X^{j+1} - X_0} \label{eq:niht_eq3ebcomment}\\
			&= 2\inner{X^j - X_0}{X^{j+1} - X_0} - 2\step_j \inner{\cA \left( X^j - X_0\right)}{\cA\left(X^{j+1} - X_0\right)} \label{eq:niht_eq3a} \\ 
			& \leq 2 \left\| I -  \step_j A^*_Q A_Q\right\|_2 \left\| X^j - X_0 \right\|_F \left\| X^{j+1} - X_0 \right\|_F, \label{eq:niht_conv2}
	\end{align}
	where in the first line we use that $b = \cA(X_0)$ is the vector of measurements\footnote{Here it would be possible to extend the result to be stable under measurement error $\epb$ as done in Theorem \ref{thm:convex_recovery} by adding an error term in \eqref{eq:niht_eq3ebcomment}.} and linearity of $\cA$, in the second line we split the inner product into two inner products by linearity of $\cA$, and the inequality in the third line is a consequence of Lemma \ref{lemma:Ainp23}. 
	
	The matrix $W^j$ can be expressed using the update rule $W^{j} = X^j - \step_j \cA^*\left( \cA \left( X^j\right) - b\right)$ in the second term of the right hand side of \eqref{eq:night_eq2} and upper bounded by Lemma \ref{lemma:Ainp23}
	\begin{align}
		\left\|W^j- X_0\right\|_F & =  \left\| X^j - X_0 \Rev{-}\step_j \cA^*\left( \cA\left( X^j - X_0 \right)  \right)\right\|_2 \\
		& \leq  \left\| I- \step_j A_Q^* A_Q\right\|_2 \,  \left\| X^j - X_0 \right\|_F.\label{eq:niht_conv3}
	\end{align}
	By Lemma \ref{lemma:Ainp23}, the eigenvalues of $\left(I-\step_j A^*_Q A_Q\right)$ are bounded by
	\begin{equation}
			1-\step_j \left( 1 + \rict_3\right) \leq \lambda\left(I - \step_j A_Q^* A_Q\right) \leq 1 \Rev{-} \step_j \left( 1 - \rict_3\right) \label{eq:niht_ric1},
	\end{equation}
	where $\rict_3 := \Rev{\rict_{3r, 3s, \mu}}$.
	
	Consider the stepsize computed in Algorithm~\ref{algo:niht}, Line~$3$ inspired by the previous work on NIHT in the context of compressed sensing \citep{Blumensath2010normalized} and low-rank matrix sensing \citep{Tanner2013normalized}
	\begin{equation}
		\step_j = \frac{\left\| \Projarg{U^j}{\Omega^j}{R^j} \right\|^{\Rev{2}}_F}{\left\| \cA\left( \Projarg{U^j}{\Omega^j}{R^j} \right) \right\|^{\Rev{2}}_2}
	\end{equation}
	where the projection $\Projarg{U^j}{\Omega^j}{R^j}$ ensures that the residual $R^j$ is projected onto the set \Rev{$\LS_{m,n}(r,s,\mu)$}. Then we can bound $\step_j$ using the RIC of $\cA$ as
	\begin{equation}
		\frac{1}{1+\rict_1} \leq \step_j \leq \frac{1}{1-\rict_1}, \label{eq:niht_ric2}
	\end{equation}
	where $\rict_1 := \rict_{r,s,\Rev{\mu}}$. Combining \eqref{eq:niht_ric1} with \eqref{eq:niht_ric2} gives
	\begin{equation}
		1- \frac{1+\rict_3 }{1-\rict_1} \leq  \lambda\left(I - \step_j A_Q^* A_Q\right) \leq 1- \frac{1-\rict_3 }{1+\rict_1}. \label{eq:niht_op_bound}
	\end{equation}
	Since $\rict_3 \geq \rict_1$, the magnitude of the lower bound in \eqref{eq:niht_op_bound} is greater than the upper bound. Therefore
	\begin{equation}
		\eta := 2 \left(\frac{1+\rict_3}{1-\rict_1} - 1\right) \geq 2\left\| I- \step_j A_Q^* A_Q\right\|_2, \label{eq:niht_op_bound2}
	\end{equation}
	where the constant $\eta$ is strictly smaller than one if $\rict_3 < 1/5$. 
	
	Finally, the error in \eqref{eq:night_eq2} can be upper bounded by \eqref{eq:niht_conv2} combined with \eqref{eq:niht_conv3} with $\eta$ being the upper bound on the operator norm in \eqref{eq:niht_op_bound2}
	\begin{equation}
		\left\| X^{j+1} - X_0 \right\|_F^2 \leq \eta \left\| X^j - X_0 \right\|_F \, \left\| X^{j+1} - X_0 \right\|_F +  \eta  \epp \left\| X^j - X_0 \right\|_F + \epp^2. \label{eq:niht_conv4}
	\end{equation}
	
	It remains to show the inequality \eqref{eq:niht_conv4} implies the update rule contracts the error and the iterates $X^j$ converge to a matrix within the precision $\epp$ of the $\RPCA$. For the ease of notation we rewrite \eqref{eq:niht_conv4} using the notation $e^{j} := \| X^{j} - X_0 \|_F$ \Rev{and arrange the inequality into a squared form}
	\begin{align}
		\left(e^{j+1}\right)^2 &\leq \eta \, e^{j} \, e^{j+1} + \eta \,\epp \, e^j  + \epp^2 \nonumber \\
		\Rev{\left(e^{j+1} - \frac{1}{2} \eta e^j\right)^2} &\Rev{\leq \left(\frac{1}{2} \eta e^j + \epp \right)^2 }. \label{eq:niht_conv4b}
	\end{align}
	\Rev{Since the right hand side of \eqref{eq:niht_conv4b} is positive, we have $ e^{j+1} \leq \eta e^j + \epp$, which by $1/5 > \rict_3 \geq \rict_1$ gives an upper bound on the convergence rate
	\begin{equation*}
		\left\| X^{j+1} - X_0 \right\|_F \leq \frac{4\rict_3}{1-\rict_3} \left\| X^{j} - X_0 \right\|_F + \epp.
	\end{equation*}}
\end{proof}

\prooftitleT{Theorem~\ref{thm:naht_convergence}}{Guaranteed recovery of NAHT, Alg.~\ref{algo:naht}}{\pageref{thm:naht_convergence}}
\begin{proof}
	Let $b = \cA(X_0)$ be the vector of measurements\footnote{Again, it is possible to extend the result to the case when there is a measurement error $\epb$ as done in Theorem \ref{thm:convex_recovery} by having $b = \cA(X_0) + e$, with $\|e\|_2\leq \epb$.} of the matrix $X_0 \in \LS_{m,n}(r,s,\mu)$ and $V^j = L^j - \step_j^L \cA^*\left(\cA(X^j)  - b\right)$ to be the update of $L^j$ before the rank $r$ projection $L^{j+1} = \mathrm{HT}(V^j;\, r\Rev{, \mu})$. As a consequence of $L^{j+1}$ being the closest rank $r$ matrix to $V^j$ in the Frobenius norm we have that
	\begin{align}
		\left\|V^j - L_0\right\|_F^2 &\geq \left\|V^j - L^{j+1}\right\|_F^2 = \left\|V^j - L_0 + L_0 - L^{j+1}\right\|_F^2 \nonumber \\
					&= \left\|V^j - L_0 \right\|_F^2 + \left\|L_0 - L^{j+1} \right\|_F^2 + 2\inner{V^j - L_0}{L_0 - L^{j+1}}.\label{eq:nahtVj1}
	\end{align}
	Subtracting $\left\|V^j - L_0 \right\|_F^2$ from both sides of \eqref{eq:nahtVj1} and rearranging terms gives
	\begin{align}
		\left\|L_0 - L^{j+1} \right\|_F^2  &\leq 2\inner{V^j - L_0}{L^{j+1}-L_0} \\
		=&  2\inner{L^j - \step_j^L \cA^*\left(\cA\left(X^j -X_0 \right)\right) - L_0}{L^{j+1}-L_0} \\
		=&  2\inner{L^j - L_0 -  \step_j^L \cA^*\left(\cA\left(L^j - L_0 + S^j - S_0 \right)\right)}{L^{j+1}-L_0} \\
		=&  2\inner{L^j - L_0}{L^{j+1}-L_0} - 2 \step_j^L \inner{\cA\left(L^j - L_0\right)}{\cA\left(L^{j+1} - L_0\right)}\nonumber \\
		& \qquad - 2 \step_j^L \inner{\cA\left(S^j - S_0\right)}{\cA\left(L^{j+1} - L_0\right)} \\
		\leq& 2 \left\| I - \step_j^L A_Q^* A_Q \right\|  \left\| L^j - L_0 \right\|_F \left\| L^{j+1} - L_0 \right\|_F \nonumber \\
		&\qquad + 2\step_j^L \, \rho_2 \left\|S^{j} - S_0\right\|_F  \left\|L^{j+1} - L_0\right\|_F, \label{eq:nahtVj}
	\end{align}
	where in the second line we expanded $V^j$ using the update rule $V^j = L^j - \step_j^L\cA\left( \cA(X^j) - b\right)$ and $b = \cA(X_0)$, in the third line we expanded $X^j = L^j + S^j$, in the fourth line we split the inner product into two inner products by linearity of $\cA$, and in the last line the inequality comes from Lemma \ref{lemma:Ainp23} bounding the first two terms and \Rev{Lemma} \ref{lemma:Ainp1} bounding the third term with $\rho_2 := \left( \rict_2 + \frac{2\rcc_2}{1-\rcc^2_2} \right)$ where $\rict_2:= \rict_{2r,2s,\mu}$ and $\rcc_2:=\mu \frac{2r\sqrt{2s}}{\sqrt{mn}}$ since $\left( L^{j+1} - L_0 + S^j - S_0 \right) \in \LS_{m,n}(2r,2s,\mu)$. Dividing both sides of \eqref{eq:nahtVj} by $\left\|L_0 - L^{j+1} \right\|_F$ gives
	\begin{equation}
		\left\|L_0 - L^{j+1} \right\|_F  \leq 2 \left\| I - \step_j^L A_Q^* A_Q \right\| \, \left\| L^j - L_0 \right\|_F + 2\step_j^L \, \rho_2 \left\|S^{j} - S_0\right\|_F. \label{eq:nahtA}
	\end{equation}
	
	Let $W^j = S^j  - \step_j^S \cA^*\left(\cA(X^{j+\frac{1}{2}})  - b\right)$ be the subsequent update of $S^j$ before the $s$-sparse projection $S^{j+1} = \HTSarg{W^j}{s}$. By $S^{j+1}$ being the closest $s$ sparse matrix to $W^j$ in the Frobenius norm and by $\left\|S_0\right\|_0 \leq s$, it follows that
	\begin{align}
		\left\|W^j - S_0\right\|_F^2 &\geq \left\|W^j - S^{j+1}\right\|_F^2 = \left\|W^j - S_0 + S_0 - S^{j+1}\right\|_F^2 \nonumber \\
					&= \left\|W^j - S_0 \right\|_F^2 + \left\|S_0 - S^{j+1} \right\|_F^2 + 2\inner{W^j - S_0}{S_0 - S^{j+1}}. \label{eq:nahtWj1}
	\end{align}
	Subtracting $\left\|W^j - S_0\right\|_F^2$ from both sides in \eqref{eq:nahtWj1} and rearranging terms gives
	\begin{align}
		\left\|S_0 - S^{j+1} \right\|_F^2  &\leq 2\inner{W^j - S_0}{S^{j+1}-S_0} \\
		=& 2\inner{S^j - \step_j^S \cA^*\left(\cA\left(X^{j+\frac{1}{2}} -X_0 \right)\right) - S_0}{S^{j+1}-S_0} \\
		=&  2\inner{S^j - S_0 -  \step_j^S \cA^*\left(\cA\left(L^{j+1} - L_0 + S^j - S_0 \right)\right)}{S^{j+1}-S_0} \\
		=&  2\inner{S^j - S_0}{S^{j+1}-S_0} - 2 \step_j^S\inner{\cA\left(S^j - S_0\right)}{\cA\left(S^{j+1} - S_0\right)} \nonumber \\
		& \qquad - 2 \step_j^S \inner{\cA\left(L^{j+1} - L_0\right)}{\cA\left(S^{j+1} - S_0\right)} \\
		\leq& 2 \left\|I - \step_j^S A_Q^* A_Q \right\| \left\| S^j - S_0 \right\|_F \left\| S^{j+1} - S_0 \right\|_F \nonumber \\
		& \qquad + 2\step_j^S \, \rho_2 \left\|L^{j+1} - L_0\right\|_F \left\|S^{j+1} - S_0\right\|_F, \label{eq:nahtWj}
	\end{align}
	where in the second line we express $W^j$ using the update rule $W^j = S^j - \step_j^S\cA\left( \cA(X^{j+\frac{1}{2}}) - b\right)$ and $b = \cA(X_0)$, in the third line we expanded $X^{j+\frac{1}{2}} = L^{j+1} + S^j$, in the fourth line we split the inner product into two inner products by linearity of $\cA$, and the inequality in the last line comes from Lemma~\ref{lemma:Ainp23} bounding the first two terms and Lemma~\ref{lemma:Ainp1} bounding the third term with $\rho_2 := \left( \rict_2 + \frac{2\rcc_2}{1-\rcc^2_2} \right)$ where $\rict_2:= \rict_{2r,2s,\mu}$ and $\rcc_2:=\rcc_{2r,2s,\mu}$ since $\left( L^{j+1} - L_0 + S^{j+1} - S_0 \right) \in \LS_{m,n}(2r,2s,\mu)$.
	Dividing both sides of \eqref{eq:nahtWj} by $\left\|S_0 - S^{j+1} \right\|_F$ gives
	\begin{equation}
		\left\|S_0 - S^{j+1} \right\|_F  \leq 2 \left\| I - \step_j^S A_Q^T A_Q \right\| \left\| S^j - S_0 \right\|_F + 2\step_j^S \, \rho_2 \left\|L^{j+1} - L_0\right\|_F. \label{eq:nahtB}
	\end{equation}
	Adding together \eqref{eq:nahtA} and \eqref{eq:nahtB}
	\begin{align}
		\left\|L_0 - L^{j+1} \right\|_F &+ \left\|S_0 - S^{j+1} \right\|_F  \leq \nonumber \\
		&  2 \left\|I - \step_j^L A_Q^T A_Q \right\|  \left\| L^j - L_0 \right\|_F + 2\step_j^L \, \rho_2 \left\|S^{j} - S_0\right\|_F \nonumber \\
		\quad + &2 \left\| I - \step_j^S A_Q^T A_Q \right\| \left\| S^j - S_0 \right\|_F + 2\step_j^S \, \rho_2 \left\|L^{j+1} - L_0\right\|_F \label{eq:nahtAB1},
	\end{align}
	which after rearranging terms in \eqref{eq:nahtAB1} becomes
	\begin{align}
		\left( 1 - 2\step_j^S \, \rho_2 \right)&\left\|L_0 - L^{j+1} \right\|_F + \left\|S_0 - S^{j+1} \right\|_F \nonumber \\  
		\leq \, &2\, \left\|I - \step_j^L A_Q^T A_Q \right\| \, \left\| L^j - L_0 \right\|_F  \nonumber \\ 
		 + &2\,\left( \left\|I - \step_j^S A_Q^T A_Q \right\|  + \step_j^L \, \rho_2 \right)\left\| S^j - S_0 \right\|_F\label{eq:nahtAB2}
	\end{align}
	and because $\step_j^S, \step_j^L,\rict_2 \geq 0$ and $\rcc_2 \in (0,1)$, subtracting $2\step_j^S \rho_2 \|S_0 - S^{j+1} \|_F$ on the left does not increase the left hand side while adding $2 \step_j^L \, \rho_2 \left\| L^j - L_0 \right\|_F $ on the right does not decrease the right hand side of \eqref{eq:nahtAB2}, therefore 
	\begin{align}
		\left( 1 - 2\step_j^S \rho_2 \right)& \left( \left\|L_0 - L^{j+1} \right\|_F + \left\|S_0 - S^{j+1} \right\|_F\right)  \nonumber \\
		\leq 2 & \left( \left\|I - \step_j A_Q^T A_Q \right\|  + \step_j^L \rho_2 \right) \left(  \left\| L^j - L_0 \right\|_F + \left\| S^j - S_0 \right\|_F \right), \label{eq:nahtAB3}
	\end{align}
	\Rev{where $\left\|I - \step_j A_Q^T A_Q \right\| = \max \left\{ \left\|I - \step_j^L A_Q^T A_Q \right\|, \left\|I - \step_j^S A_Q^T A_Q \right\| \right\}$.} Dividing both sides of \eqref{eq:nahtAB3} by $\left( 1 - 2\step_j^S \rho_2 \right)$ simplifies to
	\begin{align}
		\left\|L_0 - L^{j+1} \right\|_F &+ \left\| S_0 - S^{j+1} \right\|_F \nonumber \\
		 \leq  2 &\frac{ \left\|I - \step_j A_Q^T A_Q \right\|  + \step_j^L \rho_2}{1 - 2\step_j^S \rho_2 } \left(  \left\| L^j - L_0 \right\|_F + \left\| S^j - S_0 \right\|_F \right). \label{eq:nahtAB3b}
	\end{align}
	By Lemma \ref{lemma:Ainp23}, the eigenvalues of $\left(I-\step_j A^T_Q A_Q\right)$ can be bounded as
	\begin{equation}
			1-\step_j \left( 1 + \rict_3\right) \leq \lambda\left(I - \step_j A_Q^T A_Q\right) \leq 1 \Rev{-} \step_j \left( 1 - \rict_3\right), \label{eq:naht_ric1}
	\end{equation}
	with $\rict_3 := \rict_{3r, 3s, \mu}$ being the RIC of $\cA$. By $\step_j^L$ and $\step_j^S$ being the normalized stepsizes as introduced in \citep{Blumensath2010normalized,Tanner2013normalized} 
	\begin{equation}
		\step_j^L = \frac{\left\| \Projarg{U^j}{\Omega^j}{R^j} \right\|^2_F}{\left\| \cA\left( \Projarg{U^j}{\Omega^j}{R^j} \right)\right\|^2_2} 
		\quad\text{and}\quad 
		\step_j^S = \frac{\left\| \Projarg{U^{j+1}}{\Omega^j}{R^j} \right\|^2_F}{\left\| \cA\left( \Projarg{U^{j+1}}{\Omega^j}{R^j} \right)\right\|^2_2}
	\end{equation}
	where the projection $\Projarg{U^j}{\Omega^j}{R^j},\Projarg{U^{j+1}}{\Omega^j}{R^{j+\frac{1}{2}}}$ ensures that the residual $R^j$ and $R^{j+\frac{1}{2}}$ is projected into the set $\LS_{m,n}(r,s,\mu)$. Then, it follows from the RIC for $\cA$ that the stepsizes~$\step_j^L,\step_j^S$ can be bounded as
	\begin{equation}
		\frac{1}{1+\rict_1} \leq \step_j^{L/S} \leq \frac{1}{1-\rict_1},\label{eq:naht_ric2}
	\end{equation}
	where $\rict_1 := \rict_{r,s,\mu}$. Putting \eqref{eq:naht_ric1} and \eqref{eq:naht_ric2} together
	\begin{equation}
		1- \frac{1+\rict_3 }{1-\rict_1} \leq  \lambda\left(I - \step_j^{L/S} A_Q^T A_Q\right) \leq 1- \frac{1-\rict_3 }{1+\rict_1}. \label{eq:naht_op_bound}
	\end{equation}
	Since $\rict_3 \geq \rict_1$ we have that the magnitude of the lower bound in \eqref{eq:naht_op_bound} is greater than the upper bound. Therefore
	\begin{equation}
		\frac{1+\rict_3}{1-\rict_1} - 1\geq \left\| I- \step^{L/S}_j A_Q^T A_Q\right\|_2 \label{eq:naht_op_bound2}.
	\end{equation}
	
	Finally, the constant on the right hand side of \eqref{eq:nahtAB3b} can be upper bounded
	\begin{align}
		\eta &:= 2\frac{ \left\|I - \step_j^S A_Q^T A_Q \right\|  + \step_j^L \rho_2 }{1 - 2\step_j^S \rho_2 } \label{eq:naht_eta_bound1} \\
		&\leq 2 \frac{ \left(\frac{1+\rict_3}{1-\rict_1} - 1 \right) + \frac{ 1 }{1-\rict_1}\left( \rict_2 + \frac{2\rcc_2}{1-\rcc_2^2}\right) }{1 - 2\frac{ 1 }{1-\rict_1}\left( \rict_2 + \frac{2\rcc_2}{1-\rcc_2^2}\right) } = 2 \frac{\rict_3 + \rict_1 + \rict_2 + \frac{2\rcc_2}{1-\rcc_2^2} }{ 1- \rict_1 - 2\rict_2 - \frac{4\rcc_2}{1-\rcc_2^2} } \label{eq:naht_eta_bound2} \\
		& \leq \frac{6\rict_3 + \frac{4\rcc_2}{1-\rcc_2^2}}{1-3\rict_3 - \frac{4\rcc_2}{1-\rcc_2^2}} \label{eq:naht_eta_bound3}
	\end{align}
	where the inequality in the second line in \eqref{eq:naht_eta_bound2} comes from upper bounds in \eqref{eq:naht_op_bound2} and in \eqref{eq:naht_ric2}, and the third line in \eqref{eq:naht_eta_bound3} follows from $\rict_2 \geq \rict_1$.
	
	\Rev{To ensure that $\eta < 1$, it suffices to show that the right-hand side in \eqref{eq:naht_eta_bound3} is smaller than one, which translates to
	\begin{equation}
		\rict_3 \leq \frac{1}{9}\left(1 - 8\frac{\rcc_2}{1-\rcc^2_2}\right), \label{eq:naht_eta_bound5} 
	\end{equation}
	which is satisfied when $\rict_3 \leq \frac{1}{9} - \rcc_2$.} For $\rict_{3r,3s,\mu} < \frac{1}{9} - \rcc_2$ the inequality in \eqref{eq:nahtAB3b} implies contraction of the error
	\begin{equation}
		\left\|L_0 - L^{j+1} \right\|_F + \left\| S_0 - S^{j+1} \right\|_F \leq  \eta \left(  \left\| L^j - L_0 \right\|_F + \left\| S^j - S_0 \right\|_F \right),
	\end{equation}
	because $\eta < 1$, which guarantees linear convergence of iterates $L^j$ and $S^j$ to $L_0$ and $S_0$ respectively.
\end{proof}

\section{Numerical experiments \label{sec:numerics}}

This section demonstrates the computational efficacy of recoverying a
low-rank plus sparse matrix $X_0\in \LS_{m,n}(r,s,\mu)$ from its
undersampled values $\mathcal{A}(X_0)$.  Section
\ref{sec:num_synthetic} considers synthetic examples where matrices in
$X_0\in \LS_{m,n}(r,s,\mu)$ are created, and recovery from their
undersampled values attempted for the following algorithms: NIHT
(Alg.\ \ref{algo:niht}), NAHT (Alg.\ \ref{algo:naht}), 
SpaRCS~\cite{Waters2011sparcs}, and the convex relaxation \eqref{eq:convex}.  Figure \ref{fig:phase_nonconvex}
presents empirically observed phase transitions, which indicate the
values of model complexity $r,s,$ and measurements $p$ for which
recovery is possible.  Figure 
\ref{fig:convergence1} and \ref{fig:convergence2} gives examples of 
convergence rates for NIHT, NAHT, and SpaRCS, including contrasting
different methods to implement the 
projection NIHT, step 5 of Alg.\ \ref{algo:niht}. Section \ref{sec:applications} presents
applications to dynamic-foreground/static-background and computational
multispectal imaging.  An additional phase transition simulation for
the convex relaxation is given in Appendix \ref{app:convex}.  
Software to reproduce the experiments in this section is publicly
available\footnote{\url{https://github.com/SimonVary/lrps-recovery}}.

\subsection{Empirical average case performance on synthetic data}\label{sec:num_synthetic}

Synthetic matrices $X_0 = L_0 + S_0\in\LS_{m,n}(r,s,\mu)$ are
generated using the experimental setup proposed in the Robust PCA
literature \cite{Netrapalli2014provable,Yi2016fast,
  Cai2019accelerated}. The low-rank component is formed as $L_0 =
UV^T$, where $U\in\pR^{m\times r}, V\in\pR^{n\times r}$ are two random
matrices having their entries drawn i.i.d.\ from the standard Gaussian
distribution. The support set of the sparse component $S_0$ is
generated by sampling a uniformly random subset of $[m]\times [n]$
indices of size $s$ and each non-zero entry $\left(S_0\right)_{i,j}$
is drawn from the uniform distribution over $\left[-\E\left(|
    (L_0)_{i,j}|\right), \E\left( | (L_0)_{i,j} | \right) \right]$.
Each synthetic matrix is measured using linear operators $\A:\pR^{m\times n} \rightarrow \pR^p$. The random Gaussian measurement operators are constructed by $p$ matrices $A^{(\ell)}\in\pR^{m\times n}$ whose entries are sampled from Gaussian distribution $A^{(\ell)}_{i,j}\sim\mathcal{N}(0, 1/p)$ where $p$ is the number of measurements. The Fast Johnson-Lindenstrauss Transform is implemented as
\begin{equation}
	\A_{\mathrm{FJLT}}\left(X\right) = R H D\, \mvec{X}, \label{eq:numerics_fjlt}
\end{equation}
where $R\in\pR^{p \times mn}$ is a restriction matrix constructed from a $mn \times mn$ identity matrix with $p$ rows randomly selected, $H\in\pR^{mn\times mn}$ is discrete cosine transform matrix, $D\in\pR^{mn \times mn}$ is a diagonal matrix whose entries are sampled independently randomly from $\left\{-1,1\right\}$, and $\mvec{X}\in\pR^{mn}$ is the vectorized matrix $X\in\pR^{m \times n}$.

Theorems \ref{thm:rip_ls}, \ref{thm:convex_recovery},
\ref{thm:niht_convergence}, and \ref{thm:naht_convergence} indicate
that recovery of $X_0$ from $\mathcal{A}(X_0)$ depends on the problem
dimensions through the ratios of the number of measurements $p$ with
the ambient dimension $mn$, and the minimum number of measurements,
$r(m+n-r) + s$, through an undersampling and two oversampling ratios 
\begin{equation}\label{eq:delta_rho}
	\delta = \frac{p}{mn}\quad\text{and}\quad \rho_r = \frac{r(m+n-r)}{p},\quad \rho_s=\frac{s}{~p~}.
\end{equation}
The matrix dimensions $m$ and $n$ are held fixed, while $p$, $r$ and $s$ are chosen according to varying parameters $\delta, \rho_r$ and $\rho_s$. For each pair of $\rho_r, \rho_s\in \left\{0, 0.02, 0.04, \ldots, 1
\right\}$ where $\rho_r + \rho_s \leq 1$, with the sampling ratio
restricted to values $\delta \in \left\{0.02, 0.04, \ldots, 1
\right\}$, $20$ simulated recovery tests are conducted and we compute
the critical subsampling ratio $\delta^*$ above which more than half
of the experiments succeeded. For the linear 
transform $\A$ drawn from the (dense) Gaussian distribution, the
highest per iteration cost in NIHT and NAHT comes from applying $\A$
to the residual matrix, which requires $pmn$ scalar multiplications
which scales proportionally to $\left(mn\right)^2$.  For this reason,
our tests are restricted to the matrix size of $m=n=100$ in the
case of NIHT and NAHT, and to a smaller size $m=n=30$ for testing the recovery
by solving the convex relaxation \eqref{eq:convex} with semidefinite programming \cite{Toh1999sdpt3} that has
$\bigO\left((mn)^2\right)$ variables which is more computationally
demanding\footnote{As an example, a low-rank plus sparse matrix with
  $m = n = 100$ with $\rho_r = \rho_s = 0.1$ undersampled and measured
  with Gaussian matrix with $\delta = 0.5$ takes $2.5$ seconds and
  $2.3$ seconds to recover using NIHT and NAHT respectively, while the
  recovery using the convex relaxation takes over $7$ hours.} compared
to the hard thresholding gradient descent methods. 
Algorithms are terminated at iteration $\ell$ when either: the
relative residual error is smaller than $10^{-6}$, that is when
$\|\A(X^\ell) - b\|_2/\|b\|_2 \leq 10^{-6} \|b\|_2$, or the relative
decrease in the objective is small  
\begin{equation}
	\left(\frac{\|\A(X^{\ell}) - b\|_2}{\|\A(X^{\ell-15}) - b\|_2}\right)^{1/15} > 0.999,
\end{equation}
or the maximum of $300$ iterations is reached.  An algorithm is
considered to have successfully recovered $X_0\in\LS_{m,n}(r,s,\mu)$
if it returns a matrix $X^\ell\in\LS_{m,n}(r,s,\mu)$ that is within
$10^{-2}$ of $X_0$  in the relative Frobenius error, $\| X^\ell -
X_0\|_F\leq 10^{-2}\|X_0\|_F$. 

\newcommand\figwidthH{0.4} 
\begin{figure}[t!]
        \centering
         \begin{subfigure}[b]{\figwidthH\textwidth}
 		\includegraphics[width=1\textwidth]{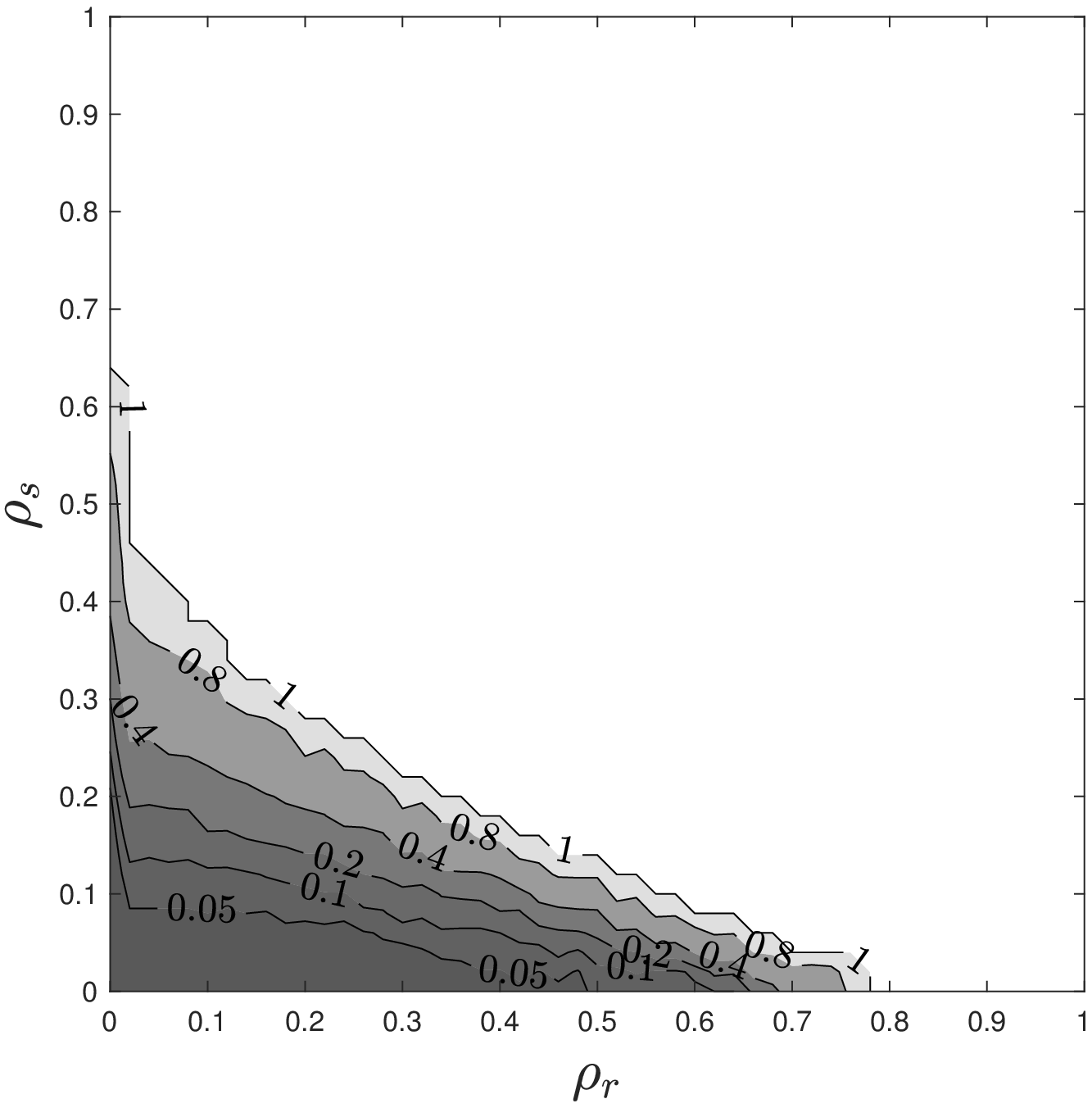}
		\subcaption{NIHT (Gaussian measurements) \label{fig:phase_niht_gausst}}
	\end{subfigure}
	\hspace{0.01\textwidth}
         \begin{subfigure}[b]{\figwidthH\textwidth}
	 	\includegraphics[width=1\textwidth]{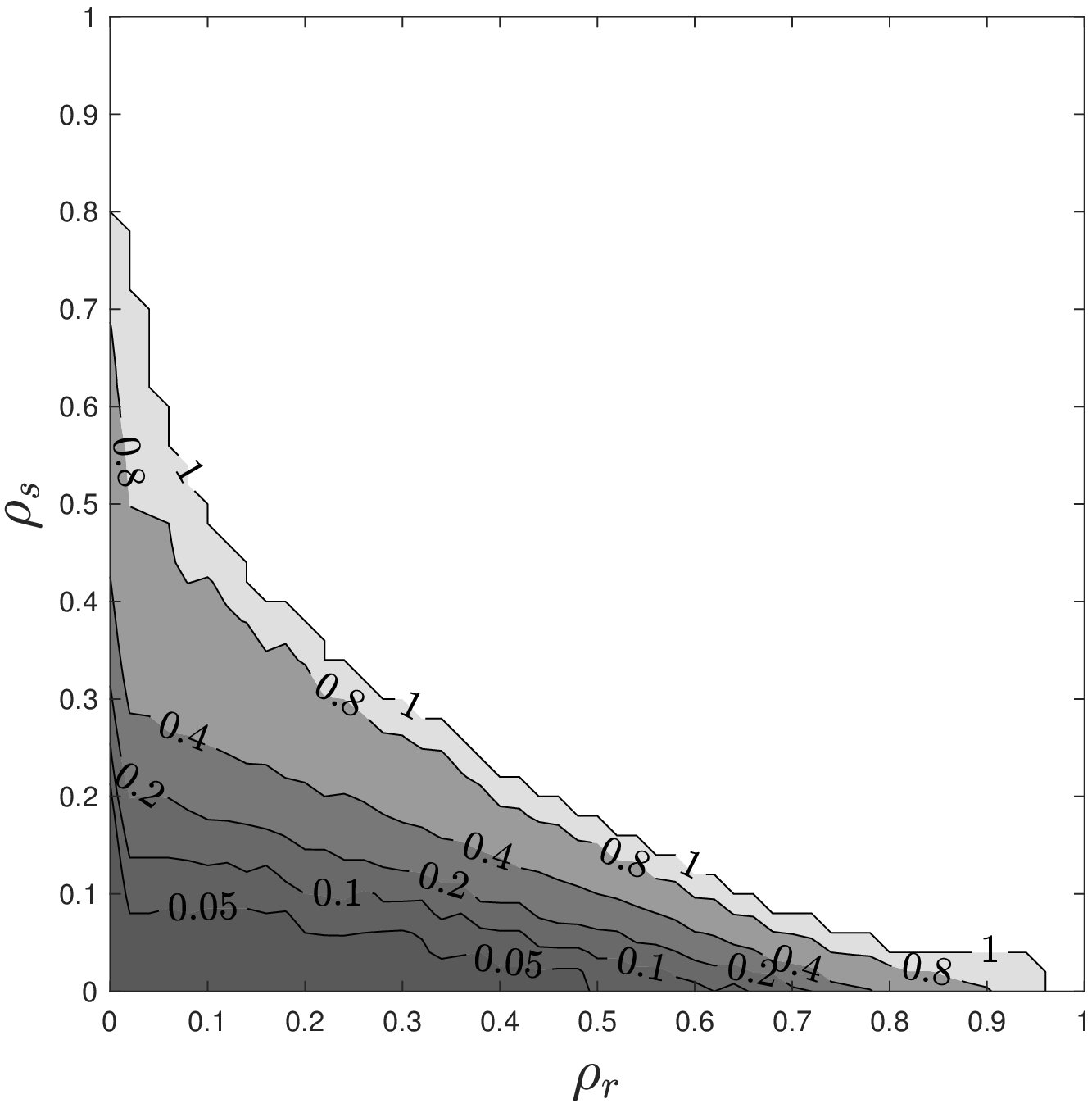}
		\subcaption{NIHT (FJLT measurements) \label{fig:phase_niht_fjlt}}
	\end{subfigure}\\
	\begin{subfigure}[b]{\figwidthH\textwidth}
 		\includegraphics[width=1\textwidth]{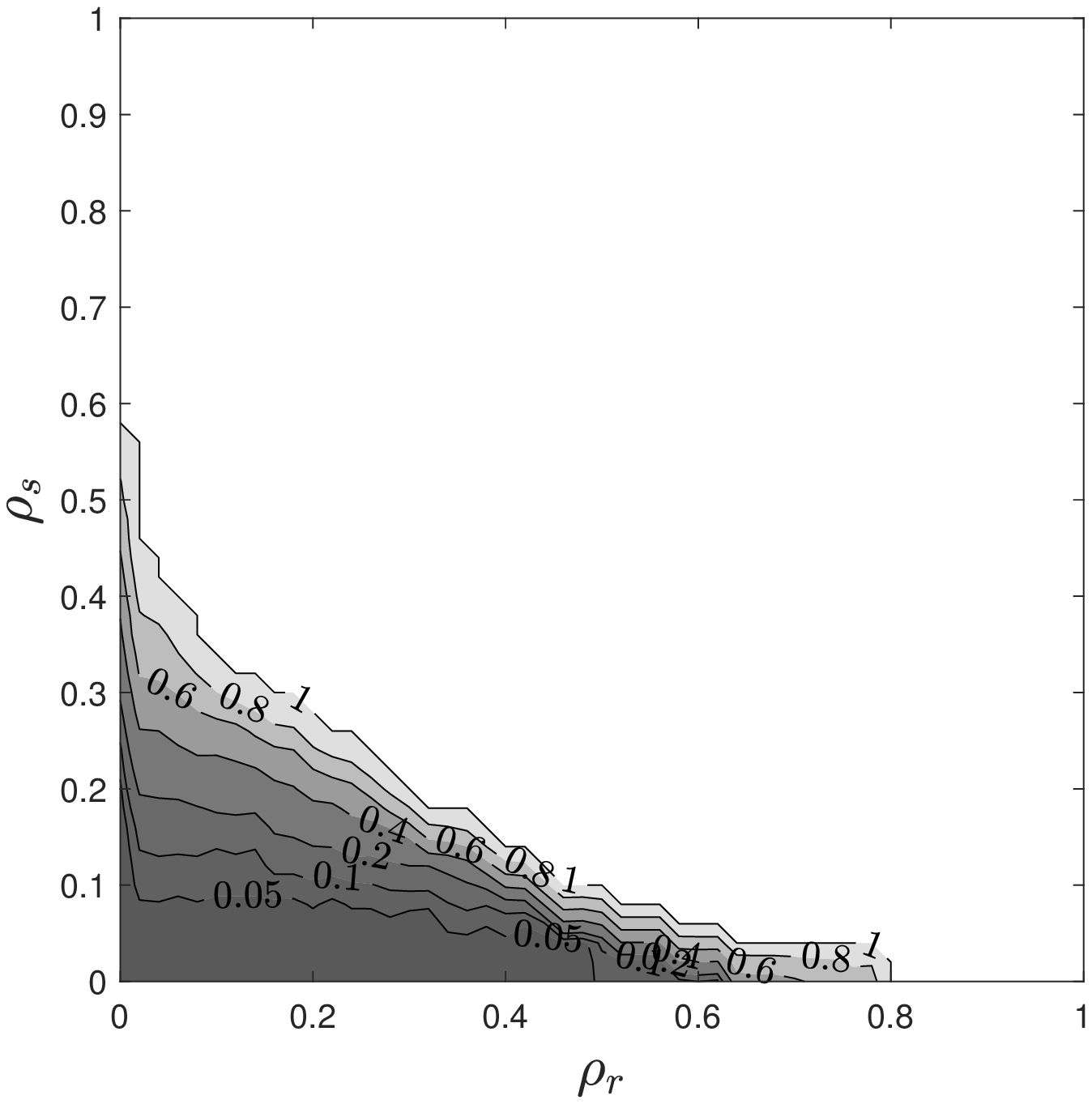}
		\subcaption{NAHT (Gaussian measurements) \label{fig:phase_naht_gausst}}
	\end{subfigure}
	\hspace{0.01\textwidth}
         \begin{subfigure}[b]{\figwidthH\textwidth}
	 	\includegraphics[width=1\textwidth]{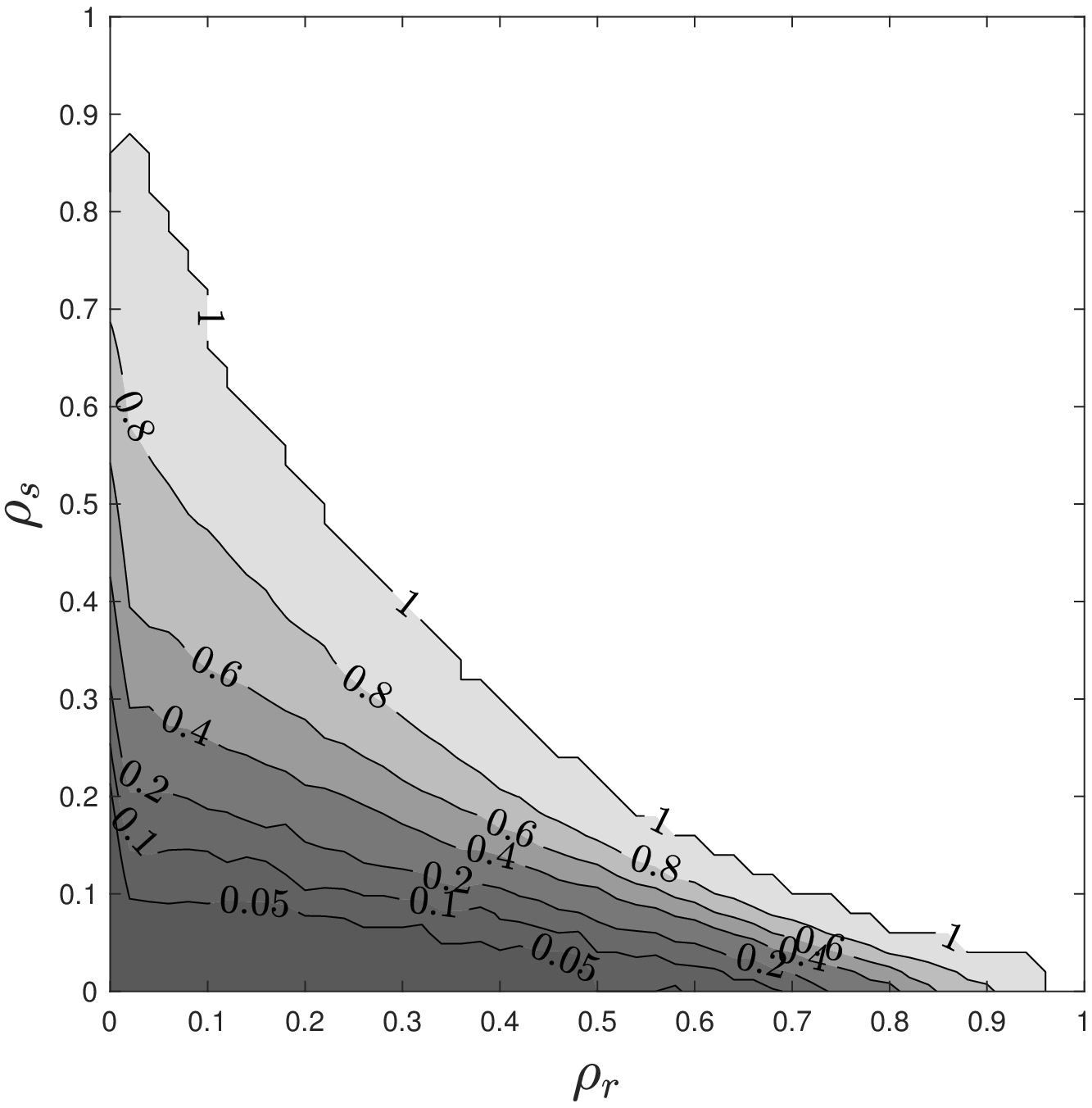}
		\subcaption{NAHT (FJLT measurements) \label{fig:phase_naht_fjlt}}
	\end{subfigure}
        	\caption{Phase transition level curves denoting the
                  value of $\delta^*$ for which values of $\rho_r$ and
$\rho_s$ {\em below which} are recovered for at least half of
the experiments for $\delta$, $\rho_r$, and $\rho_s$ as given by
\eqref{eq:delta_rho}.  NIHT is observed to recover matrices of higher
ranks and sparsities from FJLT than from Gaussian measurements, while
the phase transitions for NIHT and NAHT are comparable. The Robust PCA
projection in NIHT, step 5 in Alg.\ \ref{algo:niht}, is performed by
AccAltProj \cite{Cai2019accelerated}. \label{fig:phase_nonconvex}} 
\end{figure}

Figure~\ref{fig:phase_nonconvex} depicts the  phase transitions of
$\delta$ above which NIHT and NAHT successfully recovers $X_0$ in more
than half of the experiments.  For example, the level curve 0.4 in
Fig.~\ref{fig:phase_nonconvex} denotes the values of $\rho_r$ and
$\rho_s$ {\em below which} recovery is possible for at least half of
the experiments for $p=0.4mn$ and $\rho_r,\rho_s$ as given by
\eqref{eq:delta_rho}.  Note that the bottom left portion of
Fig.~\ref{fig:phase_nonconvex} corresponds to smaller values of model
complexity $(r,s)$ and are correspondingly easier to recover than
larger values of $(r,s)$.  Both algorithms are observed to recover
matrices with prevalent rank structure, $\rho_r \leq 0.6$, even from
very few measurements as opposed to matrices with prevalent sparse
structure requiring in general more measurements for a successful
recovery. Phase transitions corresponding to the sparse-only ($\rho_r
= 0$) and to the rank-only ($\rho_s=0$) cases are roughly in agreement
with phase transitions that have been observed for non-convex
algorithms in compressed sensing~\cite{Blanchard2015performance} and
matrix completion literature~\cite{Tanner2013normalized,
  Blanchard2015cgiht}. We observe that NAHT achieves almost identical
performance to NIHT in terms of possible recovery despite not
requiring the computationally expensive Robust PCA projection in every
iteration. For both algorithms we see that the successful recovery is
possible for matrices with higher ranks and sparsities in the case of
FJLT measurements compared to Gaussian measurements. 


Equivalent experiments are conducted for the convex relaxation
\eqref{eq:convex}, but with smaller matrix size $30\times 30$ and
limited to $10$ simulations for each set of parameters due to the
added computational demands. The convex optimization is formulated
using CVX modeling framework~\cite{Stephen2014cvx} and solved in
Matlab by the semidefinite programming optimization package SDPT3~\cite{Toh1999sdpt3}. We observe that
recovery by solving the convex relaxation is successful for
somewhat lower ranks and sparsities and requiring larger sampling
ratio $\delta$ compared to the non-convex algorithms. The observed
phase transitions of the convex relaxation alongside phase transitions for $m=n=30$ experiments with NIHT and NAHT are depicted in Figure
\ref{fig:phase_cvx} in \ref{app:convex}. Comparing the phase transitions of the non-convex algorithms in Fig.~ \ref{fig:phase_nonconvex} and Fig.~\ref{fig:phase_cvx} show that with the increased problem size, the phase transition are independent of the dimension with only small differences due to finite dimensional effects of the smaller problem size in the case of $m = n = 30$.

\newcommand\figwidthT{0.325} 
\begin{figure}[t!]
        \centering
         \begin{subfigure}[b]{\figwidthT\textwidth}
 		\includegraphics[width=1\textwidth]{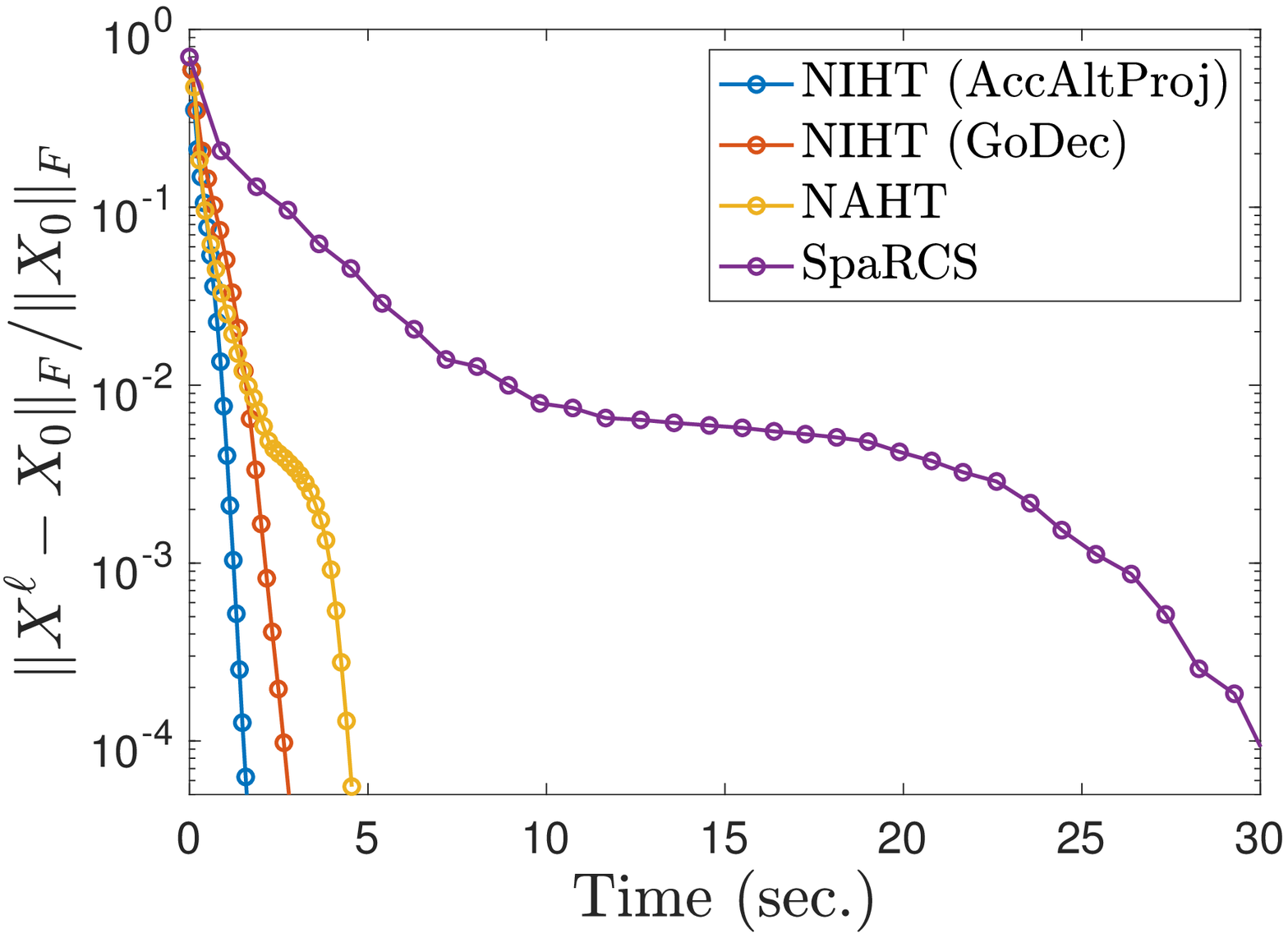}
		\caption{$\rho_r  = \rho_s = 0.05$}
	\end{subfigure}
	\begin{subfigure}[b]{\figwidthT\textwidth}
	 	\includegraphics[width=1\textwidth]{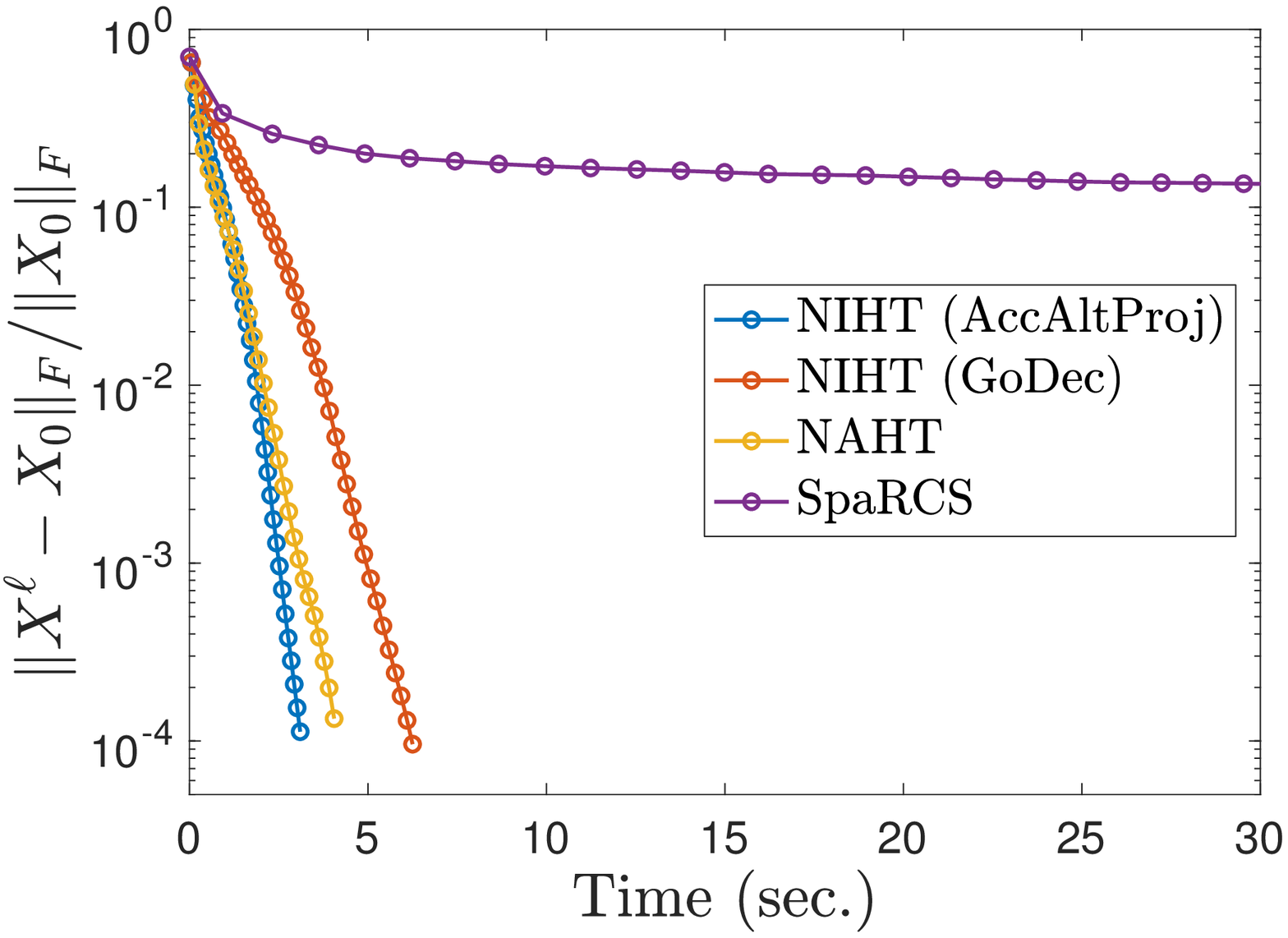}
		\caption{$\rho_r = \rho_s = 0.1$}
	\end{subfigure}
	 \begin{subfigure}[b]{\figwidthT\textwidth}
	 	\includegraphics[width=1\textwidth]{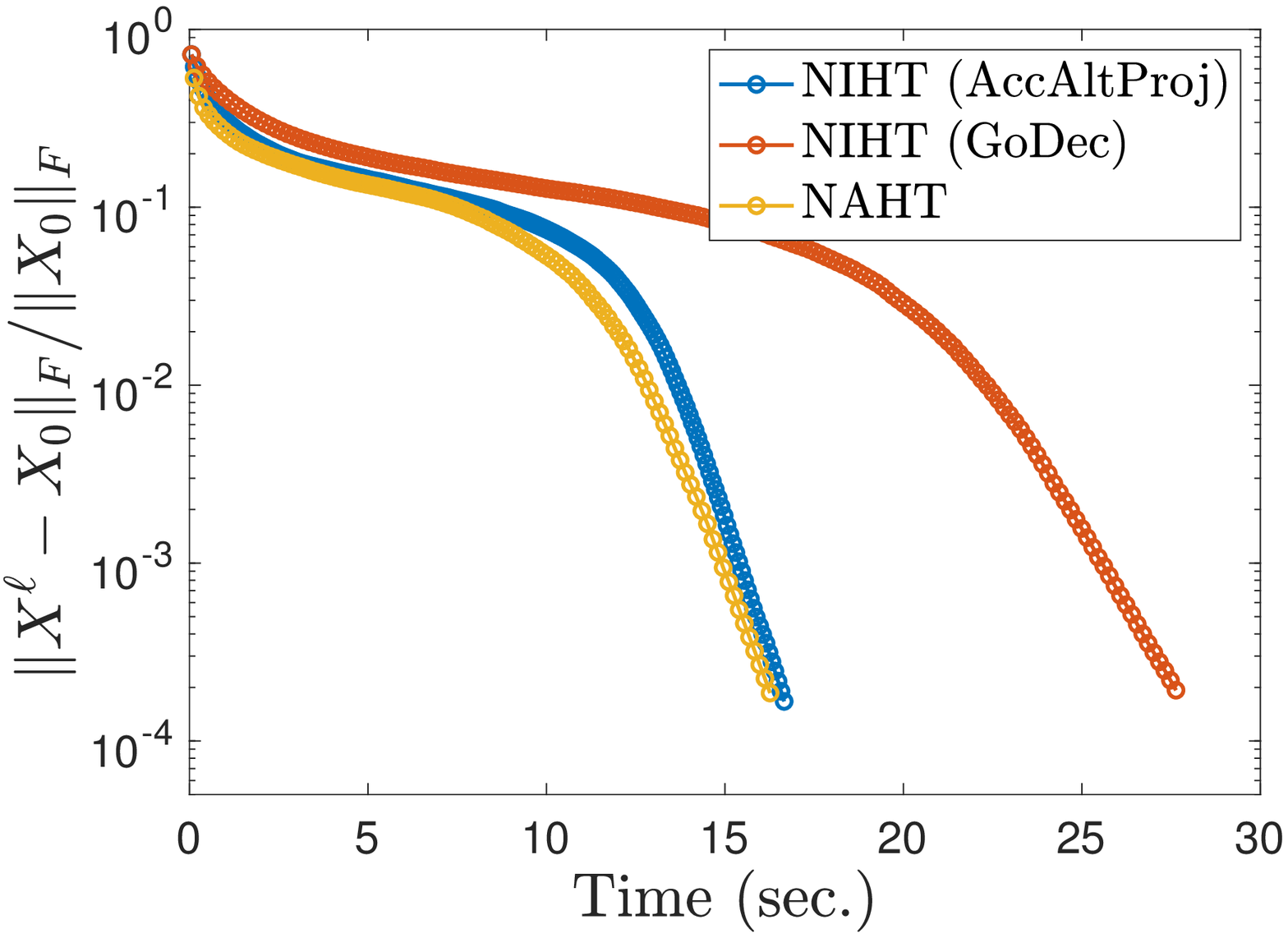}
		\caption{$\rho_r = \rho_s = 0.2$}
	\end{subfigure}
        	\caption{Relative error in the approximate $X^{\ell}$
                  as a function of time for synthetic problems with
                  $m=n=100$ and $p=(1/2)100^2$, $\delta=1/2$, for
                  Gaussian linear measurements $\A$.  In (b), SpaRCS
                  converged in 171 sec. (45 iterations), and in (c),
                  SpaRCS did not converge. \label{fig:convergence1}} 
\end{figure}

Figure~\ref{fig:convergence1} presents convergence timings of Matlab
implementations of the three non-convex algorithms used for recovery
of matrices with $m=n=100$ from $p=(1/2)10^2$ ($\delta=1/2$)
measurements and three values of $\rho_r = \rho_s =
\left\{0.05, 0.1, 0.2 \right\}$. The convergence
results are presented for two variants of NIHT with different Robust
PCA algorithms Accelerated Alternating Projection
(AccAltProj)~\cite{Cai2019accelerated} and Go Decomposition
(GoDec)~\cite{Zhou2011godec} in the projection step 5 of Alg.~\ref{algo:niht}. 
Both NIHT and NAHT converge at a much faster rate than the existing non-convex
algorithm for low-rank plus sparse matrix recovery 
SpaRCS~\cite{Waters2011sparcs}.  All the algorithms take longer to recover
a matrix for increased rank $r$ and/or sparsity $s$.  

The computational efficacy of NIHT compared to NAHT depends on the cost of computing the Robust PCA calculation in comparison to the cost of applying $\A$. NAHT computes two step sizes in each iteration which results into computing $\A$ twice per iteration in comparison to just one such computation per iteration in the case of NIHT. On the other hand, NIHT involves solving Robust PCA in every iteration for the projection step whereas NAHT performs computationally cheaper singular value decomposition (SVD) and sparse hard thresholding projection.

Figure~\ref{fig:convergence2} illustrates the convergence of the
individual low-rank and sparse components $\|L^{\ell}-L_0\|_F$ and
$\|S^{\ell}-S_0\|_F$ as a function of time.  The algorithms are
observed to approximate the the low-rank factor more accurately than
the sparse component and that the computational time increases for
larger values of sparsity fraction $\rho_s$.  Moreover, for both NIHT
and NAHT the relative error of both components decreases together. 

\begin{figure}[t!]
        \centering
         \begin{subfigure}[b]{\figwidthH\textwidth}
 		\includegraphics[width=1\textwidth]{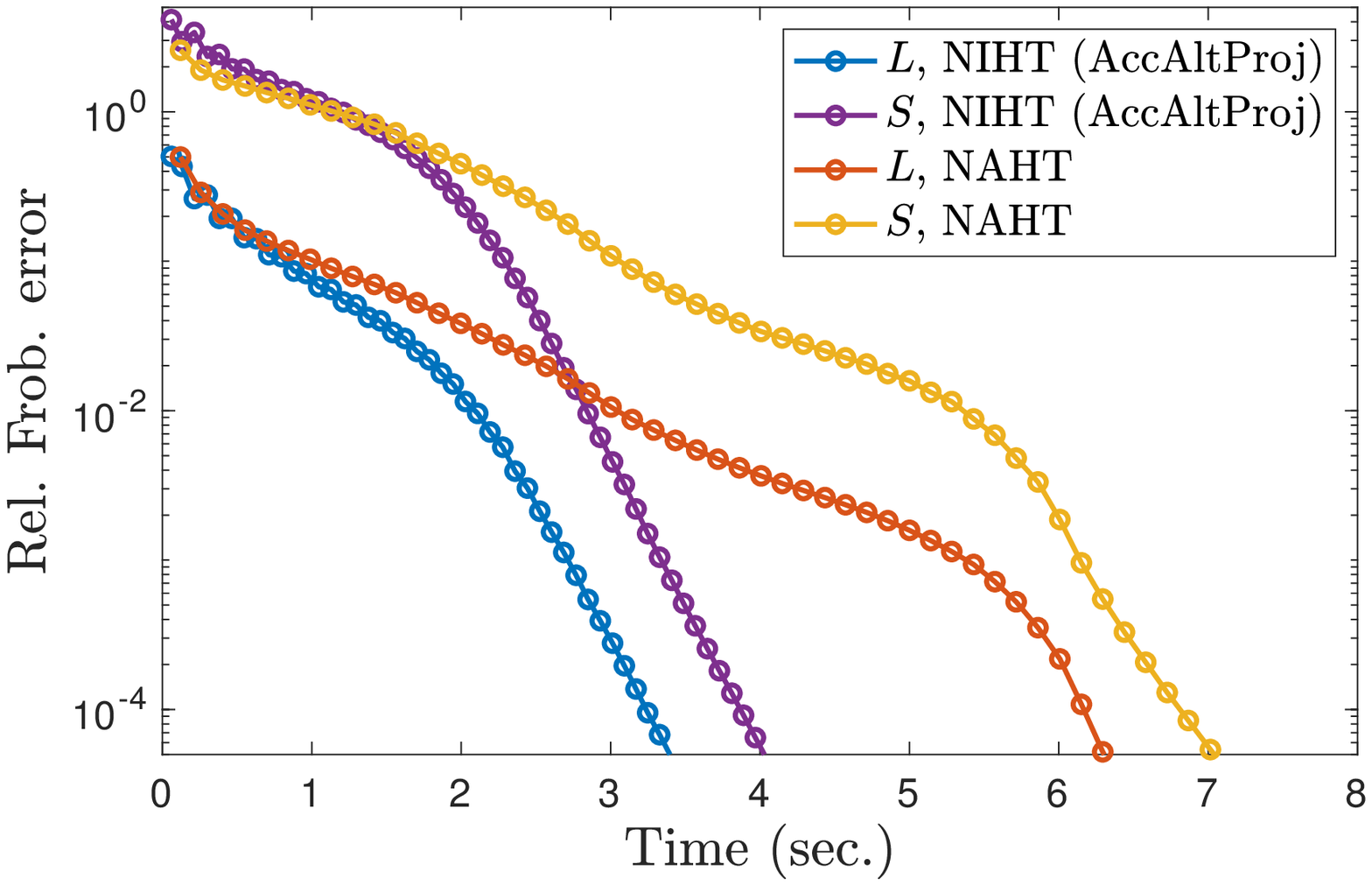}
		\caption{$\delta = 0.5 ,\rho_r = 0.05, \rho_s = 0.15$}
	\end{subfigure}
         \begin{subfigure}[b]{\figwidthH\textwidth}
	 	\includegraphics[width=1\textwidth]{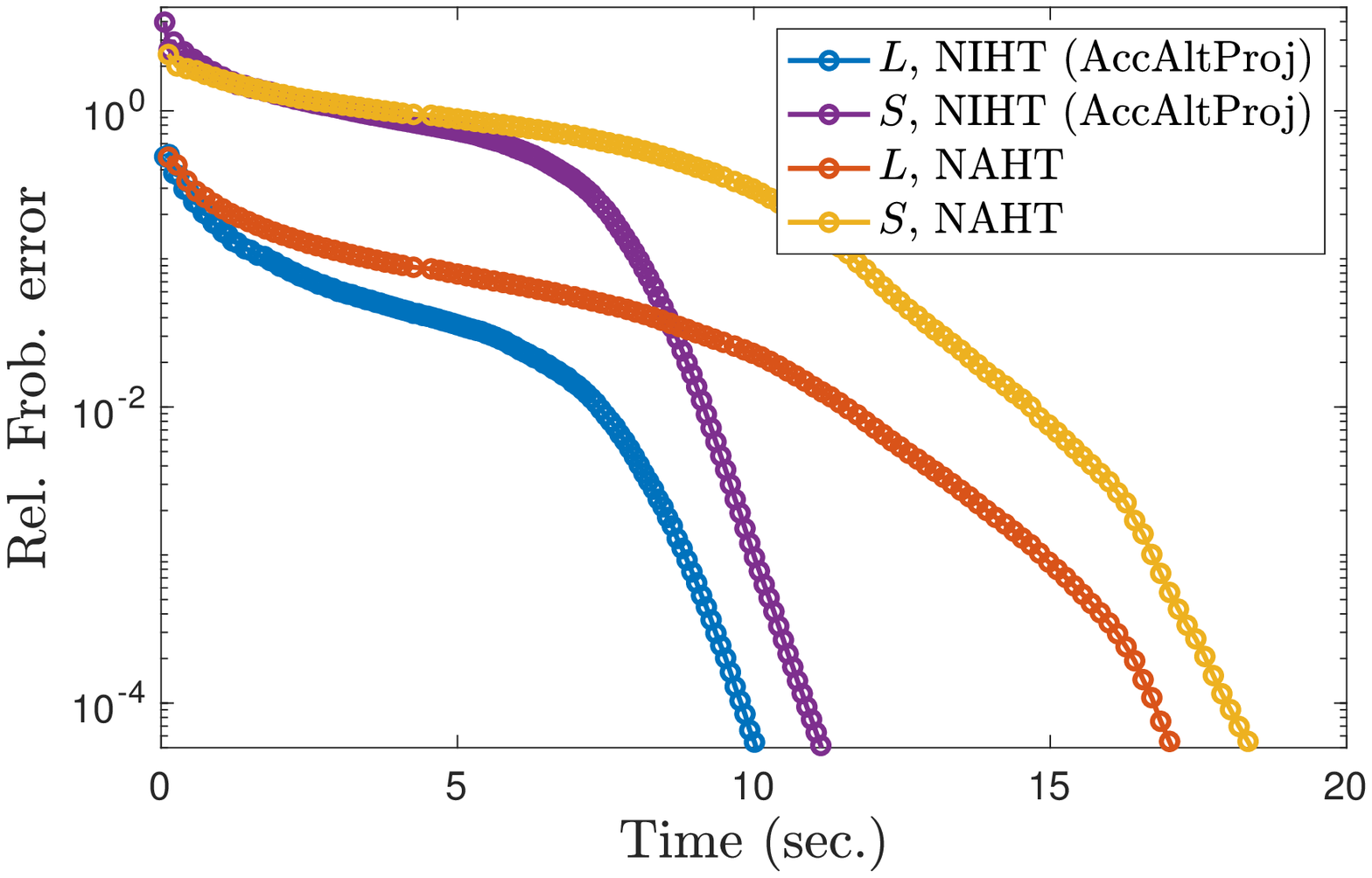}
		\caption{$\delta = 0.5 ,\rho_r = 0.05, \rho_s = 0.25$}
	\end{subfigure}
        	\caption{Error between between the approximate
                  recovered low-rank and sparse components $L^{\ell}$
                  and $S^{\ell}$ and the true low-rank and sparse
                  components $L_0$ and $S_0$.  Error is plotted as a
                  function of recovery time for synthetic problems with
                  $m=n=100$ and $p=(1/2)100^2$, $\delta=1/2$, for
                  Gaussian linear measurements $\A$.  \label{fig:convergence2}}
\end{figure}

\subsection{Applications}\label{sec:applications}

\subsubsection{Dynamic-foreground/static-background video separation}

Background/foreground separation is the task of distinguishing moving
objects from the static-background in a time series, e.g. a video
recording. A widely used approach is to arrange frames of the video
sequence into an $m\times n$ matrix, where $m$ is the number of pixels
and $n$ is the number of frames of the recording and apply Robust PCA
to decompose the matrix into the sum of a low-rank and a sparse
component which model the static background and dynamic foreground
respectively~\cite{Bouwmans2016decomposition}. Herein we consider the
same problem but with the additional challenge of recovering the video
sequence from subsampled information~\cite{Waters2011sparcs} analogous
to compressed sensing.  
\newcommand\figwidthO{0.85} 
\def\LW{\dimexpr.15\linewidth-.5em}
\begin{figure}[!t]
	\centering
	  \begin{subfigure}[b]{\figwidthO\textwidth}
		\hspace{-\LW}
		\parbox{\LW}{
			\subcaption{$X_\mathrm{rpca}$ \label{fig:video_Xrpca}}
		}\parbox{\LW}{
			\includegraphics[width=1\textwidth]{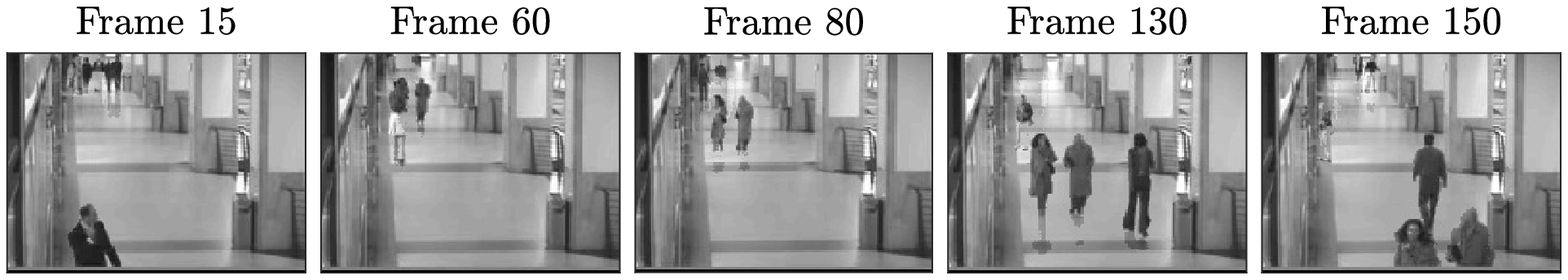}
		}\hfill%
    \end{subfigure}
       \begin{subfigure}[b]{\figwidthO\textwidth}
		\hspace{-\LW}
		\parbox{\LW}{
			\subcaption{$X_\mathrm{niht}$ \label{fig:video_Xniht}}
		}\parbox{\LW}{
			\includegraphics[width=1\textwidth]{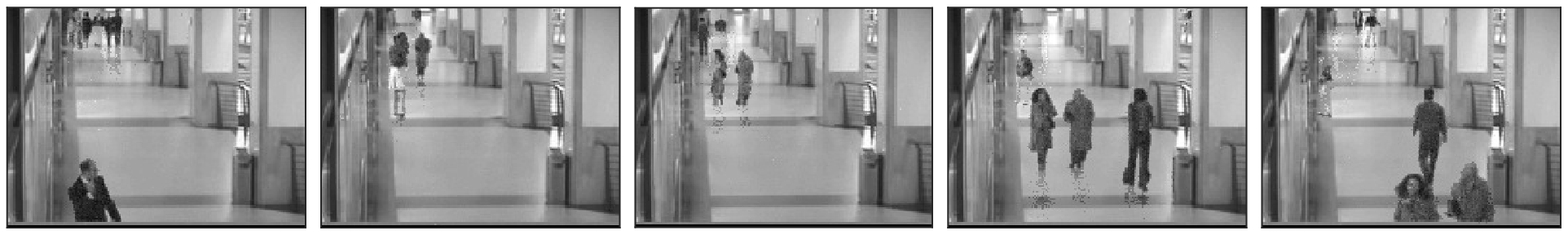}
		}\hfill%
    \end{subfigure}
     \begin{subfigure}[b]{\figwidthO\textwidth}
		\hspace{-\LW}
		\parbox{\LW}{
			\subcaption{$S_\mathrm{rpca}$ \label{fig:video_Srpca}}
		}\parbox{\LW}{
			\includegraphics[width=1\textwidth]{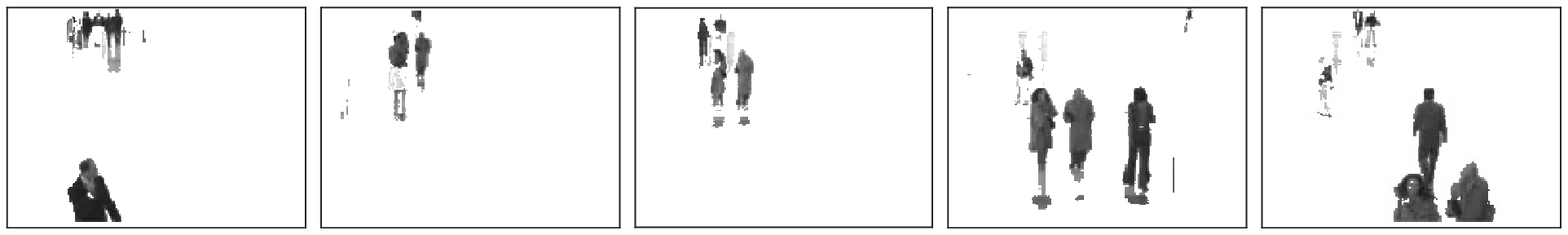}
		}\hfill%
    \end{subfigure}
    \begin{subfigure}[b]{\figwidthO\textwidth}
		\hspace{-\LW}
		\parbox{\LW}{
			\subcaption{$S_\mathrm{niht}$ \label{fig:video_Sniht}}
		}\parbox{\LW}{
			\includegraphics[width=1\textwidth]{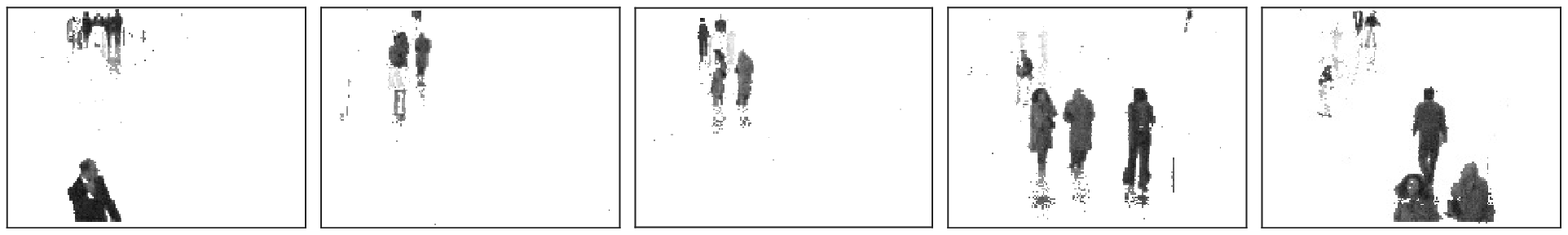}
		}\hfill%
    \end{subfigure}
    \caption{\label{fig:video}NIHT recovery results of a $256 \times
      256\times150$ video sequence compared to the approximation of
      the complete video sequence by Robust PCA
      (AccAltProj~\cite{Cai2019accelerated}). The video sequence is
      reshaped into a $26\,600 \times 150$ matrix and either recovered
      from FJLT measurements with $\delta = 0.33$ using rank $r = 1$
      and sparsity $s = 197\,505$ or approximated from the full video sequence by computing Robust PCA by AccAltProj with the same rank and sparsity parameters. Recovery by NIHT from subsampled information achieves PSNR of~$34.5~\si{\deci\bel}$ whereas the Robust PCA approximation from the full video sequence achieves PSNR of~$35.5~\si{\deci\bel}$.
    }
\end{figure}

We apply NIHT, Alg.\ \ref{algo:niht}, to the well studied shopping
mall surveillance~\cite{Li2004statistical} which is $256 \times 256
\times 150$ video sequence. The video sequence is rearranged into a
matrix of size $26\,600\times 150$ and measured using subsampled FJLT
\eqref{eq:numerics_fjlt} with one third as many meausrements as the ambient dimension,
$\delta = 0.33$.  The static-background is modeled with a rank-$r$ matrix with $r = 1$ and the dynamic-foreground by an $s$-sparse matrix with $s =
197\,505$ ($\rho_r = 0.02,\,\rho_s = 0.15$).  Figure \ref{fig:video}
displays the reconstructed image $X_{niht}$ and its sparse component
$S_{niht}$ alongside the results obtained from applying Robust PCA
(AccAltProj~\cite{Cai2019accelerated})  which makes use of the fully
sampled video sequence rather than the one-third measurements
available to NIHT.  NIHT accurately estimates the video sequence
achieving PSNR of $34.5~\si{\deci\bel}$ while also separating the
low-rank background from the sparse foreground. The results are of a
similar visual quality to the case of Robust PCA that achieves PSNR of
$35.5~\si{\deci\bel}$ which requires access to the full video sequence.

\subsubsection{Computational multispectral imaging}
A multispectral image captures a wide range of light spectra generating a vector of spectral responses at each image pixel thus acquiring information in the form of a third order tensor. 
Low-rank model has a vital role in multispectral imaging in the form of a linear spectral mixing models that assume the spectral responses of the imaged scene are well approximated as a linear combination of spectral responses of only few core materials referred to as \emph{endmembers}~\cite{Dimitris2003hyperspectral}. 
As such, the low-rank structure can be exploited by computational imaging systems which acquire the image in a compressed from and use computational methods to recover a high-resolutional image~\cite{Cao2016computational,Degraux2015generalized, Antonucci2019multispectral}.
However, when different materials are in close proximity the resulting spectrum can be highly nonlinear combination of the~\emph{endmembers} resulting in anomalies of the model~\cite{Stein2002anomaly}. Herein we propose the low-rank plus sparse matrix recovery as a way to model the spectral anomalies in the low-rank structure. 

\begin{figure}[t!]
        \centering
         \begin{subfigure}[b]{\figwidthT\textwidth}
 		\includegraphics[width=1\textwidth]{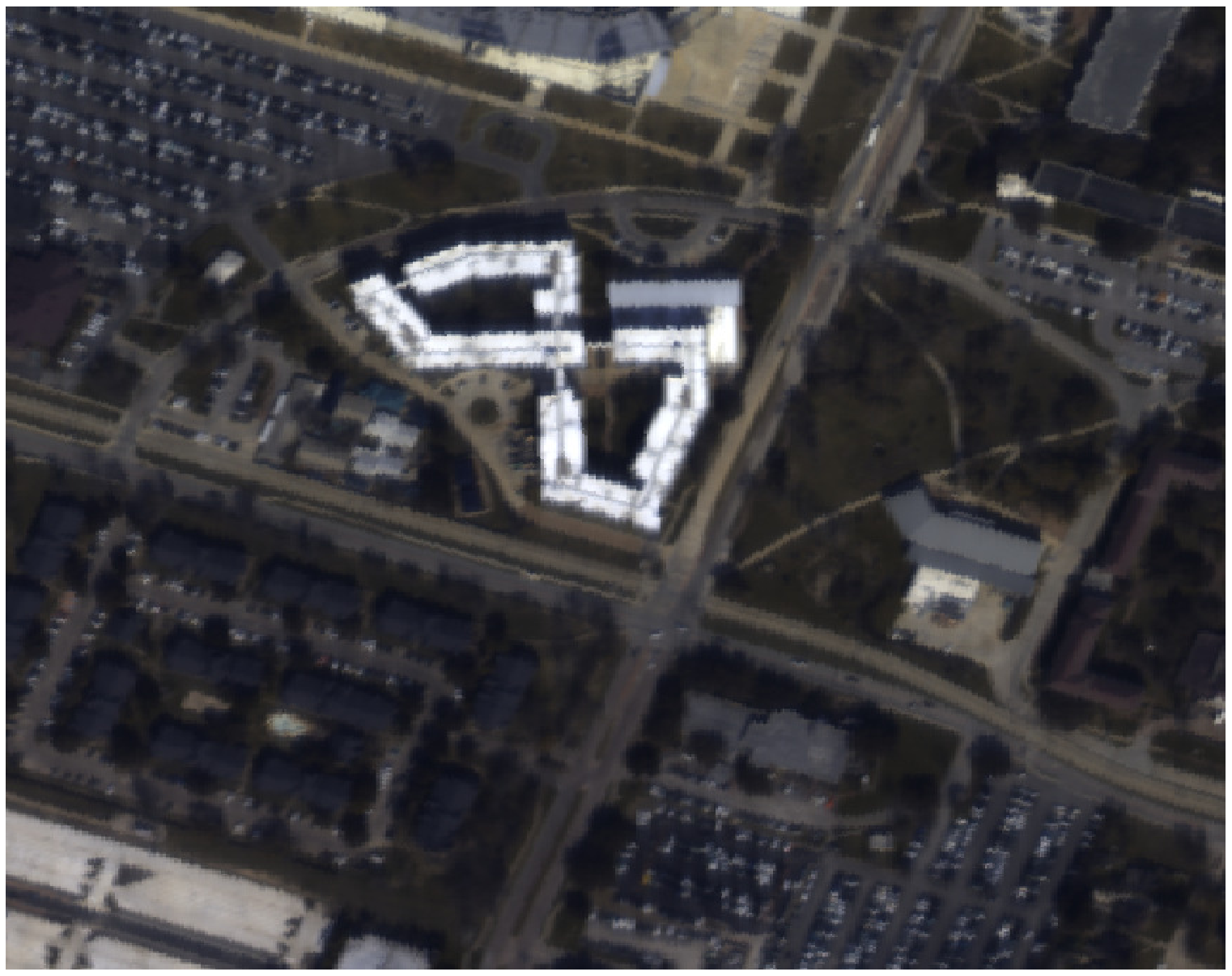}
		\subcaption{Groundtruth $\mathcal{X_\mathrm{true}}$\label{fig:hsi_truth}}
	\end{subfigure}
	\begin{subfigure}[b]{\figwidthT\textwidth}
		 \includegraphics[width=1\textwidth]{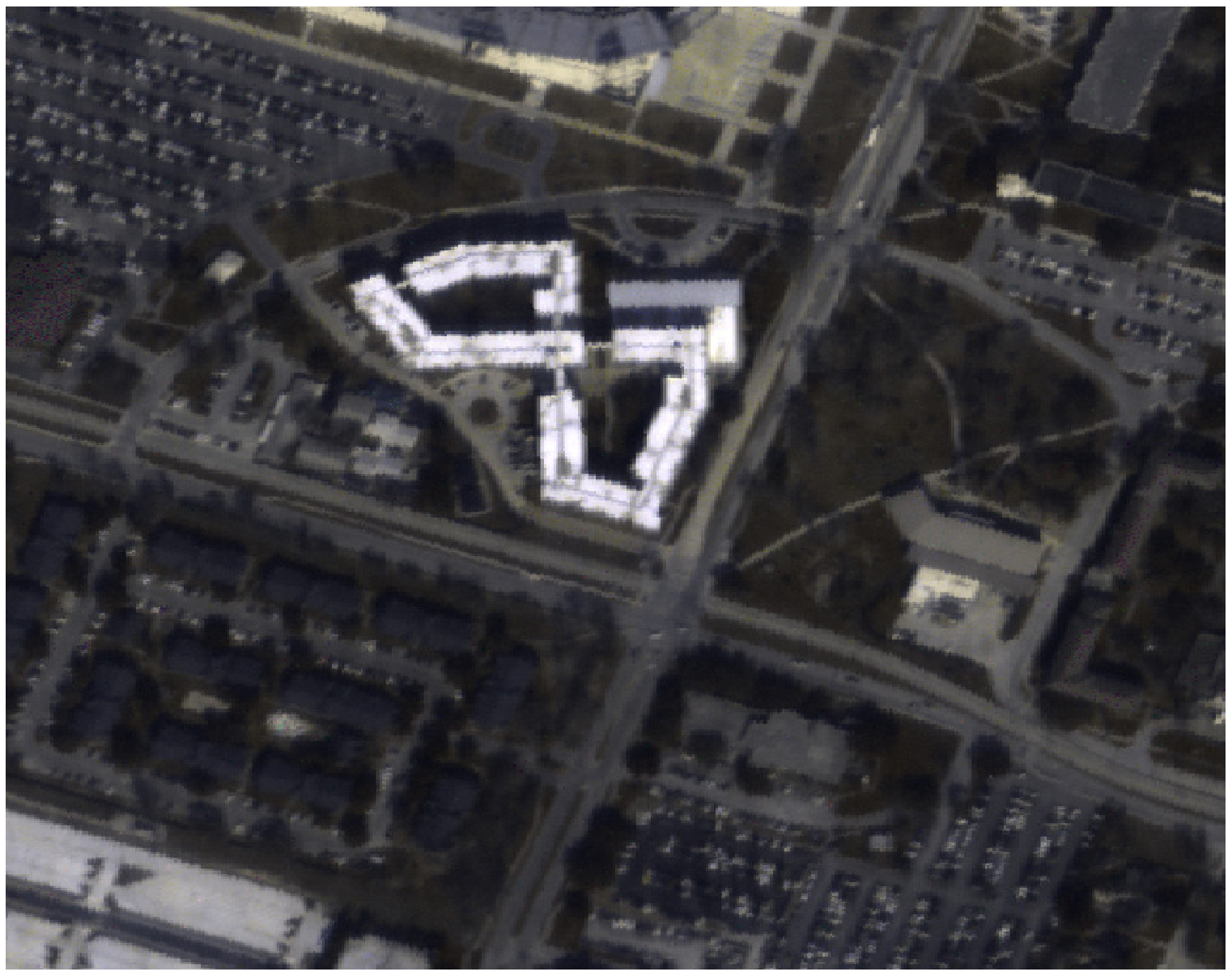}
		\subcaption{Low-rank plus sparse $\mathcal{X_\mathrm{niht}}$\label{fig:hsi_niht}}
	\end{subfigure}
	 \begin{subfigure}[b]{\figwidthT\textwidth}
	 	 \includegraphics[width=1\textwidth]{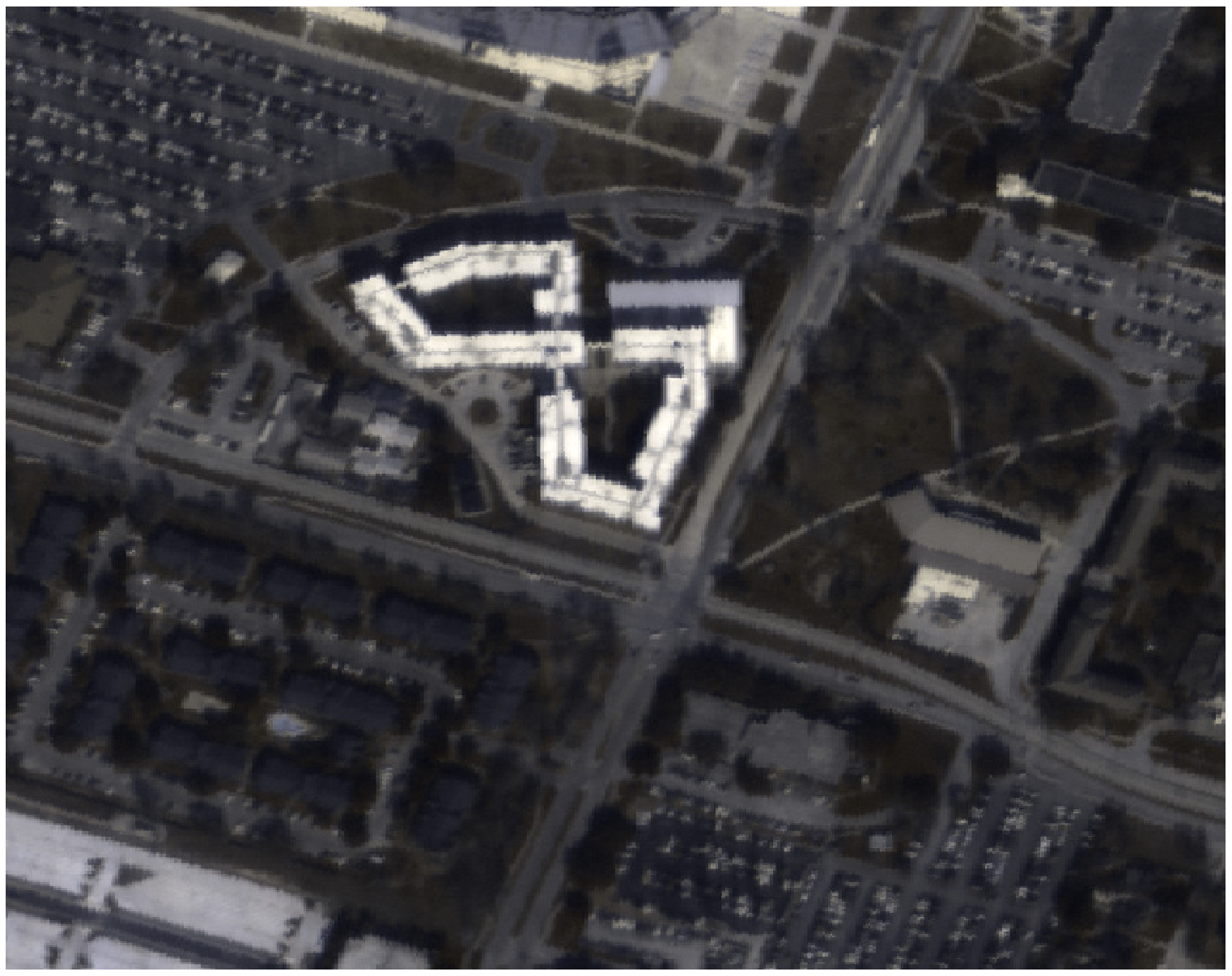}
		\subcaption{Low-rank $\mathcal{X_\mathrm{mc}}$\label{fig:hsi_mc}}
	\end{subfigure}
	\begin{subfigure}[h]{0.32\textwidth}
		\centering
		\begin{subfigure}[t]{1\textwidth}
			\centering
	 		\includegraphics[width=1\textwidth]{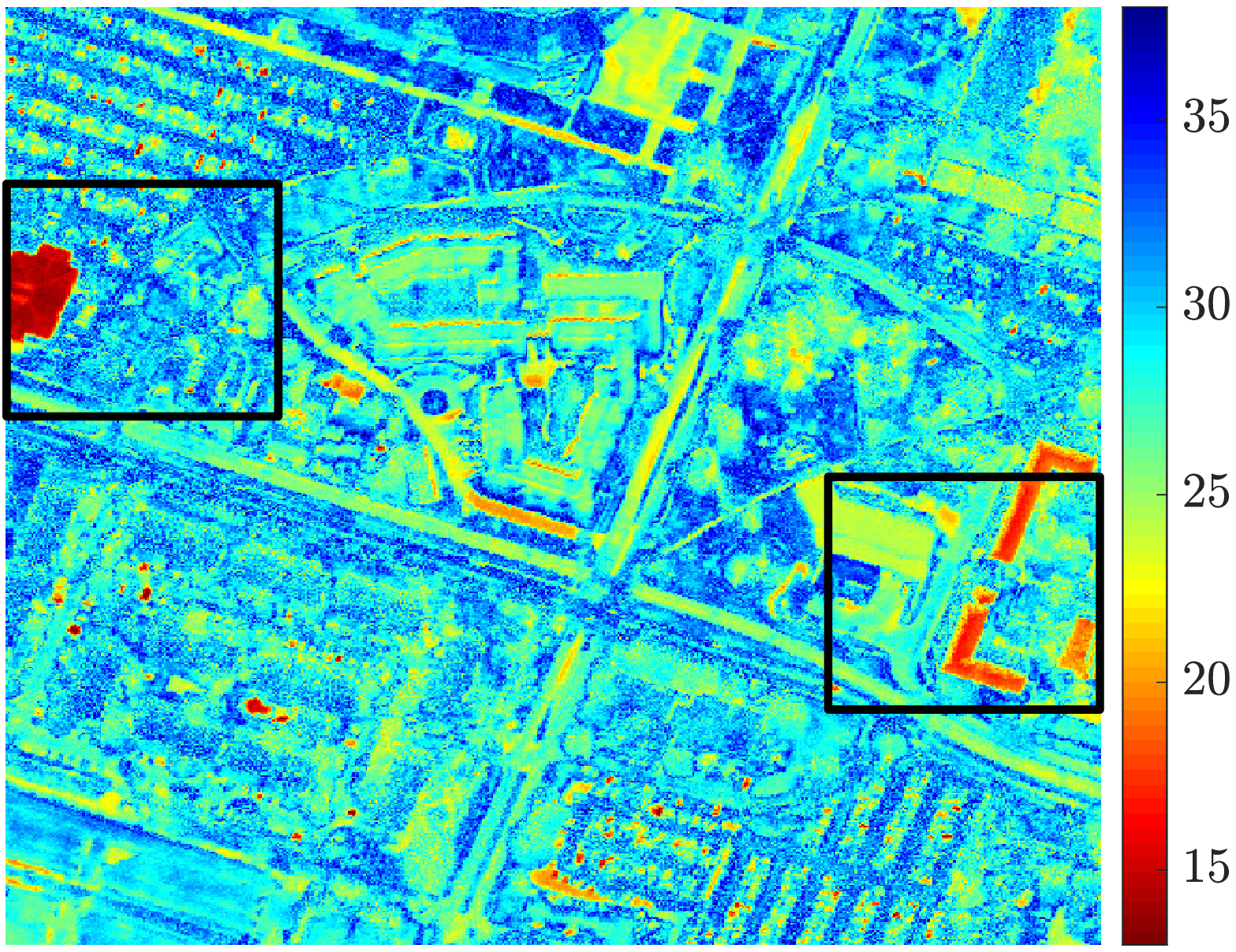}
			\subcaption{PSNR (low-rank)\label{fig:hsi_mc_psnr}}
		\end{subfigure}
		\begin{subfigure}[t]{1\textwidth}
			\centering
	 		\includegraphics[width=1\textwidth]{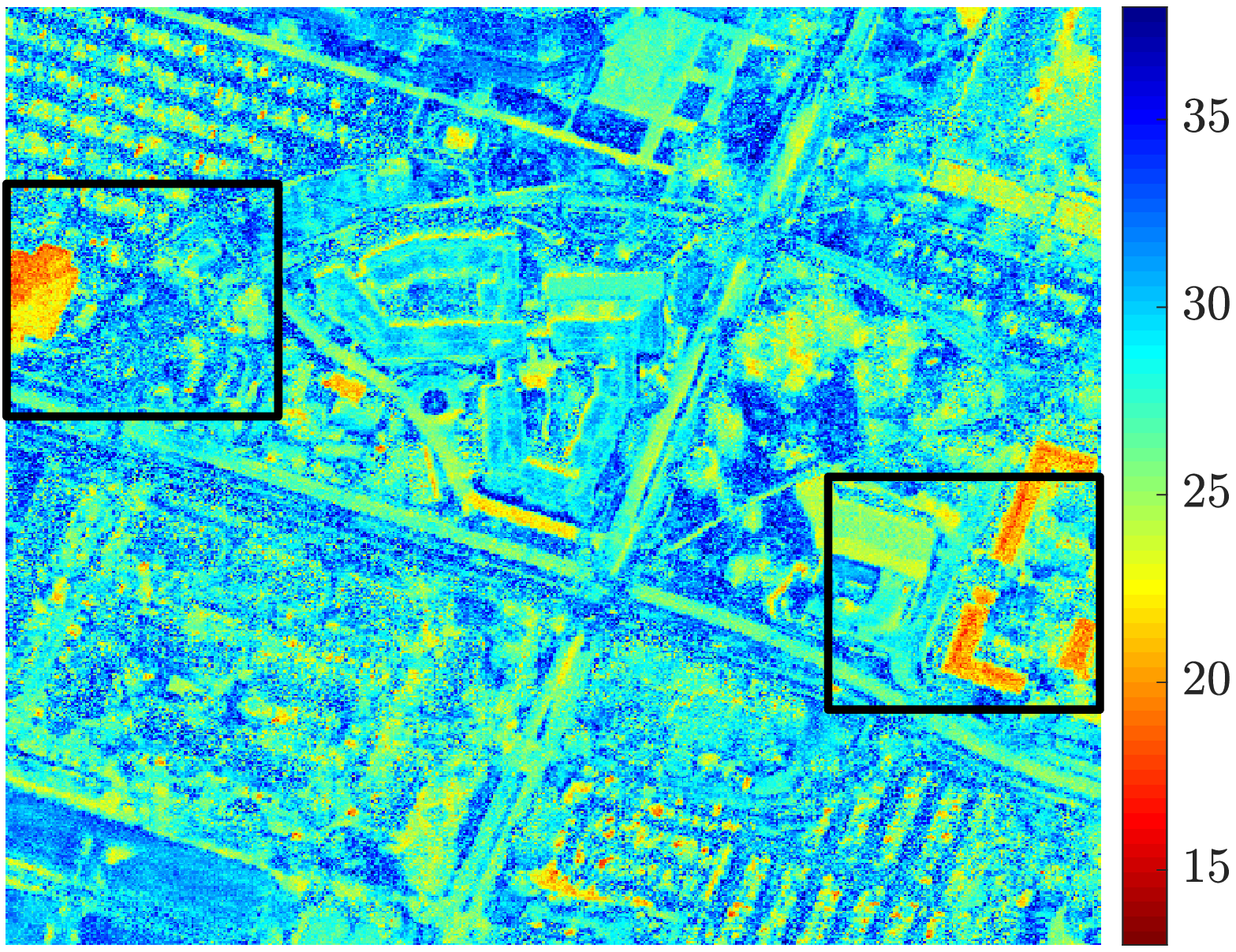}
			\subcaption{PSNR (low-rank plus sparse)\label{fig:hsi_niht_psnr}}
		\end{subfigure}
	\end{subfigure}
	\hspace{3em}
	\begin{subfigure}[h]{0.225\textwidth}
		\centering
	 	\hspace{-9\LW}
		\parbox{\LW}{
			$\mathcal{X}_\mathrm{true}$
		}\hspace{1.2em}
		\parbox{\LW}{
			\includegraphics[width=1\textwidth]{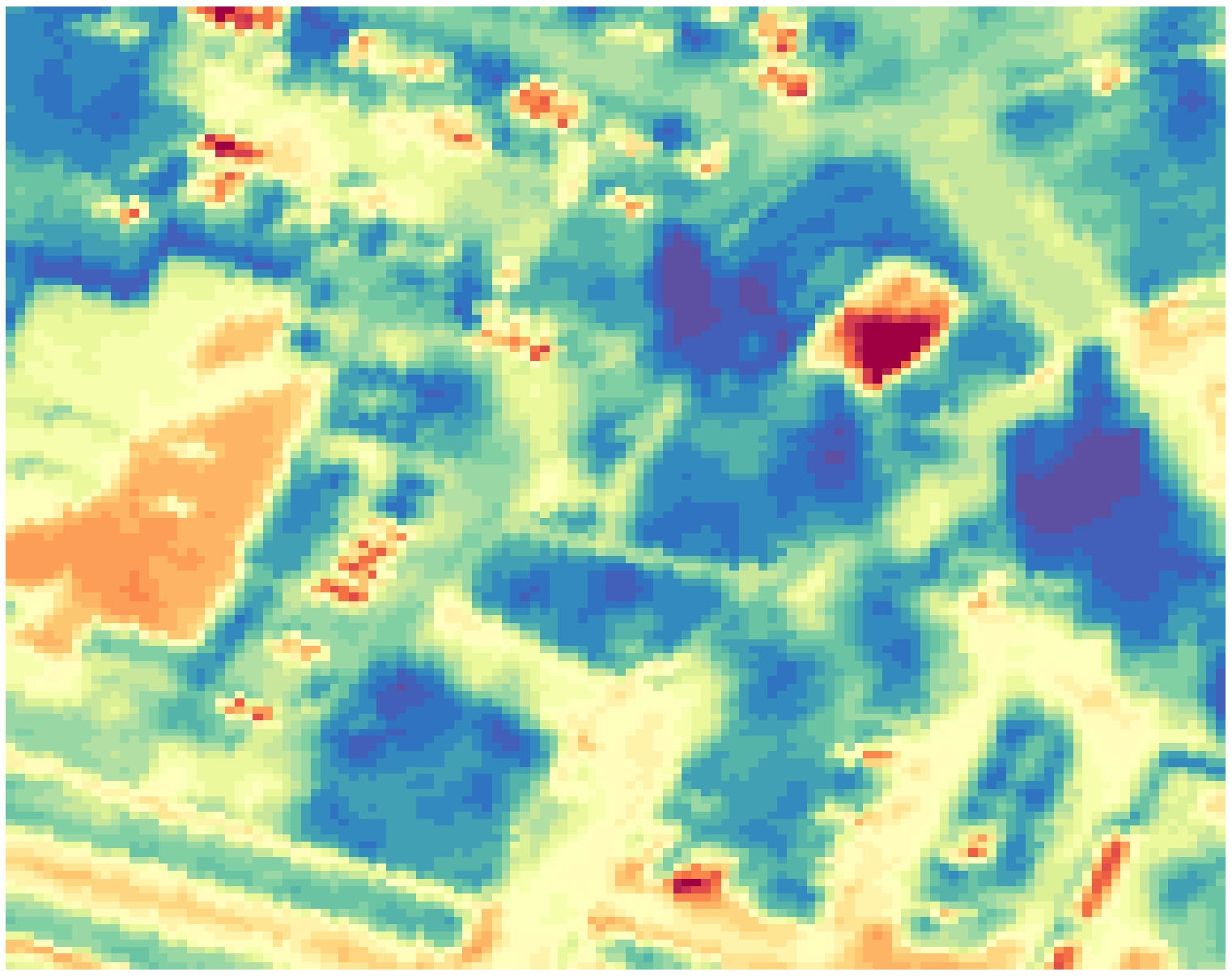}
		}\hfill \\
		\hspace{-9\LW}
		\parbox{\LW}{
			$\mathcal{X}_\mathrm{niht}$
		}\hspace{1.2em}
		\parbox{\LW}{
			\includegraphics[width=1\textwidth]{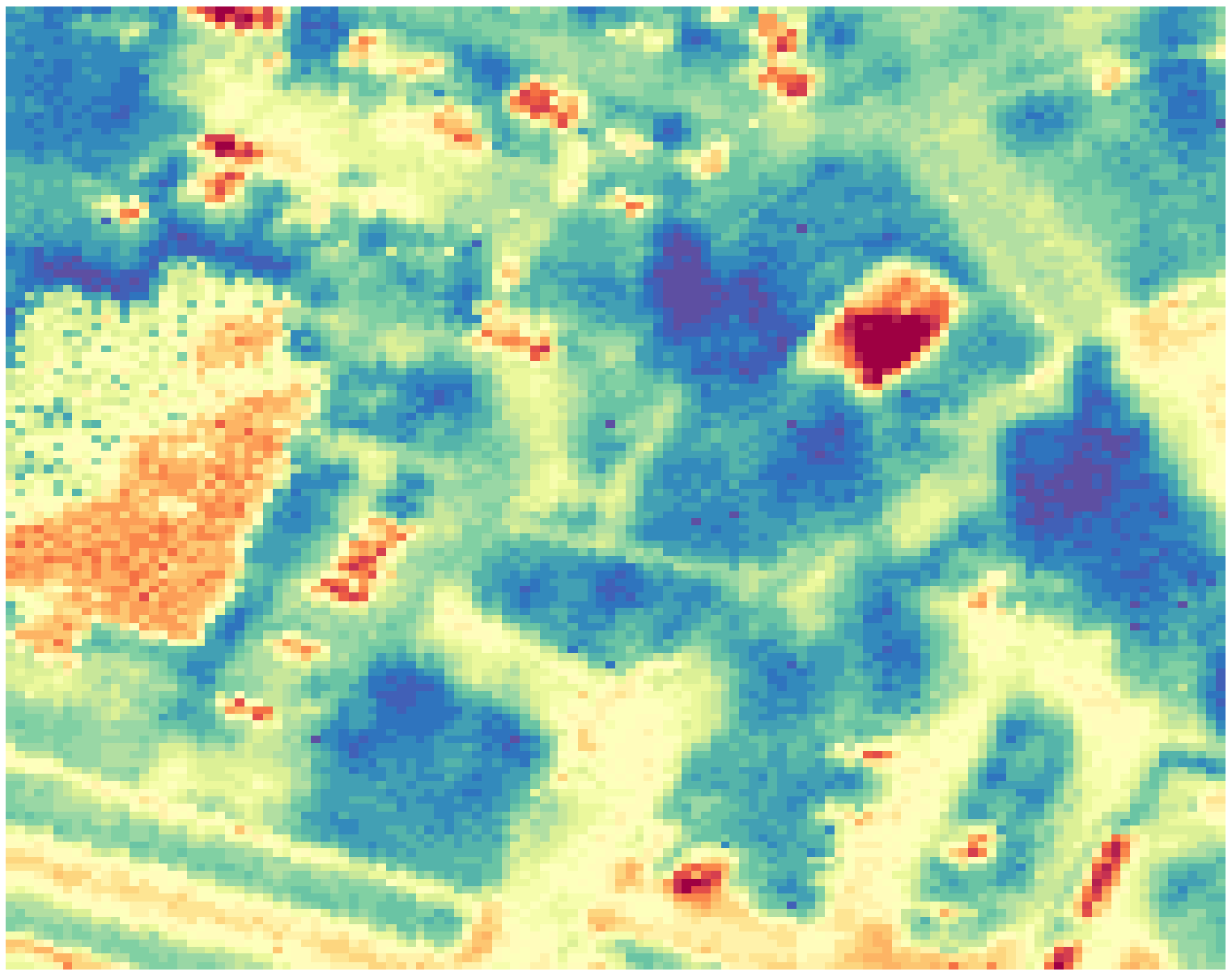}
		}\hfill \\
		\hspace{-9\LW}
		\parbox{\LW}{
			$\mathcal{X}_\mathrm{mc}$
		}\hspace{1.2em}
		\parbox{\LW}{
			\includegraphics[width=1\textwidth]{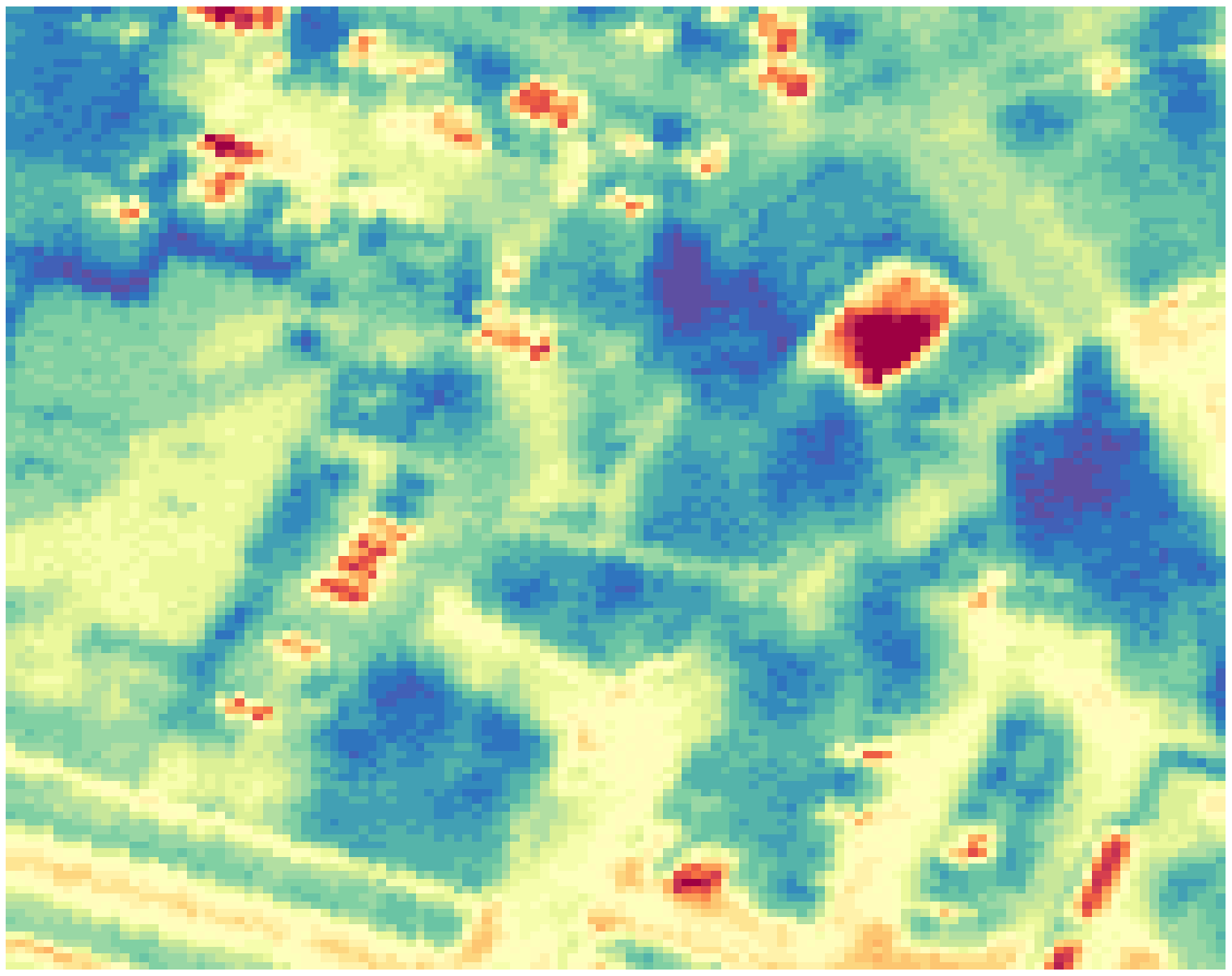}
		}\hfill \\
		\subcaption{Detail 1 ($694~\si{\nano\meter}$)\label{fig:hsi_detail1}}
	\end{subfigure}
	\hspace{1em}
	\begin{subfigure}[h]{0.225\textwidth}
		\centering
	 	\includegraphics[width=1\textwidth]{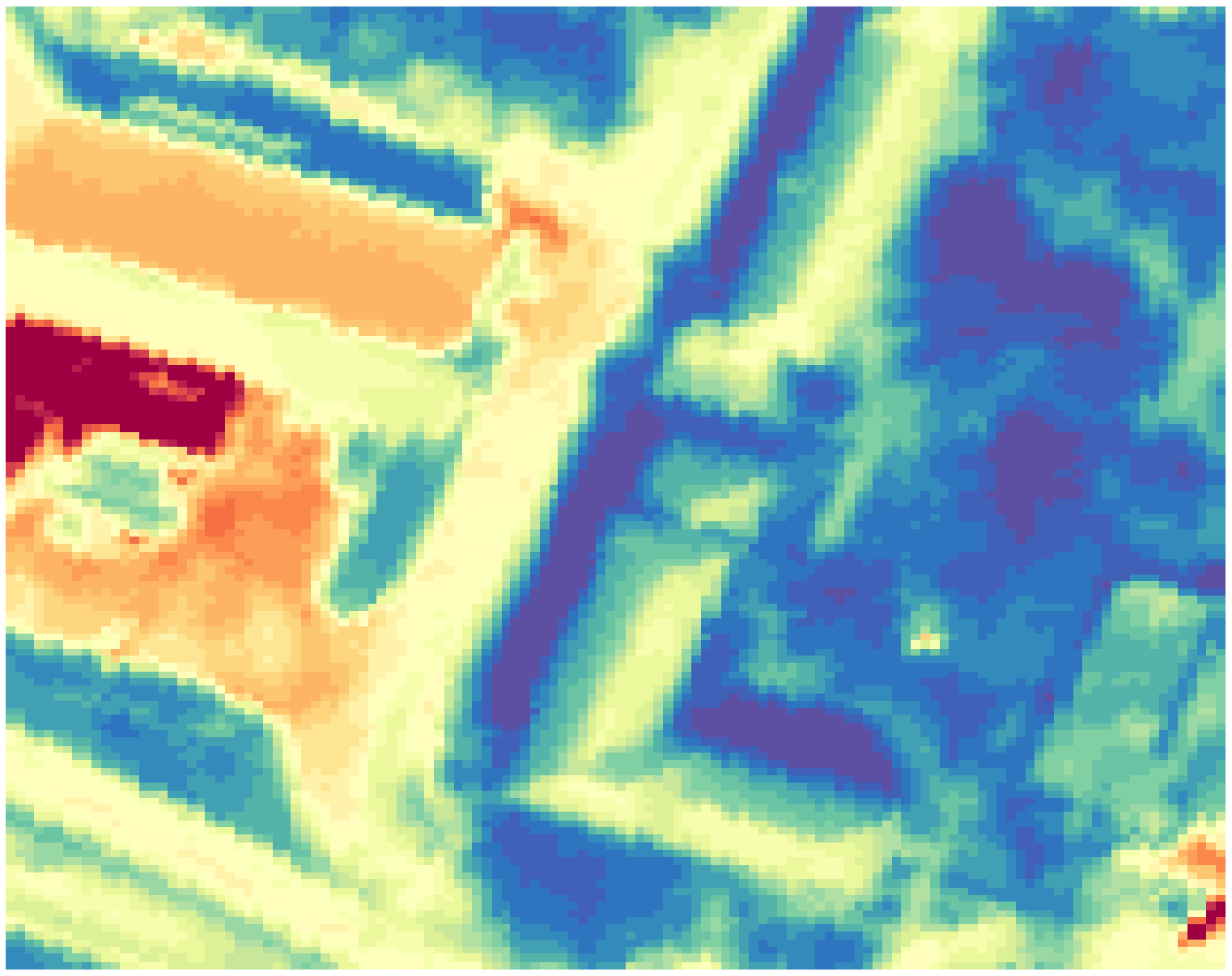}
		\includegraphics[width=1\textwidth]{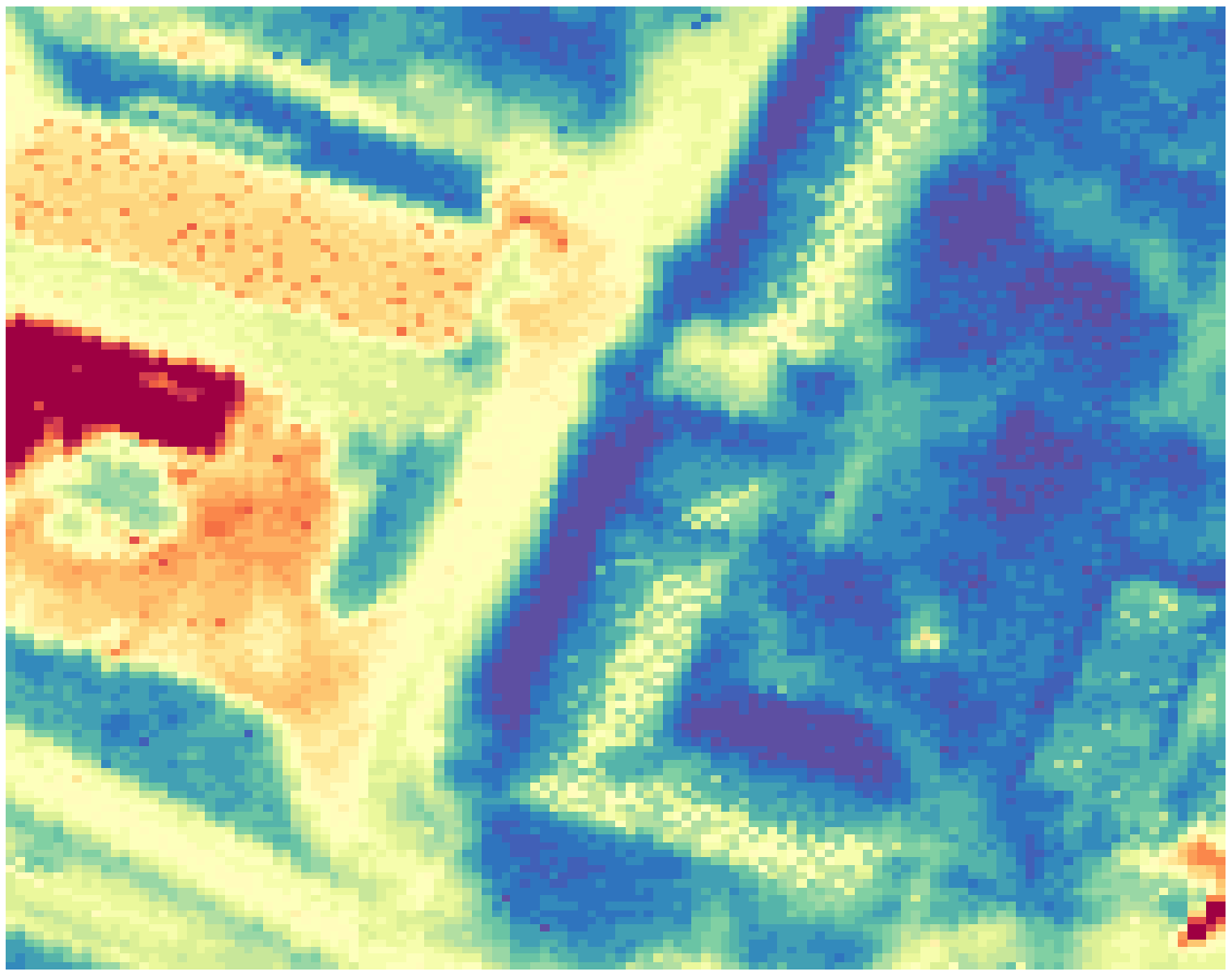}
		\includegraphics[width=1\textwidth]{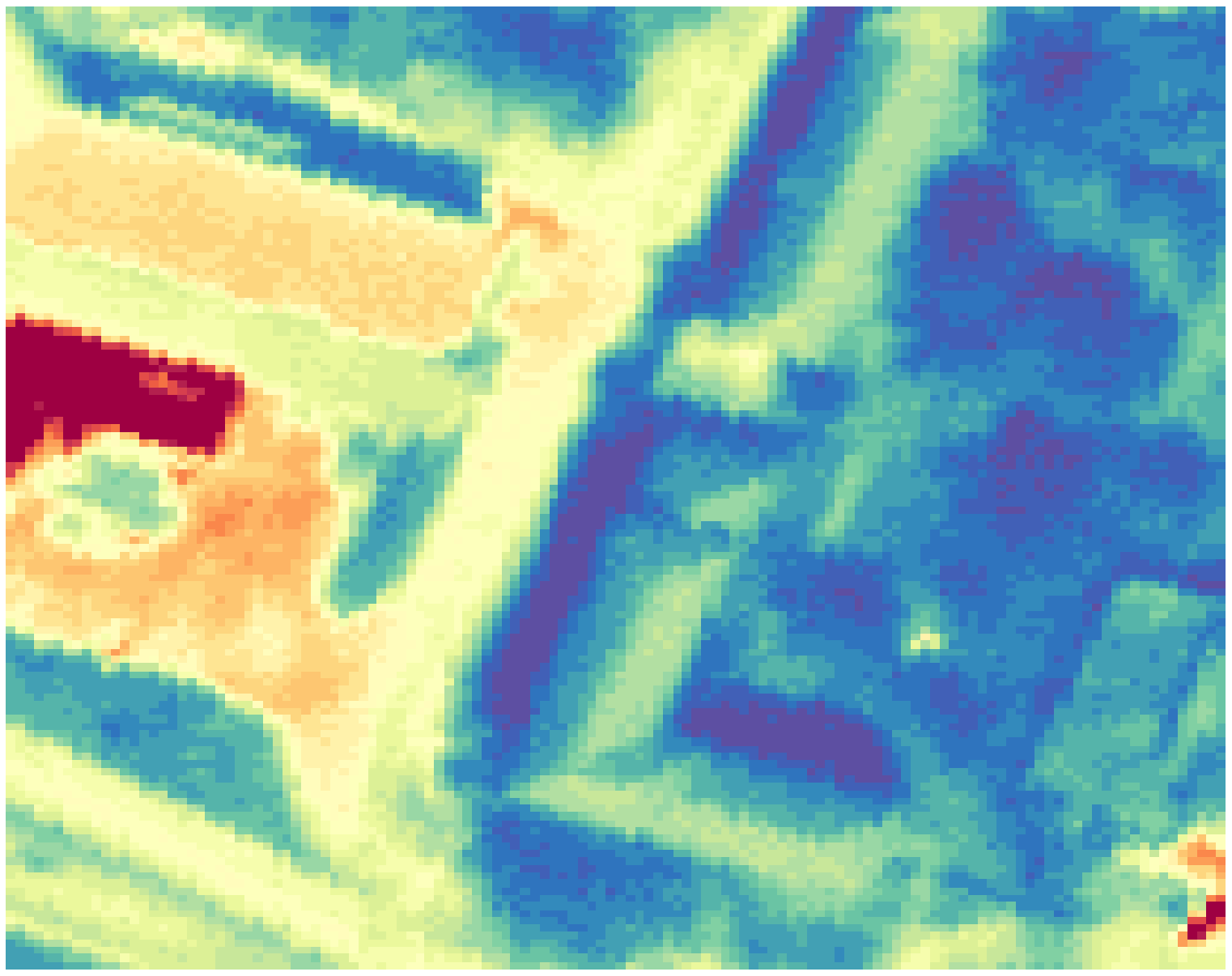}
		\subcaption{Detail 2 ($694~\si{\nano\meter}$)\label{fig:hsi_detail2}}
	\end{subfigure}
        	\caption{Recovery by NIHT from FJLT measurements with $\delta = 0.33$ using low-rank model ($\rho_r = 0.25,\, \rho_s = 0$) compared to the low-rank plus sparse model ($\rho_r = 0.25,\, \rho_s = 0.05$). Figure~\ref{fig:hsi_truth} - \ref{fig:hsi_niht} show the color renderings of the original multispectral image and the two recovered images. Figure~\ref{fig:hsi_mc_psnr} and Figure~\ref{fig:hsi_niht_psnr} show the spatial PSNR of the recovery from the low-rank only model (overall PSNR of $38.9~\si{\deci\bel}$) and the low-rank plus sparse model (overall PSNR of $39.1~\si{\deci\bel}$) respectively. Figure~\ref{fig:hsi_detail1} and Figure~\ref{fig:hsi_detail2} show two details of size $128\times128$ in the $694~\si{\nano\meter}$ band.\label{fig:hsi}}
\end{figure}

We employ NIHT on a $512 \times 512 \times 48$ airborne hyperspectral image from the GRSS 2018 Data Fusion contest~\cite{Xu2019advanced} that is rearranged into a matrix of size $262\,144 \times 48$ and subsampled using FJLT with $\delta =0.33$.  Figure~\ref{fig:hsi} demonstrates recovery by NIHT using rank $r = 3$ and sparsity $s = 150\,995$ ($\rho_r = 0.25,\,\rho_s = 0.05$) in comparison with the the low-rank model with rank $r = 3$ and $s = 0$ ($\rho_r = 0.25,\,\rho_s = 0$). Both methods recover the image well but the low-rank plus sparse recovery achieves slightly higher PSNR of $39.1~\si{\deci\bel}$ compared to the low-rank recovery that has PSNR of $38.9~\si{\deci\bel}$ and slightly better fine details. Figure~\ref{fig:hsi_mc_psnr} and Figure~\ref{fig:hsi_niht_psnr} depict the localization of the error in terms of PSNR and shows that adding the sparse component improves PSNR of a few localized parts. Although the overall gain in the  PSNR is small compared to the low-rank model, the differences in the localized regions of the image can be potentially impactful when further analyzed in practical applications such as semantic segmentation~\cite{Kemker2018algorithms}.

\section{Conclusion}
The main theorems, Theorems \ref{thm:rip_ls}, \ref{thm:thmexistunique}, \ref{thm:convex_recovery}, \ref{thm:niht_convergence}, and \ref{thm:naht_convergence}, are the natural extension of analogous results in the compressed sensing and matrix completion literature to the space of low-rank plus sparse matrices, Definition~\ref{def:ls}, see \cite{Eldar2012compressed,Foucart2013a} and references therein.  They establish the foundational theory and provide examples of algorithms for recovery of matrices that can be expressed as a sum of a low-rank and a sparse matrix from under sampled measurements.   While these results could be anticipated, with \cite{Waters2011sparcs} being an early non-convex algorithm for this setting, these advancements had not yet been proven.
We prove that the restricted isometry constants of random linear operators obeying concentration of measure inequalities, such as Gaussian measurements or the Fast Johnson-Lindenstrauss Transform, can be upper bounded when the number of measurements are of the order depending on the degrees of freedom of the low-rank plus sparse matrix.  Making use of these RICs, we show that low-rank plus sparse matrices can be provably recovered by computationally efficient methods, e.g. by solving semidefinite programming or by two gradient descent algorithms, when the restricted isometry constants of the measurement operator are sufficiently bounded. \Rev{These results also provably solve Robust PCA with the asymptotically optimal number of corruptions and improve the previously known guarantees by not requiring an assumption on the support of the sparse matrix.} Numerical experiments on synthetic data empirically demonstrate phase transitions in the parameter space for which the recovery is possible.  Experiments for dynamic-foreground/static-background video separation show that the segmentation of moving objects can be obtains with similar error from only one third as many measurement as compared to the entire video sequence.  The contributions here open up the possibility of other algorithms in compressed sensing and low-rank matrix completion/sensing to be extended to the case of low-rank plus sparse matrix recovery, e.g. more efficient algorithms such as those employing momentum \cite{Kyrillidis2014matrix,Wei2015fast} or minimising over increasingly larger subspaces \cite{Blanchard2015cgiht}.  These results also illustrate how RICs can be developed for more complex additive data models, \Rev{provided it is possible to control the correlation between them,} and one can expect that similar results would be possible for new data models.

\section*{Acknowledgement}
We would like to thank Robert A.\ Lamb and David Humphreys for useful discussions around the applications of low-rank plus sparse model to multispectral imaging.
\bibliography{library.bib}

\newpage
\appendix
\section{Phase transitions for synthetic problem of size $m=n=30$}\label{app:convex}
\vspace{-.5em}
Figure~\ref{fig:phase_cvx} depicts the  phase transitions of
$\delta$ above which NIHT, NAHT and solving the convex relaxation problem in \eqref{eq:convex} successfully recovers $X_0$ in more
than half of the experiments. Comparing Fig.~\ref{fig:phase_cvx} to
Fig.~\ref{fig:phase_nonconvex} we see that the phase transitions
roughly occur for the same parameters $\rho_r, \rho_s$ with only small
differences due to the finite dimensional effects of the smaller
problem size being more pronounced when $m = n = 30$. We also observe
that non-convex algorithms perform better than the convex relaxation
in that they are able to recover higher ranks and sparsities from
fewer samples in addition to also taking less time to converge. 
\begin{figure}[h!]
        \centering
         \begin{subfigure}[b]{\figwidthH\textwidth}
 		\includegraphics[width=.9\textwidth,height=.195\textheight]{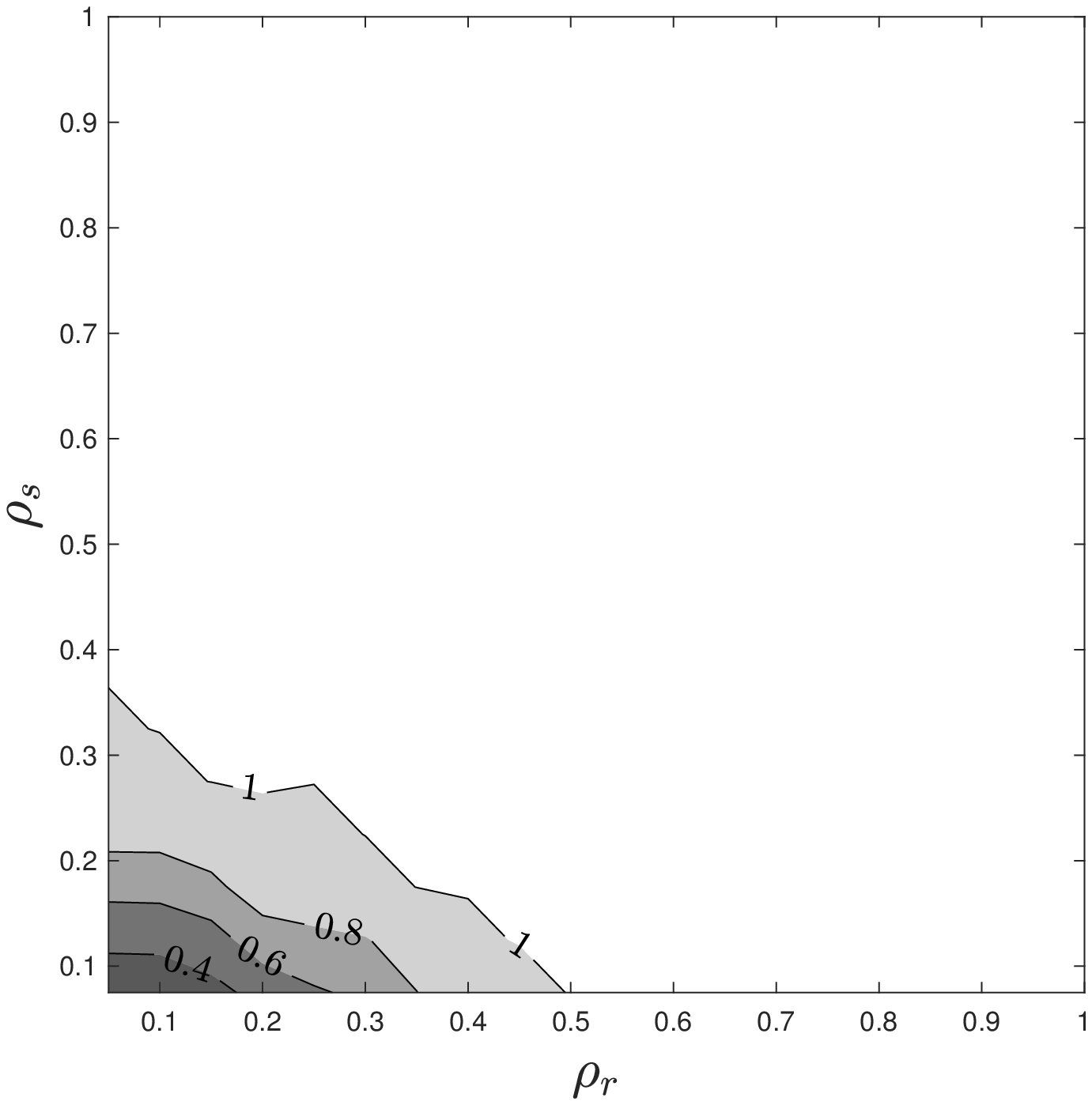}
		\caption{Convex relaxation (Gaussian measurements)}
	\end{subfigure}
	\hspace{0.03\textwidth}
         \begin{subfigure}[b]{\figwidthH\textwidth}
	 	\includegraphics[width=.9\textwidth,height=.195\textheight]{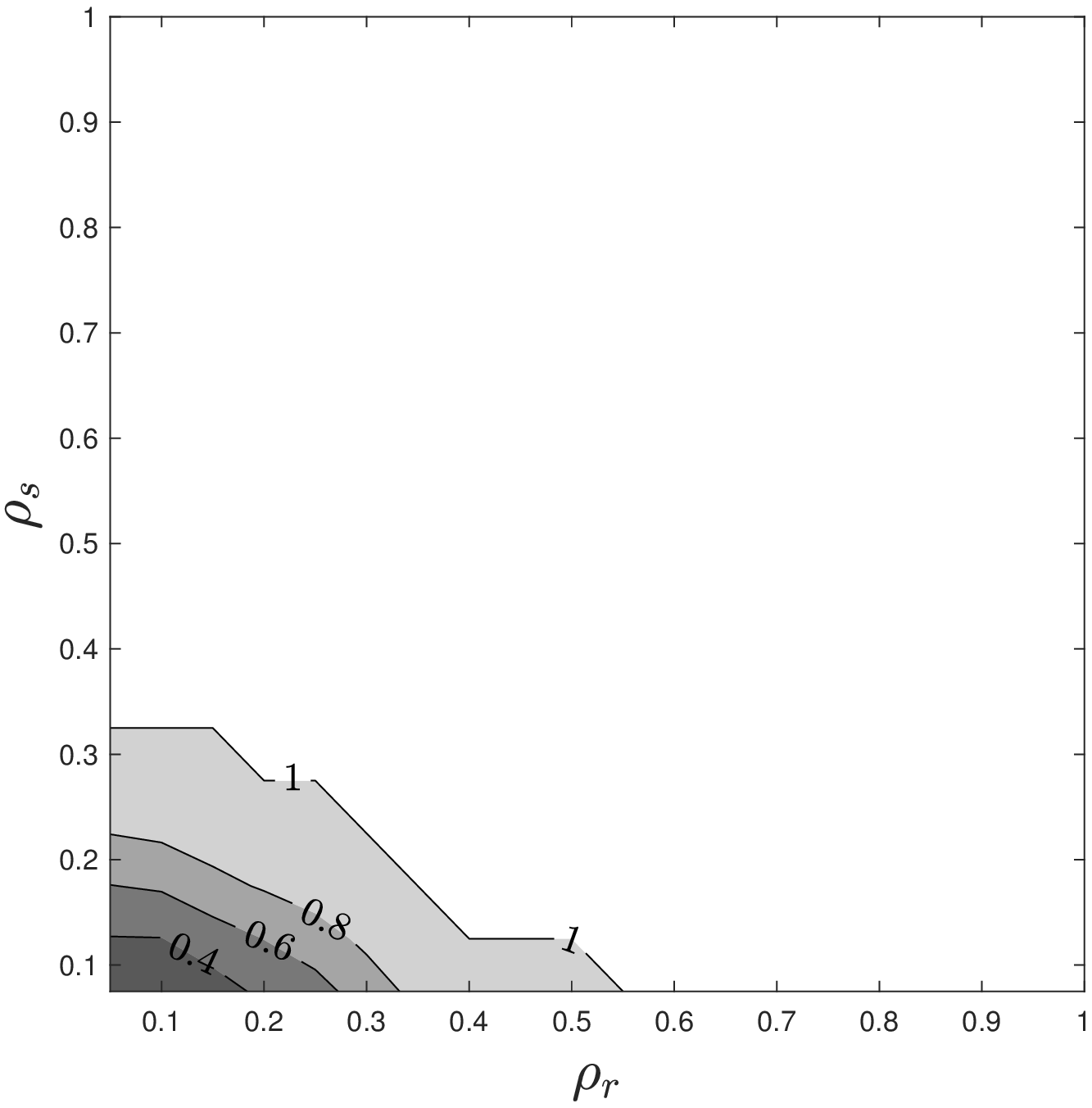}
		\caption{Convex relaxation (FJLT measurements)}
	\end{subfigure}
	\begin{subfigure}[b]{\figwidthH\textwidth}
 		\includegraphics[width=.9\textwidth,height=.195\textheight]{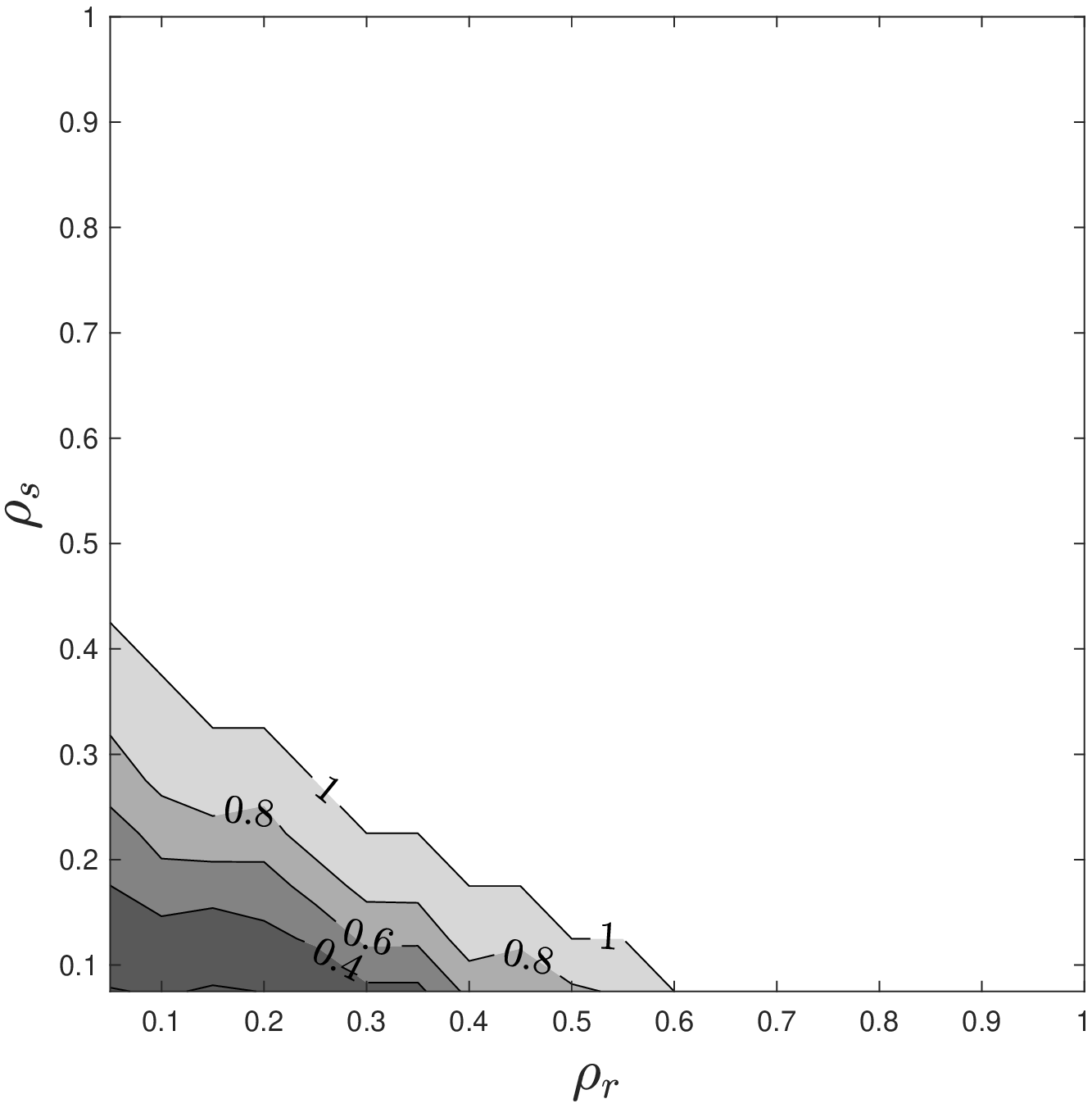}
		\caption{NIHT (Gaussian measurements)}
	\end{subfigure}
	\hspace{0.03\textwidth}
         \begin{subfigure}[b]{\figwidthH\textwidth}
	 	\includegraphics[width=.9\textwidth,height=.195\textheight]{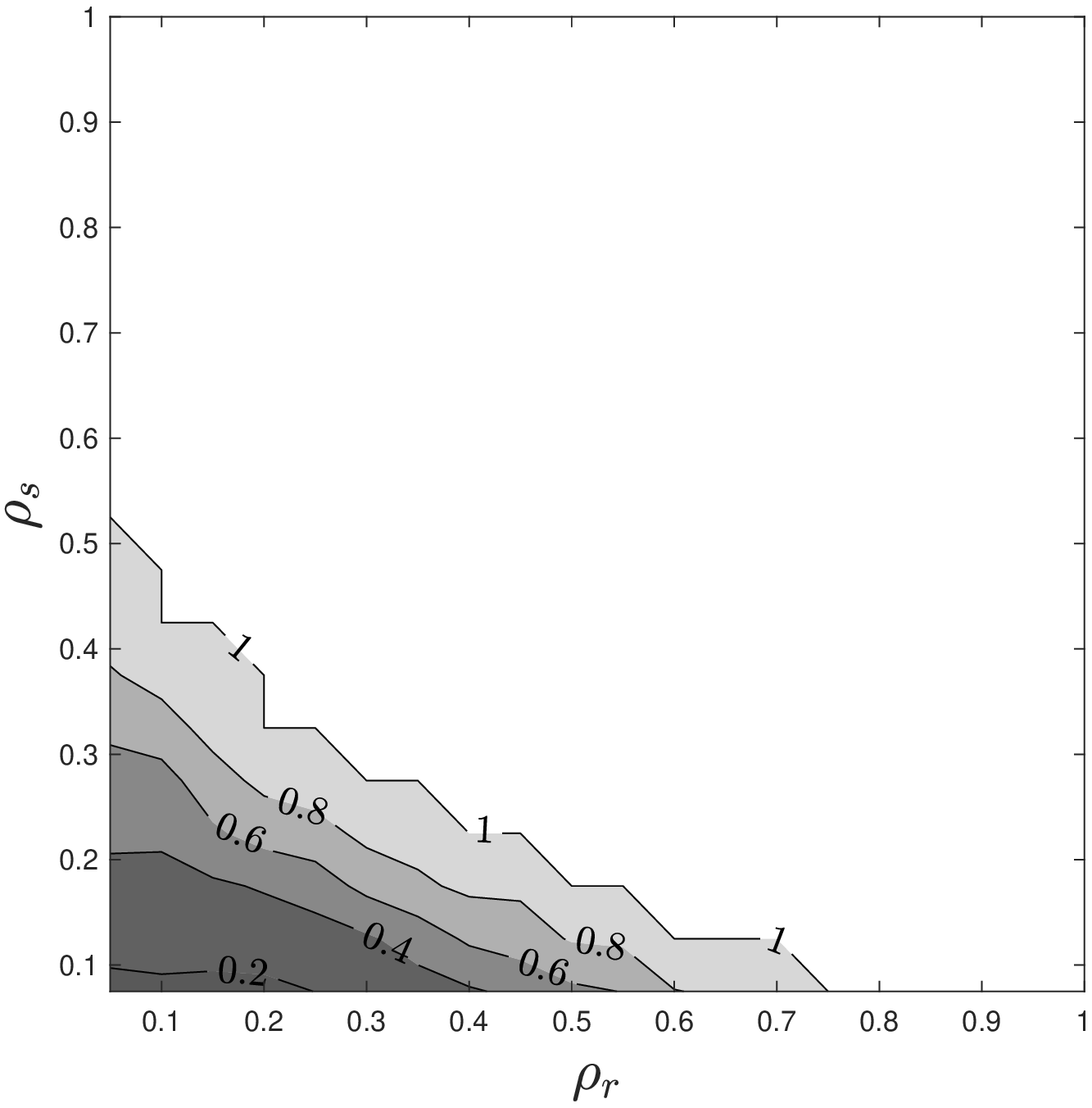}
		\caption{NIHT (FJLT measurements)}
	\end{subfigure}
	\begin{subfigure}[b]{\figwidthH\textwidth}
 		\includegraphics[width=.9\textwidth,height=.195\textheight]{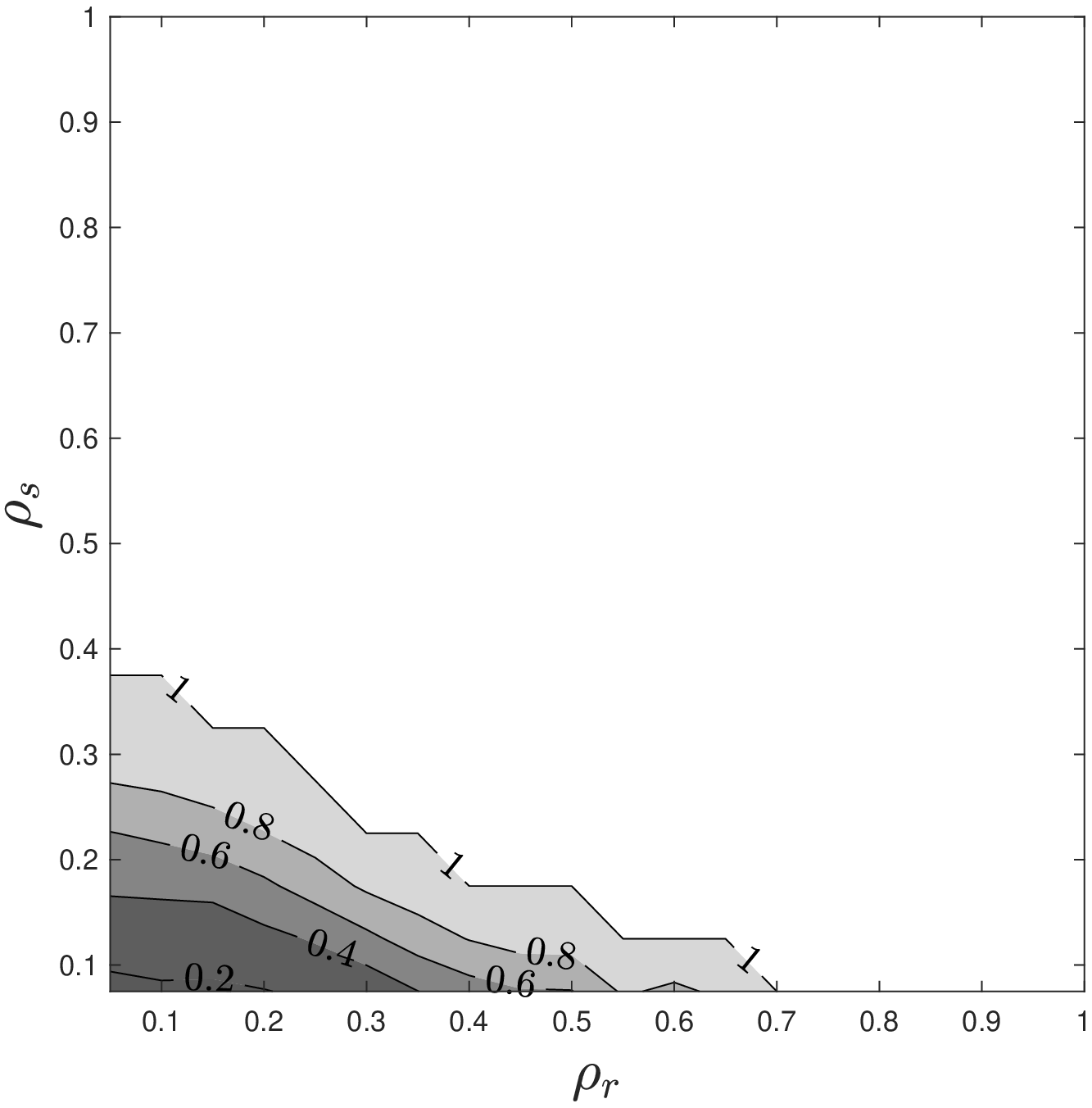}
		\caption{NAHT (Gaussian measurements)}
	\end{subfigure}
	\hspace{0.03\textwidth}
         \begin{subfigure}[b]{\figwidthH\textwidth}
	 	\includegraphics[width=.9\textwidth,height=.195\textheight]{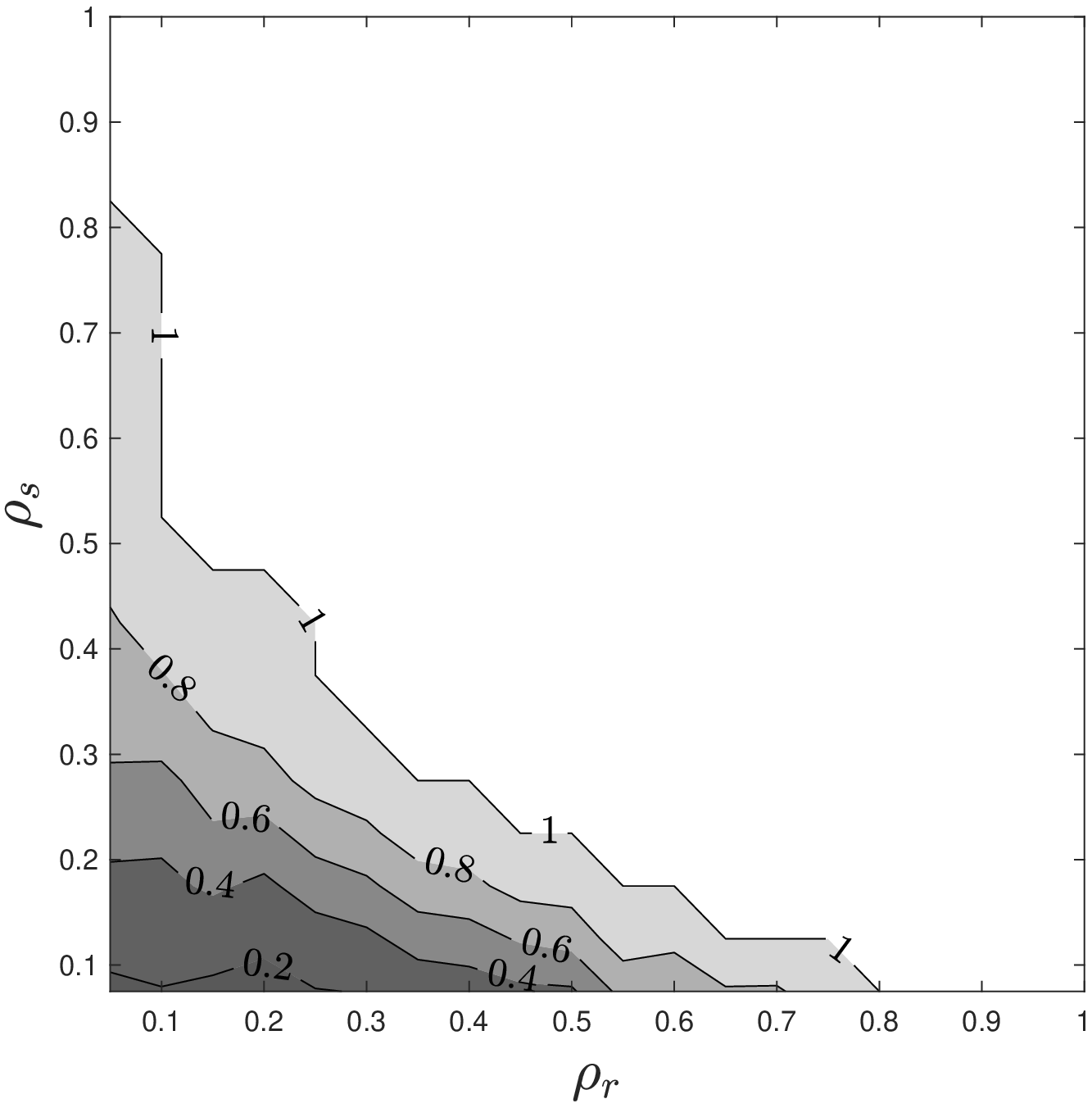}
		\caption{NAHT (FJLT measurements)}
	\end{subfigure}
	\vspace{-0.5em}
    \caption{Phase transition level curves denoting the
                  value of $\delta^*$ for which values of $\rho_r$ and
$\rho_s$ {\em below which} are recovered for at least $5$ out of $10$
 experiments for $\delta$, $\rho_r$, and $\rho_s$ as given by
\eqref{eq:delta_rho}. The convex optimization problem is solved by SDPT3 \cite{Toh1999sdpt3}. NIHT and NAHT are observed to recover matrices of higher ranks and sparsities compared to solving the convex relaxation.\label{fig:phase_cvx}}
\end{figure}


\section{Supporting lemmata \label{sec:appendix_lemma}}

\Rev{The following lemma reveals the usefulness of incoherence in controlling the correlation between \emph{incoherent low-rank} and \emph{sparse} matrices.}

\begin{restatable}[\Rev{Subadditivity of the $\LS_{m,n}(r,s,\mu)$ set}]{lemma}{lemmalsmuadditive}
	\label{lemma:ls_mu_additive}
    The sum of two incoherent low-rank plus sparse matrices $X_1, X_2 \in \LS_{m,n}(r,s,\mu)$ is also an incoherent low-rank plus sparse matrix $X_1 + X_2 \in \LS_{m,n}(2r,2s,\mu)$, and consequently
    \begin{equation*}
        \LS_{m,n}(r,s,\mu) + \LS_{m,n}(r,s,\mu) = \LS_{m,n}(2r,2s,\mu),
    \end{equation*}
	where the plus sign denotes the Minkowski sum of two sets.
\end{restatable}
\begin{proof}\label{proof:ls_mu_additive}
	Let $X_1, X_2 \in \LS_{m,n}(r,s,\mu)$ with $X_1 = L_1 + S_1$, $X_2 = L_2 + S_2$, and $U_1, U_2$ and $V_1, V_2$ being the left and the right singular vectors of $L_1$ and $L_2$ respectively. 

	Construct the sum $X = L + S$, where $L = L_1 + L_2$, $S = S_1 + S_2$, and $U, V$ are the left and right singular vectors of the newly constructed $L$. Since the column space of $U$ is a subspace of the column space of the concatenated matrix $\left[U_1 \, U_2 \right]$ we have that the projection on $U$ must have a smaller or equal norm than the projection on $\left[U_1 \, U_2 \right]$
	\begin{align*}
		\left\| U^T e_i \right\|^2_2 &\leq \left\| \left[ U_1 \, U_2 \right]^T e_i \right\|^2_2 \\
		& = e_i^T \left[U_1 \, U_2 \right] \left[U_1 \, U_2 \right]^T e_i \\
		& = \left\| U_1^T e_i \right\|_2^2 + \left\| U_2^T e_i \right\|_2^2 \leq 2\frac{\mu r}{m},
	\end{align*}
	where in the third line we use the definition of incoherence. Since the rank of the matrix doubled, the inequality yields the desired result $\| U^T e_i\| \leq \sqrt{\frac{\mu 2r}{m}}$. The argument can be followed \emph{mutatis mutandis} for the upper bound on the right singular vectors $V$.
\end{proof}

\prooftitle{Lemma~\ref{lemma:LS_mu_closed}}{$\LS_{m,n}(r,s,\mu)$ is a closed set}{\pageref{lemma:LS_mu_closed}}
\begin{proof}\label{lemma:LS_mu_closed_proof}
	The first part (1) of the statement is proved as part of Lemma~\ref{lemma:notclosed-mucorrelation}. 

	To prove the second part (2) of the statement, let $X = L + S \in \LS_{m,n}(r,s,\mu)$ and denote $\rcc = \mu \frac{r \sqrt{s}}{\sqrt{mn}}$ for which we have $\rcc <1$ by $\mu < \sqrt{mn}/(r\sqrt{s})$. By conicity of $\LS_{m,n}(r,s,\mu)$ we can assume without loss of generality $\|X\|_F = 1$. The bound on the correlation in (1) states
	\begin{equation*}
		\rcc \geq \frac{\left| \inp{L}{S} \right| }{\|L\|_F \|S\|_F},
	\end{equation*}
	which combined with the rearranged terms of the identity $\|X\|_F^2 = 1 = \|L\|^2_F + \|S\|^2_F + 2\inp{L}{S}$ yields
	\begin{align}
		\gamma \geq \frac{\left| \inp{L}{S} \right| }{\|L\|_F \|S\|_F} &= \frac{1}{2} \left| \frac{1}{\|L\|_F \|S\|_F} - \frac{\|S\|_F}{\|L\|_F} - \frac{\|L\|_F}{\|S\|_F} \right|. \label{eq:ls-equiv2}
	\end{align}

	The proof follows by showing that the inequality in \eqref{eq:ls-equiv2} implies an upper bound on $\|L\|_F$ and $\|S\|_F$. For ease of notation, we denote $x := \|L\|_F$ and $y := \|S\|_F$, and multiply the inequality in \eqref{eq:ls-equiv2} by $\left\|L\right\|_F \left\|S\right\|_F$ 
	\begin{equation}
		2 \rcc x y \geq | 1 - x^2 - y^2 | . \label{eq:ls-equiv3}
	\end{equation}
	where we used that $\left\|L\right\|_F, \left\|S\right\|_F$ are strictly positive.

	The case of $1 - x^2 - y^2 \geq 0$ implies that $x \leq 1$ and $y \leq 1$, and thus concludes the proof. The other case of $1 - x^2 - y^2 \leq 0$ in \eqref{eq:ls-equiv3} is equivalent~to
	\begin{equation*}
		2 \rcc x y \geq x^2 + y^2 -1,
	\end{equation*}
	which has two roots $y = -cx \pm \sqrt{c^2 x^2 - x^2 + 1}$. Since $x,y$ denote the Frobenius norm of $L$ and $S$ respectively, we seek only the real roots, for which to exist we need $c^2 x^2 - x^2 + 1 \geq 0$, and because $c<1$, we can rearrange the terms as
	\begin{equation*}
		x \leq \frac{1}{\sqrt{1- \gamma^2}},
	\end{equation*}
	which is equivalent to
	\begin{align*}
		\left\|L\right\|_F &\leq \frac{1}{\sqrt{1-\gamma^2}} = \left(1 - \mu^2\frac{r^2 s}{mn} \right)^{-1/2},
	\end{align*}
	prooving the second statement (2) in Lemma~\ref{lemma:LS_mu_closed}. Applying the same arguments to $\| S \|_F$ yields the bound on the Frobenius norm of the sparse component.
	
	Finally, to prove the third part of the statement (3), consider a sequence $X_i = L_i + S_i \in \LS_{m,n}(r,s, \mu)$ that converges to a matrix $X\in\bR^{m\times n}$ as $i \rightarrow \infty$. Since, also $\| X_i\|_F \rightarrow \|X\|_F$, we have that for any $\ep > 0$, there exists $i_0 \in \mathbb{N}$ such that
	\begin{equation*}
		\forall i > i_0: \qquad \left\|X\right\|_F - \ep \leq \left\| X_i \right\|_F \leq \left\| X \right\|_F + \ep,
	\end{equation*} 
	which, combined with $\left\|L_i\right\|_F \leq \tau \left\| X_i \right\|_F$, implies that for all $i \geq i_0$ we have $\left\| L_i \right\|_F \leq \tau \left\| X \right\|_F + \tau\ep$ where $\tau:=(1-\mu\frac{r^2s}{mn})^{-1/2}$ by the second part of the statement (2). 
	
	Denote the closed set of rank-$r$ matrices whose Frobenius norm is bounded by $\gamma > 0$ as 
	\begin{equation*}
		\mathcal{L}_{m,n}\left(r, \gamma \right) = \left\{Y \in \bR^{m\times n}:\quad \rank{(Y)}\leq r, \quad \|Y\|_F \leq \gamma \right\},
	\end{equation*}
    which is also compact by being closed and bounded.

	We have that $L_i\in\mathcal{L}_{m,n}\left(r,\left\| X \right\|_F + \tau\ep\right)$ for all $i \geq i_0$. Since the set is compact and closed, we can assume, by passing to a subsequence, that $L_i \xrightarrow{i \rightarrow \infty} L \in \mathcal{L}_{m,n}\left(r, \tau \left\|X\right\|_F + \tau \ep \right)$ as $i \geq i_0$. 
	
	Additionally, since $\tau>0$ is fixed, the upper bound of the Frobenius norm of the low-rank component $\left\|L_i\right\|_F \leq \tau \left\|X_i\right\|_F$ must also hold in the limit $\left\| L \right\|_F \leq \tau \left\|X \right\|_F$.

	By the set of $s$-sparse matrices being closed, we have that the limit point
	\begin{equation*}
		S_i = X_i - L_i \rightarrow X - L,
	\end{equation*}
	is also an $s$-sparse matrix, thus 
	\begin{equation*}
		X = L + (X-L) \in \LS_{m,n}(r,s,\mu),
	\end{equation*}
	proving that $\LS_{m,n}(r,s, \mu)$ is closed.
\end{proof}

\begin{lemma}[\Rev{The rank-sparsity correlation bound}]
	\label{lemma:notclosed-mucorrelation}
	Let $L,S\in\bR^{m\times n}$ and $L = U \Sigma V^T$ be the singular value decomposition of $L$, then
	\begin{equation}
		\left| \inp{L}{S} \right| \leq \left\| \abs\left( U \right) \abs\left( V^T\right) \right\|_\infty \, \sigma_\mathrm{max}\left( L \right) \left\| S \right\|_1, \label{eq:rank-sparsity-correlation1}
   \end{equation}
   where $\abs(\cdot)$ denotes the entry-wise absolute value of a matrix, the matrix norms are vectorised entry-wise $\ell_p$-norms, and $\sigma_\mathrm{max}\left( L \right)$ is the largest singular value of $L$. As a consequence, if $L$ is a rank-$r$ matrix that is $\mu$-incoherent and $S$ is an $s$-sparse~matrix
	\begin{equation}
		\left| \inp{L}{S} \right| \leq \mu \frac{r \sqrt{s}}{\sqrt{mn}} \left\| L \right\|_F \left\| S \right\|_F. \label{eq:notclosed-mucorrelation}
	\end{equation}
\end{lemma}
\begin{proof}\label{proof:notclosed-mucorrelation}
	For $L, S\in\bR^{m\times n}$ and $L = U \Sigma V^T$ being the singular value decomposition of $L$, we have
	\begin{align}
		\left| \inp{L}{S} \right| &= \left| \sum_{(i,j)\in [m] \times [n]} S_{i,j} \, L_{i,j} \,\right| = \left| \sum_{(i,j)\in [m] \times [n]}  S_{i,j} \,  e_i^T U \Sigma V^T f_j \, \right| \label{eq:coherence-correlation0}\\
		&\leq \sum_{(i,j)\in [m] \times [n]} \left| S_{i,j} \right| \left| \left( U^T e_i \right)^T \Sigma \left( V^T \Rev{f_j} \right) \right| \label{eq:coherence-correlation1} \\
		& = \sum_{(i,j)\in [m] \times [n]} \left| S_{i,j} \right| \left| \sum_{k=1}^r \sigma_k\, \left( U^T e_i \right)_k \left( V^T \Rev{f_j} \right)_k \right| \label{eq:coherence-correlation2} \\
		& \leq \sum_{(i,j)\in [m] \times [n]} \left| S_{i,j} \right| \, \sum_{k=1}^r \sigma_k \, \abs \left( U^T e_i \right)_k \, \abs\left( V^T f_j \right)_k \label{eq:coherence-correlation3} \\
		& \leq \sigma_{\mathrm{max}}(L) \sum_{(i,j)\in [m] \times [n]} \left| S_{i,j} \right| \abs\left( U^T e_i \right)^T \abs\left( V^T \Rev{f_j} \right) \label{eq:coherence-correlation4} \\
		& \leq \sigma_{\mathrm{max}}(L) \, \|S\|_1 \left\| \abs\left( U \right) \abs\left( V^T\right) \right\|_\infty  \label{eq:coherence-correlation5}
	\end{align}
\Rev{where in the first line in \eqref{eq:coherence-correlation0} we denote $e_i\in\bR^m$, $f_i\in\bR^n$ to be the canonical basis vectors of $\bR^m$ and $\bR^n$}, the inequality in the second line \eqref{eq:coherence-correlation1} comes from the subadditivity of the absolute value, in the third line \eqref{eq:coherence-correlation2} we write out the inner product as a sum, in the fourth line \eqref{eq:coherence-correlation3} we use the subadditivity and multiplicativity of the abslolute value and denote $\abs(\cdot)$ as the entry-wise absolute value of a vector, the fifth line \eqref{eq:coherence-correlation4} comes from $\sigma_{\mathrm{max}}(L)$ being the largest singular value of $L$, and the final line in \eqref{eq:coherence-correlation5} comes from the entry-wise $\ell_\infty$-norm bounding the absolute value of all  entries of $\left(\abs\left( U \right) \abs\left( V^T\right)\right)$.

If the low-rank component $L$ is also $\mu$-incoherent, we further have 
\begin{align}
	\left\| \abs\left( U \right) \abs\left( V^T\right) \right\|_\infty &= \max_{(i,j)\in [m]\times [n]} \abs\left( U^T e_i \right)^T \abs\left( V^T f_j \right) \\
	&\leq \left\| U^T e_i \right\|_2 \left\| V^T f_j \right\|_2 \\
	&\leq \mu \frac{r}{\sqrt{mn}}, \label{eq:coherence-correlation6}
\end{align}
where the first upper bound comes from the Cauchy-Schwarz inequality on the entry-wise absolute values of $U^Te_i$ and $V^T f_j$, and the second upper bound comes from the definition of incoherence in \eqref{eq:coherence} Combining \eqref{eq:coherence-correlation6} with \eqref{eq:coherence-correlation5}, the fact that $\left\|S\right\|_1 \leq \sqrt{s} \left\| S \right\|_F$ for $s$-sparse matrices, and that $\sigma_{\mathrm{max}} \leq \|L\|_F$ yields the result in \eqref{eq:notclosed-mucorrelation}
\end{proof}

\prooftitle{Lemma~\ref{lemma:rip_single}}{RIC for a fixed $\LS$ subspace}{\pageref{lemma:rip_single}}\\
The proof uses similar to arguments as \cite[Lemma 5.1]{Baraniuk2008a} and \cite[Lemma 4.3]{Recht2010guaranteed} with the exception that here we consider two subsets, one for the low rank and another for the sparse component.
\begin{proof}\label{lemma:rip_single_proof}
	\Rev{By the linearity of $\cA(\cdot)$ and conicity of $\Sigma_{m,n}(V, W, T, \cohc)$ we can assume without loss of generality that $\|X\|_F = 1$}. By Lemma~\ref{lemma:LS_mu_closed} with $\| X \|_F = 1$ and $\cohc < \frac{\sqrt{mn}}{r \sqrt{s}}$, we can bound the Frobenius norm of the low-rank and the sparse component as $\| L \| _F \leq \tau $ and $\| S \| _F \leq \tau$, where $\tau := \left(1 - \mu^2 \frac{r^2 s}{mn}\right)^{-1/2}$. 
	
	There exist two finite $(\ric/8)$-coverings of the two matrix sets with bounded norms
	\begin{gather}
		 \left\{ L\in\bR^{m\times n}\,:\, \, \Col{L} \subseteq V,\, \Col{L^T} \subseteq W,\, \|L\|_F \leq \tau \right\} \label{eq:rip_single_setL} \\
		\left\{ S\in\bR^{m\times n}\,:\, \, S \subseteq T,\, \|S\|_F \leq \tau \right\}, \label{eq:rip_single_setS}
	\end{gather}
	that we denote $\Lambda^L, \Lambda^S$ and by \cite[Chapter~13]{Lorentz1996constructive} they are subsets of the two sets in \eqref{eq:rip_single_setL} and \eqref{eq:rip_single_setS}, and their covering numbers are upper bounded as
	\begin{equation}
		\left|\Lambda^L\right| \leq \left(\frac{24}{\ric} \tau \right)^{\dim V \cdot \dim W} \qquad \left|\Lambda^S\right| \leq \left(\frac{24}{\ric} \tau \right)^{\dim T}.
	\end{equation}
	Let $\Lambda := \left\{Q^{L} + Q^{S}:\, Q^L\in \Lambda^L,\, Q^S\in\Lambda^S \right\}$ be the set of sums of all possible pairs of the two coverings. The set $\Lambda$ is a $(\ric/4)$-covering of the set $\Sigma_{m,n}\left(V,W,T, \cohc\right)$ since for all $X \in \Sigma_{m,n} \left(V, W, T, \cohc\right)$ there exists a pair $Q \in\Lambda$ such that 
	\begin{align}
		\left\|X - Q\right\|_F &= \left\| L + S - \left(Q^L + Q^S\right)\right\|_F \\ 
		&\leq \left\| L - Q^{L}\right\|_F + \left\| S - Q^{S}\right\|_F \leq \frac{\ric}{8}+\frac{\ric}{8},
	\end{align}
	where in the first line we used the fact that $X$ can be expressed as $L + S$, and in the second line we applied the triangular inequality combined with the $Q^{L}, Q^{S}$ being $(\ric/8)$-coverings of the matrix sets for the low-rank component and the sparse component respectively.
	
	Applying the probability union bound on concentration of measure of $\cA$ as in \eqref{eq:conc_measure} with $\ep = \ric/2$ gives that
	\begin{equation}
		(\forall Q \in \Lambda): \left(1 -  \frac{\ric}{2}\right) \left\| Q \right\|_F \leq \left\|\cA(Q)\right\|_2 \leq \left(1 + \frac{\ric}{2} \right) \left\| Q \right\|_F, \label{eq:rip_single_eq1}
	\end{equation}
	holds with the probability at least 
	\begin{equation}
		1 - 2 \left(\frac{24}{\ric} \tau \right)^{\dim V \cdot \dim W}  \left(\frac{24}{\ric} \tau \right)^{\dim T} \exp{\left(-\frac{p}{2}\left( \frac{\ric^2}{8} - \frac{\ric^3}{24} \right) \right)}.
	\end{equation}
	
	By $\Sigma_{m,n}(V, W, T, \cohc)$ being a closed set by Lemma~\ref{lemma:LS_mu_closed}, the maximum
	\begin{equation}
		M = \max_{Y\in \Sigma_{m,n}(V, W, T, \cohc),\,\|Y\|_F = 1 } \| \cA(Y) \|_2,\label{eq:rip_single_max}
	\end{equation}
	is attained. Then there exists $Q\in\Lambda$ such that
	\begin{equation}
		\left\| \cA(X)\right\|_2 \leq \left\| \cA(X) \right\|_2 +  \left\| \cA(X-Q) \right\|_2 \leq 1 + \frac{\ric}{2} + M\frac{\ric}{4}, \label{eq:rip_single_eq2}
	\end{equation}
	where the first inequality comes from applying the triangle inequality to $X$ and $Q-X$ and in the second inequality we used \eqref{eq:rip_single_eq1} to upper bound $\|\cA(X)\|_2$ since $(X-Q) \in \Sigma_{m,n}(V,W,T, \cohc)$ by Lemma~\ref{lemma:ls_mu_additive} and the upper bound of $\| X- Q \|_F$ comes from $Q\in\Lambda$ combined with $\Lambda$ being a $(\ric/4)$-covering. Note that the inequality \eqref{eq:rip_single_eq2} holds for all $X\in\Sigma_{m,n}(V, W, T, \cohc)$ whose Frobenius norm \Rev{$\|X\|_F= 1$} and thus also for a matrix $\widehat{X}$ for which the maximum in \eqref{eq:rip_single_max} is attained. The inequality in \eqref{eq:rip_single_eq2} applied to the matrix that attains the maximum $\widehat{X}$ yields
	\begin{equation}
		M \leq 1 + \frac{\ric}{2} + M \frac{\ric}{4} \quad \implies \quad M \leq 1+\ric. \label{eq:rip_single_Mbound}
	\end{equation}
	The lower bound follows from the reverse triangle inequality
	\begin{equation}
		\| \cA(X)\|_2 \geq \| \cA(Q) \|_2 - \| \cA(X-Q) \|_2 \geq \left(1-\frac{\ric}{2}\right) - (1+\ric) \frac{\ric}{4} \geq 1-\ric \label{eq:rip_single_lower}
	\end{equation}
	where the second inequality comes from $\left\| \cA(X-Q) \right\|_2 \leq M \left\| X-Q\right\|_F \leq \left(1+\ric \right)\frac{\ric}{4}$ by \eqref{eq:rip_single_max} combined with $Q$ being an element of a $(\ric/4)$-covering.
	
	Combining \eqref{eq:rip_single_eq2} with the bound on $M$ in \eqref{eq:rip_single_Mbound} gives the upper bound and \eqref{eq:rip_single_lower} gives the lower bound on $\left\|\cA(X)\right\|_2$ completing the proof.
	\end{proof}


\begin{lemma}[{{$\varepsilon$-covering of the Grassmannian \cite[Theorem 8]{Szarek1998metric}}}] \label{lemma:szarek98}
Let $\left(\G(D, d), \rho(\cdot, \cdot) \right)$ be a metric space on a Grassmannian manifold $\G(D,d)$ with the metric $\rho$ as defined in \eqref{eq:grass_distance}. Then there exists $\varepsilon$-covering $\G(D,d)$ with $\Lambda = \left\{U_i\right\}_{i=1}^{N}\subset \G(D,d)$ such that
\begin{equation}
\forall U\in\G(D,d): \quad \min_{\widehat{U} \in \Lambda} \rho(U, \widehat{U}) \leq \varepsilon,
\end{equation}
and $N \leq \left( \frac{C_0}{\varepsilon}\right)^{d(D-d)}$ with $C_0$ independent of $\varepsilon$, bounded by $C_0 \leq 2\pi$.
\end{lemma}
The above bound on the covering number of the Grassmannian is used in the following lemma to bound the covering number of the set $\LS_{m,n}(r,s,\tau)$.

\prooftitle{Lemma~\ref{lemma:ls_rs_cover}}{Covering number of $\LS_{m,n}(r,s)$}{\pageref{lemma:ls_rs_cover}}
\begin{proof}\label{lemma:ls_rs_cover_proof}
By Lemma \ref{lemma:szarek98} there exist two finite $(\varepsilon/2)$-coverings $\Lambda_1:= \left\{ V_i \right\}_{i = 1}^{|\Lambda_1|} \subseteq \G(m,r)$ and $\Lambda_2:= \left\{ W_i \right\}_{i = 1}^{|\Lambda_2|} \subseteq \G(n,r)$, with their covering numbers upper bounded as
\begin{align}
	|\Lambda_1| \leq \left( \frac{4\pi}{\varepsilon} \right)^{r ( m- r)} \qquad & |\Lambda_2| \leq \left( \frac{4\pi}{\varepsilon} \right)^{r ( n- r)},\label{eq:ls_rs_cover1}
\end{align}
as given in \cite[(4.18)]{Recht2010guaranteed} that uses \cite[Theorem 8]{Szarek1998metric}.
By $\Lambda_1, \Lambda_2$ being $(\varepsilon/2)$-coverings
\begin{align}
	\forall V\in \G(m,r): \quad &\exists V_i\in \Lambda_1, \quad \rho(V, V_1) \leq \varepsilon/2, \\
	\forall W\in \G(n,r): \quad &\exists W_i\in \Lambda_2, \quad \rho(W, W_1) \leq \varepsilon/2.
\end{align}
Let $\Lambda_3 = \V(mn,s) $ where $\V(mn,s)$ is the set of all possible support sets of an $m\times n$ matrix that has $s$ elements. Thus the cardinality of $\Lambda_3$ is ${mn \choose s}$. 

Construct $\Lambda = (\Lambda_1 \times \Lambda_2 \times \Lambda_3)$ where $\times$ denotes the Cartesian product. Choose any $V\in \G(m,r), W\in \G(n,r)$ and $T\in \V(mn,s)$ for which we now show there exists $\left(\widehat{V}, \widehat{W},\widehat{T}\right)\in\Lambda$ such that $\rho\left(\left(V,W\right), \left(\widehat{V},\widehat{W}\right)\right)\leq \varepsilon$ and $T = \hat{T}$, thus showing that the set $\Lambda$ is an $\varepsilon$-covering of $\LS_{m,n}(r,s,\tau)$.

Satisfying $T = \widehat{T}$ comes from $\Lambda_3 = \V(mn,s)$ containing all support sets with at most $s$ entries. The projection operator onto the pair $(V,W)$ can be written as $P_{(V, W)} = P_V \otimes P_W$, so for the two pairs of subspaces $(V, W)$ and $(\widehat{V}, \widehat{W})$ we have the following
\begin{align}
\rho\left(\left(V,W\right), \left(\widehat{V},\widehat{W}\right)\right) &= \| P_{(V,W)} - P_{(\widehat{V},\widehat{W})}\| \\
 	& = \| P_{V} \otimes P_{W} - P_{\widehat{V}} \otimes P_{\widehat{W}}\| \\
	& = \| \left(P_{V} - P_{\widehat{V}} \right) \otimes P_{W} + P_{\widehat{V}}\left(P_{W} - P_{\widehat{W}} \right)  \| \\
	& \leq \| P_{V} - P_{\widehat{V}} \| \| P_{W} \| + \| P_{\widehat{V}} \| \| P_{W} - P_{\widehat{W}}\| \\
	& = \rho\left(V, \widehat{V}\right) + \rho\left(W, \widehat{W}\right).
\end{align} 
By $\Lambda_1$ and $\Lambda_2$ being $(\varepsilon/2)$-coverings, we have that for any $V,W$ exist $\widehat{V}\in\Lambda_1$  and $\widehat{W}\in\Lambda_2$, such that $\rho\left(\left(V,W\right), \left(\widehat{V},\widehat{W}\right)\right)\leq\rho\left(V, \widehat{V}\right) + \rho\left(W, \widehat{W}\right)\leq \varepsilon$. Using the bounds on the cardinality of $\Lambda_1, \Lambda_2$ in \eqref{eq:ls_rs_cover1} combined with $|\Lambda_3| = {mn \choose s}$ yields that the cardinality of $\Lambda$ is bounded above by
\begin{equation}
	\cR(\varepsilon)  = |\Lambda_1|\,|\Lambda_2|\,|\Lambda_3| \leq {mn\choose s} \left( \frac{4\pi}{\varepsilon} \right)^{r \left(m + n -2r\right)}.
\end{equation}
\end{proof}

\prooftitle{Lemma~\ref{lemma:var_delta}}{Variation of $\ric$ in RIC in respect to a perturbation of $(V,W)$}{\pageref{lemma:var_delta}}
\begin{proof}\label{lemma:var_delta_proof}
	Recall the notation used in Lemma~\ref{lemma:var_delta} that there are sets $\Sigma_1 := \Sigma_{m,n}\left(V_1, W_1, T, \mu\right)$ and $\Sigma_2 :=\Sigma_{m,n}\left(V_2, W_2, T, \mu\right)$ which have a shared support $T$ of the sparse component.
	
	Let $Y\in \Sigma_2$, so we can write $Y = L + S$ such that $\supp(S) = T, \Col{L} \subseteq V_2, \Col{L^T} \subseteq W_2$ and $\|L\|_F \leq \Rev{\tau \|Y\|_F}$ \Rev{for $\tau:=(1-\mu^2\frac{r^2 s}{mn})^{-1/2}$ by Lemma~\ref{lemma:LS_mu_closed}}. By linearity of $\cA$ assume without loss of generality \Rev{$\|Y\|_F = 1$ and therefore $\|L\|_F \leq \tau$}.
	Denote $U_1 = (V_1, W_1)$ and $U_2 = (V_2, W_2)$ and let $P_{U_i}$ be an orthogonal projection onto the space of matrices whose column and row space is defined by $V_i,W_i$ such that left and right singular vectors of $P_{U_i}Y$ lie in $V_i$ respectively $W_i$. Then
	\begin{align}
		\|\cA(Y)\| &= \left\| \cA(L+S) \right\| = \left\| \cA\left( P_{U_1} L + S - \left(P_{U_1} L - P_{U_2} L\right) \right) \right\| \label{eq:var_delta1}\\
		&\leq \left\| \cA\left( P_{U_1} L + S \right) \right\| + \left\|\cA \left( \left[ P_{U_1} - P_{U_2}\right] L \right) \right\|  \label{eq:var_delta2}\\
		&\leq (1+\ric) \left\| P_{U_1} L + S \right\| + \| \cA \| \rho\left(U_1, U_2\right) \left\|L\right\| \label{eq:var_delta3}\\
		&= (1+\ric) \left\| P_{U_2} L + S + \left[P_{U_1} - P_{U_2}\right] L \right\| + \| \cA \| \rho\left(U_1, U_2\right) \left\|L\right\|  \label{eq:var_delta4}\\
		&\leq (1+\ric) \left( \|Y \|_F + \rho(U_1, U_2)\|L\|\right) + \| \cA \| \rho\left(U_1, U_2\right) \left\|L\right\|  \label{eq:var_delta5}\\
		&\leq \|Y\|_F \left( 1+\ric + \tau \rho(U_1, U_2) \left( 1+ \ric + \| \cA\| \right) \right),  \label{eq:var_delta6}
	\end{align}
	where in the first line \eqref{eq:var_delta1} we use the fact that $P_{U_2} L = L$, the second line \eqref{eq:var_delta2} follows by the triangle inequality and linearity of $\cA$, and in the third inequality we bound the effect of $\cA$ on $(P_{U_1} L+S)$ using the RICs of $\cA$ combined with the definition of $\rho$ in \eqref{eq:grass_distance}. We proceed in \eqref{eq:var_delta4} and \eqref{eq:var_delta5} by projecting $L$ to space $U_2$ and again bounding the effect of $\cA$ on $(P_{U_2} L+S)$. Finally, in \eqref{eq:var_delta6} we use $\| L\|_F\leq \tau$.
	We obtain a similar lower bound using the reverse triangular inequality
	\begin{align}
		\|\cA(Y)\| &= \left\| \cA\left( P_{U_1} L + S - \left(P_{U_1} L - P_{U_2} L\right) \right) \right\| \\
		 &\geq \left\| \cA \left( P_{U_1} L + S \right) \right\| - \left\|\cA \left( \left[ P_{U_1} - P_{U_2}\right] L \right) \right\|  \\
		 &\geq \left(1-\ric \right)\left\| P_{U_1} L + S \right\| - \|\cA\| \rho(U_1, U_2) \| L \|_F \\
		 &= \left(1-\ric \right)\left\| P_{U_2} L + S - \left[P_{U_2} - P_{U_1}\right] L \right\| - \|\cA\| \rho(U_1, U_2) \| L \|_F \\
		 & \geq \left(1-\ric \right)\left( \|Y \|_F - \rho(U_1, U_2) \|L \|_F \right) - \|\cA\| \rho(U_1, U_2) \| L \|_F \\
		 & \geq \|Y \|_F \left( 1- \ric - \tau \rho(U_1, U_2) (1-\ric  +\| \cA \| ) \right). \label{eq:var_delta6b}
	\end{align}
	Combining \eqref{eq:var_delta6} and \eqref{eq:var_delta6b}  yields
	\begin{equation}
		\forall Y \in \Sigma_2: \quad (1-\ric')\| Y\|_F \leq \| \cA (Y)\| \leq (1+\ric') \|Y \|_F,
	\end{equation}
	with $\ric' = \ric + \tau \rho(U_1, U_2) \left(1+ \ric + \|\cA\| \right)$.
\end{proof}

In the proof of Theorem~\ref{thm:convex_recovery} we make use of the following Lemma~\ref{lemma:R0L} and Corollary~\ref{cor:convex_nucsum} from \cite{Recht2010guaranteed} which we restate here for completeness with the small addition of the incoherence property in (2).
\begin{lemma}[{{\cite[Lemma 3.4]{Recht2010guaranteed}}}]\label{lemma:R0L}
Let $A\in\LS_{m,n}(r,0,\mu)$ and $B\in\bR^{m\times n}$. Then there exist matrices $B_1$ and $B_2$ such that
\begin{enumerate}
	\item{$B = B_1 + B_2$,}
	\item{$B_1 \in \LS_{m,n}(2r, 0, \mu)$,}
	\item{$AB_2^T = 0$ and $A^T B_2 =0$,}
	\item{$\left<B_1,\, B_2\right> = 0$.}
\end{enumerate}
\end{lemma}
\begin{proof}
	Consider a full singular value decomposition of $A$,
	\begin{equation}
		A = U \left[ \begin{array}{c|c}
			\Sigma & 0 \\
			\hline
			0 & 0
			\end{array}\right] V^T,
	\end{equation}
	and let $\hat{B} := U^T B V$. Partition $\hat{B}$ as
	\begin{equation}
		\hat{B} = \left[ \begin{array}{c|c}
			\hat{B}_{11} & \hat{B}_{12} \\
			\hline
			\hat{B}_{21} & \hat{B}_{22}
			\end{array}\right].
	\end{equation}
	Defining now 
	\begin{equation}
		B_1 := U \left[ \begin{array}{c|c}
			\hat{B}_{11} & \hat{B}_{12} \\
			\hline
			\hat{B}_{21} & 0
			\end{array}\right] V^T, \qquad 
			B_2 := U \left[ \begin{array}{c|c}
				0 & 0 \\
				\hline
				0 & \hat{B}_{22}
				\end{array}\right] V^T, 
	\end{equation}
	it can be verified that $B_1$ and $B_2$ satisfy the conditions of the lemma.
\end{proof}

\begin{corollary}[{{\cite[Lemma 2.3]{Recht2010guaranteed}}}] \label{cor:convex_nucsum}
Let $A$ and $B$ be matrices of the same dimensions. If $AB^T=0$ and $A^TB =0$, then $\|A+B\|_* = \|A\|_* + \|B\|_*$.
\end{corollary}

\begin{lemma}[Decomposing $R^S = R^S_0 + R^S_c$] \label{lemma:R0S}
Let $\supp S_0 = \Omega_0$ and construct a matrix $R_0^S$ that has the entries of $R^S$ at indices $\Omega_0$
\begin{equation}
	(R_0^S)_{i,j} = \begin{cases}
						(R^S)_{i,j} &\text{if}\quad (i,j) \in \Omega_0,\\
						0 &\text{if}\quad (i,j) \notin \Omega_0,
				\end{cases}\label{eq:decomposing1}
\end{equation}
and a matrix $R^S_c  = R^S - R_0^S$ that has the entries of $R^S$ at the indices of the complement of $\Omega_0$. Then
\begin{enumerate}
	\item{$\| R^S_0\|_0 \leq \|S_0\|_0$ = s \quad (by $|\Omega_0| = s$),}
	\item{$\|S_0 +R_c^S\|_1 = \|S_0\|_1 + \|R_c^S\|_1$ \quad (by $\supp(R^S_0)\cap \supp(R^S_c) = \emptyset$),}
	\item{$\inp{R_0^S}{R_c^S} = 0$ \quad (by $\supp(R^S_0)\cap \supp(R^S_c) = \emptyset$).}
\end{enumerate}
\end{lemma}
\begin{proof}
It can be easily verified that $R_0^S$ and $R^S_c$ constructed as in \eqref{eq:decomposing1} satisfy the conditions \mbox{(1)-(3)}.
\end{proof}

\begin{lemma}[\Rev{Decomposing $R^{L}_c$ into a sequence of incoherent low-rank matrices}]\label{lemma:decomposing_RL}
	Let $R^L_c\in\bR^{m\times n}$ be an arbitrary matrix and $M_r\in\mathbb{N}$ be a fixed rank of the decomposition. There exists a decomposition $R^L_c = \sum_{i = 1}^{mn} R_i^L$ such that
	\begin{gather}
		 R_i^L \in \LS_{m,n}(M_r, 0, 1) \label{eq:appendix-convex_RcL_rank}\\
		R_i^L \left(R_j^L\right)^T = 0_{m\times m} \quad \text{and}\quad \left(R_i^L\right)^T R_j^L = 0_{n\times n},\quad \forall i\neq j \label{eq:appendix-convex_RcL_ortho}\\
		\left\|R_{i+1}^L\right\|_F^2 \leq \frac{1}{M_r}\left\|R_i^L\right\|_*^2 \label{eq:appendix-convex_RcL_decay}.
	\end{gather}
\end{lemma}
\begin{proof}
	Let $Y = [y_1, y_2, \ldots, y_m]\in\bR^{m\times m}$ and $Z = [z_1, z_2, \ldots, z_n]\in\bR^{n\times n}$ be two bases whose vectors are maximally incoherent with the canonical basis
	\begin{gather}
		\forall i,j\in [m] \qquad \| y_j^T e_i\|_2 = \frac{1}{\sqrt{m}} \\ 
		\forall i,j\in [n] \qquad \| z_j^T e_i\|_2 = \frac{1}{\sqrt{n}},
	\end{gather}
	which can be constructed by taking $m$ columns and the same rows of a Hadamard matrix and rescaling it such that it forms an orthonormal basis.

	Denote $E = \left\{ y_i z_j^T \right\}_{i,j = 1}^{m,n} \subset \bR^{m\times n}$. Since $E$ is a basis, there are coefficients $c_1, c_2, \ldots, c_{mn}\in\bR$ such that
	\begin{equation}
		R_c^L = \sum_{k = 1}^{mn} c_k  \, y_k \, z_k^T.
	\end{equation}
	Since the columns of $Y$ and $Z$ can be arbitrarily permutated, we can assume without loss of generality that $|c_{k}| \geq | c_{k+1}|$ for all $i$. We split the indices of $\{1, \ldots, mn\}$ into sets of size $M_r$ as
   \begin{equation}
		I_i := \left\{ (i-1) M_r+ 1, \ldots, i M_r \right\}.
   \end{equation}
   Constructing $R_i^L:= \sum_{k\in I_i} c_k  \, y_k \, z_k^T$ results into the decomposition with desirable properties. The first property \eqref{eq:appendix-convex_RcL_rank} follows from the subadditivity of the incoherence in Lemma~\ref{lemma:ls_mu_additive}, the second property in \eqref{eq:appendix-convex_RcL_ortho} follows from $E$ being an orthogonal basis, and finally, the last property in \eqref{eq:appendix-convex_RcL_decay} comes from the $c_k$ being the singular values of each constructed $R_i^L$.
\end{proof}

\begin{lemma}[Upper bound on $\inp{\cA(\cdot)}{\cA(\cdot)}$] \label{lemma:Ainp1}
	For an operator $\cA(\cdot)$ whose RICs are upper bounded by $\rict_2:=\rict_{2r,2s,\mu}$ and two incoherent low-rank plus sparse matrices $X_1 = L_1 + S_1\in \LS_{m,n}(r,s,\mu)$, $X_2 = L_2 + S_2\in \LS_{m,n}(r,s,\mu)$ that have orthogonal components $\inp{L_1}{L_2}= 0$, $\inp{S_1}{S_2} = 0$ and have bounded the rank-sparsity coefficient $\rcc_2:= \rcc_{2r,2s,\mu}<1$, we have that
	\begin{equation}
		\Big| \inp{\cA(X_1)}{\cA(X_2)} \Big| \leq \left( \rict_2 + \frac{2\rcc_2}{1-\rcc_2^2} \right) \| X_1 \|_F \, \| X_2 \|_F, \label{eq:inpbound_final}
	\end{equation}
	where $\rcc_2 = \mu\frac{2r\sqrt{2s}}{\sqrt{mn}}$ is the rank-sparsity correlation coefficient as defined in Lemma~\ref{eq:notclosed-mucorrelation} on page\pageref{eq:notclosed-mucorrelation}.
	\end{lemma}
	\begin{proof}
	By $\cA(\cdot)$ being a linear transform, bilinearity of the inner-product, and conicity of $\LS_{m,n}(r,s,\mu)$, we can assume without loss of generality that $\left\|X_1\right\|_F = 1$ and $\left\|X_2\right\|_F=1$. The parallelogram law applied to $\left\| \cA(X_1)\right\|_2$ and $\left\| \cA(X_2)\right\|_2$ yields
	\begin{equation}
		2\left(\left\| \cA(X_1)\right\|_2^2 + \left\| \cA(X_2)\right\|_2^2\right) =  \left\| \cA(X_1) + \cA(X_2)\right\|_2^2 + \left\| \cA(X_1) - \cA(X_2)\right\|_2^2. \label{eq:inpbound1}
	\end{equation}
	Subtracting $2\left\| \cA(X_1) - \cA(X_2)\right\|_2^2$ from both sides of \eqref{eq:inpbound1}
	\begin{equation}
		4\inp{\cA(X_1)}{\cA(X_2)} =  \left\| \cA(X_1) + \cA(X_2)\right\|_2^2 - \left\| \cA(X_1) - \cA(X_2)\right\|_2^2. \label{eq:inpbound2}
	\end{equation}
	We can expand the equality in \eqref{eq:inpbound2} to bound its right-hand side using the RICs as
	\begin{align}
		\left| \inp{\cA(X_1)}{\cA(X_2)} \right| &= \frac{1}{4} \left| \| \cA(X_1 + X_2) \|^2_F - \| \cA( X_1 - X_2 )\|_F^2\right| \label{eq:inpbound3}\\
				&\leq \frac{1}{4} \left| (1+\rict_2 ) \left\| X_1 + X_2 \right\|^2_F -  (1-\rict_2) \left\| X_1 - X_2 \right\|^2_F \right| \label{eq:inpbound4} \\
				& \leq \frac{1}{4} \, \bigg| (1+\rict_2 ) \left( \left\| X_1 \right\|^2_F + 2\inp{X_1}{X_2} + \left\| X_2 \right\|^2_F \right) \nonumber \\
				& \qquad\quad -  (1-\rict_2) \left( \left\| X_1 \right\|^2_F - 2\inp{X_1}{X_2} + \left\| X_2 \right\|^2_F \right) \bigg|  \label{eq:inpbound5}\\
				& = \left| \frac{\rict_2}{2} \left( \left\| X_1 \right\|_F^2 + \left\| X_2 \right\|_F^2\right) + \inp{X_1}{X_2} \right| = \Big| \rict_2 + \inp{X_1}{X_2} \Big| \label{eq:inpbound6}
	\end{align}
	where the inequality in the second line in \eqref{eq:inpbound4} comes from the RICs of $\cA(\cdot)$ and by $X_1+X_2$ and $X_1-X_2$ being in the set $\LS_{m,n}(2r,2s,\mu)$ combined with \Rev{Lemma~\ref{lemma:ls_mu_additive}}, the equality in the third line in \eqref{eq:inpbound5} is the result of expanding the inner products, and finally, the last equality in \eqref{eq:inpbound5} comes from elementary operations and using the fact that $\left\|X_1\right\| = 1$ and $\left\|X_2\right\| = 1$.
	
	Moreover, by $X_1$ and $X_2$ being component-wise orthogonal $\inp{L_1}{L_2} = 0$ and $\inp{S_1}{S_2} = 0$, we can upper-bound the magnitude of the correlation between $X_1$ and $X_2$ as
	\begin{align}
		\left| \inp{X_1}{X_2} \right| &= \left| \inp{L_1}{L_2} + \inp{L_1}{S_2} + \inp{L_2}{S_1} + \inp{S_1}{S_2} \right| \label{eq:inpbound7} \\
		& = \left|  \inp{L_1}{S_2} + \inp{L_2}{S_1} \right| \label{eq:inpbound8} \\
		& \leq \rcc_2 \Big( \left\| L_1 \right\|_F \left\| S_2 \right\|_F + \left\| L_2 \right\|_F \left\|S_1 \right\|_F \Big) \label{eq:inpbound9} \\
		& \leq \frac{2 \rcc_2}{1-\rcc_2^2}, \label{eq:inpbound10}
	\end{align}
	where in the first equality in \eqref{eq:inpbound7} we expanded the inner-product, the second equality in \eqref{eq:inpbound8} is the consequence of the components being orthogonal, the inequality in the third line in \eqref{eq:inpbound9} is the consequence of Lemma~\ref{lemma:notclosed-mucorrelation}, and the last inequality in \eqref{eq:inpbound10} comes from the upper-bound of the norms $\left\| L_1 \right\|_F, \left\| L_2 \right\|_F, \left\| S_1 \right\|_F, \left\| S_2 \right\|_F$ from Lemma~\ref{lemma:LS_mu_closed} and by $\left\|X_1\right\|_F = 1$ and $\left\|X_2\right\|_F = 1$. 
	
	We can now further upper bound \eqref{eq:inpbound6} using the bound in \eqref{eq:inpbound6} combined with the triangle on the absolute value
	\begin{align}
		\Big| \inner{\cA(X_1)}{\cA(X_2)} \Big| &\leq \rict_2 + \frac{2\rcc_2}{1-\rcc_2^2},
	\end{align}
	when $\left\|X_1\right\|_F = 1$ and $\left\|X_2\right\|_F = 1$ which translates into the bound in \eqref{eq:inpbound_final} in the general case
	\begin{align}
		\left| \inner{\cA\left(\frac{X_1}{\left\|X_1\right\|_F}\right)}{\cA\left(\frac{X_2}{\left\|X_2\right\|_F}\right)} \right|& \left\|X_1\right\|_F \left\|X_2\right\|_F \nonumber\\
				\leq \left( \rict_2 + \frac{2\rcc_2}{1-\rcc_2^2} \right) &  \left\|X_1\right\|_F \left\|X_2\right\|_F,
	\end{align}
	by linearity of $\cA(\cdot)$ and the inner product.
	
	Note that the bound can be lowered for specific matrices $X_1, X_2$ such that the matrices of their sums $X_1+X_2$ and $X_1-X_2$ are in $\LS_{m,n}(r,s,\cohc)$ sets with smaller ranks or sparsities.
\end{proof} 

\begin{lemma}\label{lemma:Ainp23}
	Let $X^j, X^{j+1},X_0$ be any matrices in the set $\LS_{m,n}(r,s,\mu)$ with \Rev{$\mu < \sqrt{mn} \big/ \left(3r \sqrt{3s}\right)$}, $\alpha_j\geq0$, and $\cA(\cdot)$ be an operator whose RICs are sufficiently upper bounded, then the following two inequalities hold
	\begin{align}
		\inp{X^j - X_0}{X^{j+1} - X_0}&- \step_j \inp{\cA(X^j -X_0)}{\cA(X^{j+1} - X_0)} \nonumber\\
			&\leq \|I - \step_j A^T_Q A_Q \|_2\|X^j - X_0\|_F \|X^{j+1} - X_0\|_F,\label{eq:Ainp2_eq1}
	\end{align}
	and 
	\begin{equation}
		\|X^j - X_0 - \step_j \cA^*\left( \cA\left(X^j - X_0 \right) \right) \|_F \leq \|I - \step_j A^T_Q A_Q \|_2 \|X^j - X_0\|_F, \label{eq:Ainp3_eq1}
	\end{equation}
	where the spectrum of the matrix $\left(I - \step_j A^T_Q A_Q\right)\in\bR^{mn \times mn}$ is bounded as
	\begin{equation}
		1-\step_j \left( 1 + \Delta_{3r, 3s, \mu}\right) \leq \lambda\left(I - \Rev{\step_j} A_Q^T A_Q\right) \leq 1 \Rev{-} \step_j \left( 1 - \Delta_{3r, 3s, \mu}\right), \label{eq:Ainp23_ric}
	\end{equation}
	which gives an upper bound on the norm $ \|I - \step_j A^T_Q A_Q \|_2 \leq \left| 1-\step_j \left( 1 + \Delta_{3r, 3s, \mu}\right)\right| $ as the lower bound in \eqref{eq:Ainp23_ric} is larger then the upper bound. 
\end{lemma}
\begin{proof}
	We vectorize the matrices on the left hand side of \eqref{eq:Ainp2_eq1} using a mapping $\mathrm{vec}(\cdot):\bR^{m\times n} \rightarrow \bR^{mn}$ that stacks columns of a given matrix into a vector and a mapping $\mathrm{mat}(\cdot)$ from the space of linear operators $\cA:\bR^{m\times n}\rightarrow \bR^p$ to the space of matrices of size $p\times mn$
	\begin{equation}
		\begin{gathered}
			x_0 = \mvec{X_0},\,x^j = \mvec{X^j},\, x^{j+1} = \mvec{X^{j+1}}\in\bR^{mn}\\
			A = \mmat{\cA} = \begin{bmatrix}
						\horzbar & \mvec{A_1}^T& \horzbar \\
						 & \vdots & \\
						\horzbar & \mvec{A_p}^T & \horzbar \\
					\end{bmatrix}\in\bR^{p\times mn}.
		\end{gathered}
	\end{equation}
	Let $X_0 = U^0 \Sigma^0 V^0 + S^0,\, X^j = U^j \Sigma^j V^j + S^j, \, X^{j+1} = U^{j+1} \Sigma^{j+1} V^{j+1} + S^{j+1}$ be the singular value decompositions where the matrices of the left singular vectors are $U^{j}\in\bR^{m\times r}$ and their sparse components are supported at indices $\Omega^j = \supp\left(S^j\right)$. Consider the union of the index sets $\Omega := \left\{\Omega^0, \Omega^j, \Omega^{j+1} \right\}$ and construct the following frame
	\begin{equation}
	Q = \left[ I_n \otimes U \quad  E \right] =
		\begin{bmatrix}
			U & 0_{n, 3r} & \ldots & 0_{n, 3r} 	& \vline & \vline& & \vline\\
			0_{n, 3r} & U & \ldots &  0_{n, 3r}	& \vline & \vline& & \vline\\
			\vdots & & \ddots & 				& \vline & e_{\Omega_1} & \ldots & e_{\Omega_{3s}}\\
			0_{n, 3r}  & \ldots & \ldots & U 	& \vline & \vline& & \vline \\
		\end{bmatrix}\in\bR^{mn\times 3(nr+s)},\label{eq:Ainp2_qproj}
	\end{equation}
	where $U\in\bR^{m\times 3r}$ is formed by concatenating $U^0, U^j, U^{j+1}$ and $e_{\Omega_i}$ is a vector corresponding to a vectorized matrix with a single entry $1$ at the index $\Omega_i$.
	Note that $P_Q = Q\left(Q^T Q\right)^{-1}Q^T$ is an orthogonal projection matrix on the low-rank plus sparse subspace defined by the matrix $U$ and the index set $\Omega$. Note that by $Q$ being formed by the low-rank plus sparse bases of $X_0, X^j, X^{j+1}$ we have that the projection does not change the vectorized matrices
	\begin{equation}
		P_Q x_0 = x_0, \quad P_Q x^j = x^j, \quad P_Q x^{j+1} = x^{j+1}.	\label{eq:Ainp2_eq2}
	\end{equation}
	To establish the bound in \eqref{eq:Ainp2_eq1} we write the left hand side in its vectorized form
	\begin{equation}
		\left(x^j-x_0 \right)^T\left(x^{j+1}-x_0\right) - \step_j \left(A (x^j - x_0)\right)^T\left(A(x^{j+1} -x_0)\right), \label{eq:Ainp2_eq3}
	\end{equation}
	and replacing $A$ with $A_Q = A P_Q$ in \eqref{eq:Ainp2_eq3} using the identities in \eqref{eq:Ainp2_eq2} simplifies the term as
	\begin{align}
		\left(x^j- x_0 \right)&^T\left(x^{j+1}-x_0\right) - \step_j \left(A_Q (x^j - x_0)\right)^T\left(A_Q(x^{j+1} -x_0)\right) \\
		& = \left(x^j - x_0\right)^T\left( (x^{j+1} - x_0) - \step_j A_Q^* A_Q (x^{j+1} - x_0) \right) \\
		& = \left(x^j - x_0\right)^T\left( (I- \step_j A_Q^* A_Q) (x^{j+1} - x_0) \right) \\
		& \leq \| I- \step_j A_Q^* A_Q \|_2 \, \| x^j - x_0 \|_2 \, \| x^{j+1} - x_0 \|_2 \\
		&= \| I- \step_j A_Q^* A_Q\|_2 \,  \| X^j - X_0 \|_F \, \| X^{j+1} - X_0 \|_F,
	\end{align}
	where $\| I- \step_j A_Q^* A_Q \|_2$ is the $\ell_2$ operator norm of an $mn \times mn$ matrix.
	
	Similarly we now establish the bound in \eqref{eq:Ainp3_eq1} 
	\begin{align}
		\left\|X^j - X_0 - \step_j  \cA^*\left( \cA\left(X^j - X_0 \right) \right) \right\|_F &= \left\| x^j -x_0 + \step_j A^T A\left(x_0 - x^j\right)\right\|_2 \\
		&= \left\| \left(I- \step_j A^T A\right) \left(x^j - x_0\right)\right\|_2 \\
		& \leq  \left\| I- \step_j A_Q^* A_Q\right\|_2 \,  \left\| X^j - X_0 \right\|_F,\label{eq:Ainp3_eq2}
	\end{align}
	where we just vectorized the matrices and the linear operator $\cA(\cdot)$ and upper bounded the expression using $\ell_2$-operator norm $\| I- \step_j A_Q^* A_Q \|_2$. Matrix $A_Q$ acts on a subspace of $\LS_{m,n}(3r,3s,\mu)$ and is self-adjoint, as such its eigenvalues can be bounded using the RICs as done by \cite{Tanner2013normalized} and by \cite{Blanchard2015cgiht}
	\begin{equation}
		1-\step_j \left( 1 + \Delta_{3r, 3s, 2\mu}\right) \leq \lambda\left(I - \Rev{\step_j} A_Q^* A_Q\right) \leq 1 \Rev{-} \step_j \left( 1 - \Delta_{3r, 3s, 2\mu}\right). \label{eq:Ainp3_eq3}
	\end{equation}
\end{proof}

\end{document}